\documentclass[10.5pt,a4paper]{article}
\usepackage{bbm}
\usepackage{amsmath}
\usepackage{amsthm}
\usepackage[english]{babel}
\usepackage[latin1]{inputenc}
\usepackage{amsmath}
\usepackage{graphics}
\usepackage{latexsym}
\def \a {\alpha}
\usepackage{amsfonts}
\numberwithin{equation}{section}

\def \esp {[0,T]\times \R}
\def \ess {\mbox{esssup}}
\def \I {{\cal I}}
\def \z {\zeta}
\def \eps {\epsilon}
\def \us {\Upsilon}
\def \d {\delta}
\def \wr {\mbox{w.r.t.}}
\def \ms {\medskip}
\def \bs {\bigskip}
\def \R{\mathbb{R}}
\def \cF {{\cal F}}
\def \cC {{\cal C}}
\def \th {\theta}
\def \qq {\qquad}

\def \dt {\Delta}

\def \bal{\begin{array}{l}}
\def \nf {\infty}
\def \ea{\end{array}}
\def \q {\quad}

\def \beq {\begin{eqnarray}}
\def \eeq {\end{eqnarray}}
\def \espyz {\cS^2\times\cH^2(\ell^2)}
\def \nd {\noindent}
\def \xt {X^{t,x}}
\def \cP {\mathbf{P}}
\def \E {\mathbf{E}}
\def \cS {{\cal S}}
\def \cL {{\cal L}}
\def \cH {{\cal H}}
\def \cA {{\cal A}}
\def \chl {{\cal H}(\ell^2)}
\def\t {\tau}
\def \pl {\Pi_g}
\def \ccp {\cC_p^{1,2}}
\newcommand{\be}{\begin{equation}}
\newcommand{\ee}{\end{equation}}
\newcommand{\bea}{\begin{eqnarray}}
\newcommand{\eea}{\end{eqnarray}}
\newcommand{\lb}{\label}
\newcommand{\ed}{ \end{document}}
\newcommand {\brl}{\begin{array}{l}}
\def \rw {\rightarrow}
\def \vy {\overrightarrow{y}}
\newtheorem{thm}{Theorem}[section]
\newtheorem{rem}{Remark}[section]
\newtheorem{propo}{Proposition}[section]
\newtheorem{defi}{Definition}[section]
\newtheorem{lem}{Lemma}[section]
\newtheorem{cor}{Corollary}[section]
\begin{document}
\title{Systems of Integro-PDEs with Interconnected Obstacles and Multi-Modes Switching Problem Driven by Lévy Process.}\author{
Said Hamad\`ene\thanks{Universit\'e du
Maine, LMM, Avenue Olivier Messiaen, 72085 Le Mans, Cedex 9, France ; e-mail: hamadene@univ-lemans.fr} \,\, and \, Xuxhe Zhao \thanks{School of Mathematics and Statistics, Xidian University, Xi'an 710071, PR China and Universit\'e du Maine, LMM, Avenue Olivier Messiaen, 72085 Le Mans, Cedex 9, France ; e-mail: sosmall129@hotmail.com.}}

\date{\today}
\maketitle
\begin{abstract}
In this paper we show existence and uniqueness of the solution in
viscosity sense for a system of nonlinear $m$ variational
integral-partial differential equations with interconnected
obstacles  whose coefficients $(f_i)_{i=1,\cdots, m}$ depend on
$(u_j)_{j=1,\cdots,m}$. From the probabilistic point of view, this
system is related to optimal stochastic switching problem when the
noise is driven by a L\'evy process. The switching costs depend on $(t,x)$. 
As a by-product of the main result we obtain that the value function of the 
switching problem is continuous and unique 
solution of its associated Hamilton-Jacobi-Bellman 
system of equations. The main tool we used is the
notion of systems of reflected BSDEs with oblique reflection driven
by a L\'evy process.
\medskip

\medskip

\end{abstract}
\setcounter{table}{1}
\section{Introduction}
In this paper, we study the existence and uniqueness of a solution
to the system of integro-partial differential equations (IPDEs in
short) of the following form: $\forall i=1,\cdots,m$, \be
\label{edpintro}\left\{
    \begin{array}{ll}
      min\{u_i(t,x)-\max\limits_{j\neq i}(u_j(t,x)-g_{ij}(t,x));\\\qq\qq
      -\partial_t u_i(t,x)-{\cal L} u_i(t,x)-f_i(t,x,(u_1,u_2,\cdots, u_m)(t,x))\}=0,\,\,(t,x)\in [0,T)\times \R,\\
      u_i(T,x)=h_i(x)
    \end{array}
    \right.\ee
where ${\cal L}$ is a generator associated with a stochastic
differential equation whose noise is driven by a L\'evy process $L:=(L_t)_{t\leq T}$
defined on a filtered probability space $(\Omega, {\cF},
(\cF)_{t\leq T}, \cP)$ and then $\cL$ is a non local operator (see (\ref{gener}) for its definiton).

This system is related to a stochastic
optimal switching problem since a particular case is actually its associated Hamilton-Jacobi-Bellman system.

Let us describe briefly the stochastic optimal switching problem. Let $(t,x)\in \esp$ and
$(X^{t,x}_s)_{s\leq T}$ be the solution of the following standard
stochastic differential equation:
$$dX^{t,x}_s=b(s,X^{t,x}_s)ds+
\sigma(s,X^{t,x}_{s-}) dL_s, \,\,\forall s\in [t,T]\mbox{ and
}X^{t,x}_s=x \mbox{ for }s\leq t.$$ Next let $(a_s)_{s\in [0,T]}$ be
the following pure jump process:
$$a_s:=\alpha_0\mathbbm{1}_{\{\theta_0\}}(s)+\sum\limits^\infty_{j=1}\alpha_{j-1}\mathbbm{1}_{]\theta_{j-1},\theta_j]}(s), \,\forall s\le T,$$ where
$\{\theta_j\}_{j\geq 0}$ is an increasing sequence of stopping times
with values in $[0,T]$ and $(\alpha_j)_{j\geq 0}$ are random
variables with values in $A:=\{1,\dots,m\}$ (the set of modes to which the controller can switch) such that for any $j\ge 0$,
$\alpha_j$ is $\cF_{\theta_j}-$measurable. The pair
$\Upsilon=((\theta_j)_{j\geq 0},(\alpha_j)_{j\geq 0})$ is called a
strategy of switching and when it satisfies $\cP[\theta_n<T, \forall
n\ge 0]=0$ it is moreover said admissible. Finally we denote by ${\cal A}^i_t$ the set of admissible
strategies such that $\a_0=i$ and $\theta_0=t$.

Assume next that for any $i=1,\dots,m$,
$f_i(t,x,(y_i)_{i=1,...,m})=f_i(t,x)$, i.e., $f_i$ does not depend
on $(y_i)_{i=1,m}$. Let $\us$ be an admissible
strategy of ${\cal A}^i_t$ with which one associates a payoff given by:
\be\lb {defpayoff}
J^{a}(t,x)=J(\us)(t,x):=\E[\int_t^Tf_{a(s)}(s,X^{t,x}_s)ds-\sum\limits_{j\ge
1}g_{\alpha_{j-1},\alpha_j}(\theta_j,X^{t,x}_{\theta_j})
\mathbbm{1}_{\{\theta_j<T\}}+h_{a_T}(X^{t,x}_T)]\ee where $f_{a(s)}(s,X^{t,x}_s)=\sum_{i\in A}f_i(s,X^{t,x}_s)1_{[a(s)=i]}$, $s\in [t,T]$,  (resp. $h_{a_T}(X^{t,x}_T) =\sum_{i\in A}h_i(X^{t,x}_T)1_{[a_T=i]}$) is the instantaneous (resp. terminal) payoff when the strategy $a$ (or $\Upsilon$) is implemented while $g_{i\ell}$ is the switching cost function when moving from mode $i$ to mode $\ell$ ($i,\ell \in A$, $i\neq \ell$). Next let us define the optimal payoff when starting from mode $i\in A$ at time $t$ by \be \label{valfunc} u_i(t,x):=\inf_{\us \in {\cal
A}_t^i}J(\us)(t,x) \ee

As a by-product of our general result we obtain that the value functions $(u_i(t,x))_{i\in A}$ (or optimal payoffs) of this switching problem is continuous and of polynomial growth and is the unique solution in viscosity sense of system (\ref{edpintro}). A similar problem has been already considered by Biswas et al. \cite{Bis}, however one should emphazise that in that work, the switching costs are constant and do not depend on $(t,x)$. This latter feature makes the problem easier to handle
since one can directly work with the functions $u_i$ defined in (\ref {defpayoff})-(\ref{valfunc}).
\ms

Optimal switching problems are well documented in the literature (see e.g.
\cite{Bis, carmonaludkovski, djehicheetal, chassagneux,
hamadenejeanblanc, hamadenezhang, brahimsaid, hutang, pham1, pham2,
zervos1, zervos2} etc. and the references therein), especially in connection with mathematical finance, energy market, etc.
\ms

The main objective and novelty of this paper is to study system
(\ref{edpintro}) in the general case, i.e., to allow for $f_i$ to depend on $(u_i)_{i=1,m}$ and
the switching costs $g_{ij}$ to depend on $(t,x)$ and to show that
(\ref{edpintro}) has a unique solution. Our method is based on the link
of (\ref{edpintro}) with systems of reflected BSDEs with inter-connected
obstacles driven by a L\'evy process, i.e., systems of the following
form: $\forall j=1,\dots,m$, $\forall s\leq T$, \be
\label{finitude1}\left\{
    \begin{array}{ll}
          Y^{j,t,x}_s=h_j(X^{t,x}_T)+\int^T_s f_j(r,X^{t,x}_r,(Y^{k,t,x}_r)_{k\in A},(U_r^{j,t,x,i})_{i\geq 1})dr-\sum\limits^\infty_{i=1}\int^T_s U^{j,t,x,i}_r dH^{(i)}_r +K^{j,t,x}_T-K^{j,t,x}_s\\
     Y^{j,t,x}_s\geq \max\limits_{k\neq j} \{Y^{k,t,x}_s-g_{jk}(s,X^{t,x}_s)\} \mbox{ and }
[Y^{j,t,x}_s-\max\limits_{k\neq j}\{Y^{k,t,x}_s-g_{jk}(s,X^{t,x}_s)\}]dK^{j,t,x}_s=0\\
    \end{array}
    \right.\ee
where $((H^{(i)}_s)_{s\leq T})_{i\geq 1}$ are the Teugels martingales
associated with the L\'evy process $L$. Under appropriate
assumptions on the data $(f_i)_{i=1,\dots,m}$, $(h_i)_{i=1,\dots,m}$
and $(g_{ij})_{i,j=1,\dots,m}$ we show existence and uniqueness of
$\cF_s$-adapted processes $((Y^{j,t,x}_s, (U^{j,t,x,i}_s)_{i\geq 1},
K^{j,t,x}_s)_{s\leq T})_{j\in A}$ which satisfy (\ref{finitude1}). Additionaly
there exist deterministic continuous functions $(u_j(t,x))_{j\in A}$
such that: \be\label{introfeynmankac} \forall s\in [t,T],
Y^{j,t,x}_s=u_j(s,X^{t,x}_s),\ee and we show that $(u_j(t,x))_{j\in A}$ is the
unique solution of (\ref{edpintro}).

In the Brownian framework of noise, the link between systems of PDEs
with interconnected obstacles and systems of reflected BSDEs with
oblique reflection has been already stated in several papers (see
e.g. \cite{saidmorlais, hutang}, etc.). Therefore in this work we extend this link to
the setting where the noise is driven by a L\'evy process.

This article is organized as follows. In Section 2 we collect the main
results on Teugels martingales. Section 3 is devoted to reflected
BSDEs driven by a Lévy process (existence and uniqueness of a
solution and comparison) and their connection with IPDEs with
obstacle. We finally consider the system of reflected BSDEs with
inter-connected obstacles (\ref{finitude1}) and we show existence
and uniqueness of a solution of this system when, mainly, the
functions $(f_i)_{i\in A}$ are Lipschitz in $((y_i)_{i\in A},\z)$ and
the switching costs verify the so-called non free loop property. We
construct a mapping which is a contraction in an appropriate Banach
space and which has a unique fixed point which provides the solution of
system (\ref{finitude1}). Section 4 is devoted to the study of
system of IPDEs (\ref{edpintro}). Contrarily to system of reflected
BSDEs (\ref{finitude1}), we only consider the case when the functions
$f_i$, $i\in A$, do not depend on $\z$. We first show that this
system has a solution in viscosity sense when for any $i\in A$, the
function $f_i$ is non-decreasing $\wr$ to $y_k$ ($k\neq i$) when the
other components are fixed. We then give a comparison result of
subsolutions and supersolutions of system (\ref{edpintro}) based on Jensen-Ishii's Lemma on PDEs with non-local term \cite{barles,Bis}.
As usual this comparison result insures continuity and uniqueness
of the solution of system (\ref{edpintro}). Finally we provide
another existence and uniqueness result of a solution for system
(\ref{edpintro}) in the case when for any $i\in A$, $f_i$ is
decreasing $\wr$ $y_k$ for any $k\neq i$ when the other components
are fixed. This result is deeply based on the first
existence and uniqueness of the solution of system (\ref{edpintro}) and,
on the other hand, the existence and uniqueness result of a solution of
system of reflected BSDEs (\ref{finitude1}). According to our
knowledge it cannot be obtained by using PDE techniques only. At the
end of this paper we give an Appendix where two complementary results  are
collected. The first one is related to the representation of the
$Y^j_s$ of the solution of system (\ref{finitude1}) as a value
function of a switching problem. As for the second one, it provides
an  equivalent definition of the viscosity solution of system
(\ref{edpintro}) which is somehow of local type. \qed
\section{Preliminaries}
A Lévy process is an $\R$-valued  RCLL (for right continuous with left limits)
stochastic process $L=\{L_t,t\geq 0\}$ defined on a probability
space $(\Omega,\cal F,\cP)$ with stationary and independent
increments ($L_0=0$) and stochastically continuous.

For $t\le T$ let us set ${\cal F}_t={\cal G}_t\vee\cal N$ where
${\cal G}_t:=\sigma\{L_s,0\leq s\leq t\}$ and $\cal N$ is the $\cP$-null sets of $\cal F$, therefore $\{ {\cal F}_t\}_{t\leq T}$ is
complete and right continuous. Next by $\cal P$ we denote the
$\sigma$-algebra of predictable processes on $[0,T]\times \Omega$ and finally for any RCLL process $(\Gamma_t)_{t\leq }$ we denote by
$\Gamma_{t-}:=\lim_{s\nearrow t}\Gamma_s$ and $\Delta
\Gamma_t:=\Gamma_{t}-\Gamma_{t-}$ its jump at $t$, $t\in (0,T]$.

We now introduce the following spaces:
\ms

\noindent (a) $\cS^2:= \{\varphi:=\{{\varphi_t},0\leq t\leq T\}$ is an
$\R$-valued, ${\cal F}_t$-adapted RCLL process  s.t.
$\E(\sup\limits_{0\leq t\leq T}{\vert {\varphi_t}\vert}^2)<\infty\}$
; $\cA^2$ is the subspace
of $\cS^2$ of non-decreasing continuous processes null at $t=0$\, ; \\
\noindent (b) $H^2:= \{\varphi:=(\varphi_t)_{t\leq T}$ is an $\R$-valued,
${\cal F}_t$-progressively measurable process such that $\E(\int^T_0{\vert {\varphi_t}\vert}^2dt)<\infty\} $;\\
\noindent (c) $\ell^2:=  \{x=(x_n)_{n\geq 1}$ is an $\R$-valued sequence
s.t.
${\Vert x\Vert}^2:=\sum\limits_{i=1}^\infty x^2_i<\infty\}$\,;\\
\noindent (d) $\cH^2(\ell^2):= \{\varphi=(\varphi_t)_{t\leq
T}=((\varphi^n_t)_{n\geq 1})_{t\leq T}$ such that $\forall n\geq 1$,
$\varphi^n$ is a ${\cal P}$-measurable process and
$$\bal \E(\int^T_0{\Vert \varphi_t\Vert}^2
dt)=\sum\limits^\infty_{i=1} \E(\int^T_0{\vert {\varphi^i_t}\vert}^2
dt)<\infty\}\,;\ea$$ (e) ${\cal L}^2: =\{\xi$, an $\R$-valued and $\cF_T$-measurable random variable such that
$\E[\vert\xi\vert^2]<\infty\}$ ;\\
\noindent (f) $\Pi_g$ is the space of deterministic functions $v(t,x)$ from
$\esp$ into $\R$ of polynomial growth, i.e., such that for some positive constants $p$ and $C$ one has,
$$
|v(t,x)|\leq C(1+|x|^p),\,\, \forall (t,x)\in \esp \,;$$
(g) $\ccp:=\cC^{1,2}(\esp)\cap \Pi_g$.\qed
\ms

Let us now recall the Lévy-Khintchine formula of a Lévy process
$(L_t)_{t\leq T}$ whose characteristic exponent is $\Psi$, $i.e.$,
$$\forall t\leq T \mbox{ and }\theta \in \R,\,\,\E(e^{i\theta L_t})=e^{t\Psi(\theta)}$$ with
$$\begin{array}{ll}
\Psi(\theta)&=ia\theta-\frac{1}{2}\varpi^2 \theta^2+\int_{\R} (e^{i\theta x}-1-i\theta x\mathbbm{1}_{(\vert x\vert < 1)} )\Pi(dx)\\
&=ia\theta-\frac{1}{2}\varpi^2  \theta^2+\int_{\vert x\vert
\geq 1}(e^{i\theta x}-1){\Pi(dx)}+\int_{0<\vert x\vert <1}(e^{i\theta
x}-1-i\theta x)\Pi(dx)
\end{array}$$
where $a\in\R$, $\varpi \geq 0$ and $\Pi$ is a $\sigma$-finite measure
on $\R^*:=\R-\{0\}$ (we set $\Pi(\{0\})=0$ and then the domain of integration is
the whole space), called the Lévy
measure of $L$, verifying \be\label{cdtint} \bal \int_{\R} (1\wedge x^2)\Pi (dx)<\infty.\ea\ee
Moreover we assume that $\Pi$ satisfies the following assumption:
\be\label{2cdtint}\bal \exists \epsilon >0 \mbox{ and } \lambda >0~ \mbox{ such that }
\int_{(-\epsilon,\epsilon)^c}e^{\lambda|x|}\Pi(dx)<+\infty.\ea\ee
Conditions (\ref{cdtint})-(\ref{2cdtint}) imply that for any $i\geq 2$,
\be\bal\label{levy1}\int_\R|x|^i\Pi(dx)<\infty\ea\ee
and then the process $(L_t)_{t\le T}$ have moments of any order.\\

Next following Nualart-Schoutens \cite{nualart} we define, for every $i\geq 1$,
the so-called power-jump processes $L^{(i)}$ and their compensated
version $Y^{(i)}$, also called Teugels martingales, as follows: $\forall t\leq T$,
$$\begin{array}{l}
L^{(1)}_t=L_t \mbox{ and for }i\ge 2,\,\,
L^{(i)}_t=\sum_{s\leq t}(\Delta L_s)^i,Y^{(i)}_t=L^{(i)}_t-t\E(L^{(i)}_1).
\end{array}$$
Note that for any $i\geq 2$ and $t\leq T$, $\E(L^{(i)}_t)=t\int_{\R}x^i\Pi(dx)$ exists, i.e., is defined and belongs to $\R$ (\cite{dellacheriemeyer}, pp.29).

An orthonormalization procedure can be applied to the martingales
$Y^{(i)}$ in order to obtain a set of pairwise strongly orthonormal
martingales ${(H^{(i)})}_{i\geq 1}$ such that each $H^{(i)}$
is a linear combination of $(Y^{(j)})_{j=1,i}$, i.e.,
$$
H^{(i)}=c_{i,i}Y^{(i)}+...+c_{i,1}Y^{(1)}.
$$
It has been shown in \cite{nualart} that the
coefficients $c_{i,k}$ correspond to the orthonormalization of the
polynomials $1, x, x^2, . . .$ with respect to the measure
$\nu(dx)=x^2\Pi(dx)+
\varpi^2\delta_0(dx)$ ($\delta_0$ is the Dirac measure at $0$). Specifically the
polynomials $(q_i)_{i\geq 0}$ defined by, for any $i\ge 1$, 
$$q_{i-1}(x)=c_{i,i}x^{i-1}+c_{i,i-1}x^{i-2}+...+c_{i,1}
$$ and satisfying
$$\bal \int_\R q_n(x)q_m(x)\nu(dx)=\delta_{nm},\,\, \forall n,m\geq
0.\ea$$ Next let us set
$$\begin{array}{l}
p_i(x)=xq_{i-1}(x)=c_{i,i}x^{i}+c_{i,i-1}x^{i-1}+...+c_{i,1}x \mbox{ and }\\
\tilde
p_i(x)=x(q_{i-1}(x)-q_{i-1}(0))=c_{i,i}x^{i}+c_{i,i-1}x^{i-1}+...+c_{i,2}x^2.\end{array}
$$
Then for any $i\geq 1$ and $t\leq T$ we have:
$$\begin{array}{ll}
H^{(i)}_t&=\sum_{0<s\leq t}\{c_{i,i}(\Delta
L_s)^i+...+c_{i,2}(\Delta
L_s)^2\}+c_{i,1}L_t-t\E[c_{i,i}(L_1)^{(i)}+...+c_{i,2}(L_1)^{(2)}]-
tc_{i,1}\E(L_1)\\\\
{}&=q_{i-1}(0)L_t+\sum_{0<s\leq t}\tilde p_i(\Delta
L_s)-t\E[\sum_{0<s\leq 1}\tilde p_i(\Delta L_s)]-tq_{i-1}(0)\E(L_1).
\end{array}
$$
As a consequence, for any $t\leq T$ and $i\ge 1$, $\Delta H^{(i)}_t=p_i(\Delta L_t)$ for each $i\geq
1$. In the particular case of $i=1$, we obtain
\be \label{formc11}\bal H^{(1)}_t=c_{1,1}(L_t-t\E(L_1)) \mbox{ with }c_{1,1}=[\int_\R x^2\Pi(dx)+\varpi^2]^{-\frac{1}{2}}\mbox{
and }\E[L_1]=a+\int_{|x|\geq 1}x\Pi(dx).\ea\ee Finally note that for any
$i,j\geq 1$ the predictable quadratic variation process of $H^{(i)}$ and $H^{(j)}$ is\\ $\langle H^{(i)},H^{(j)}\rangle_t=\delta_{ij}t, \forall t\leq T$.\qed
\begin{rem} If $\Pi=0$, we are in the classical Brownian case and all non-zero
degree polynomials $q_i(x)$ will vanish, giving $H^{(i)}=0$,
$i\geq 2$. On the other hand, if $\Pi$ only has mass at $1$, we are in the Poisson case
and once more $H^{(i)}=0$, $i\geq 2$. Both cases are degenerate
ones in this Lévy process framework. \qed
\end{rem}

The main result in the paper by Nualart-Schoutens  \cite{nualartschoutens} is the
following representation property which allows for developping the BSDE theory in this Lévy framework.
\begin{thm}\lb{nusch}(\cite{nualartschoutens}, pp.118). Let $\xi$
be a random variable of $\cL^2$, then there exists a process
$Z=(Z^i)_{i\geq 1}$ that belongs to $\cH^2(\ell^2)$ such that:
$$\bal
\xi=\E(\xi)+\sum_{i\geq 1}\int_0^TZ^i_sdH^{(i)}_s.\qed \ea$$
\end{thm}

\section{ Systems of Reflected BSDEs with Oblique Reflection driven by a Lévy process}

\subsection{Reflected BSDE driven by a Lévy process and their relationship with IPDEs}
As a consequence of Theorem \ref{nusch}, and as in the framework of
Brownian noise only, one can study standard BSDEs or reflected ones.
The result below related to existence and uniqueness of a solution
for a reflected BSDE driven by a Lévy process, is proved in
\cite{renotmani}. Indeed let us introduce a triple $(f, \xi,S)$
that satisfies: \ms

\noindent {\bf Assumptions (A1):}\\
\noindent (i) $\xi$ a  random variable of $\cL^2$ which stands for
the terminal value ;\\
(ii) $f$: $[0,T]\times\Omega\times \R\times \ell^2\longrightarrow
\R$  is a function such that the process $(f(t,0,0))_{t\leq T}$
belongs to $H^2$ and there exists a constant $\kappa>0$ verifying
$$
|f(t,y,\z)-f(t,y',\z')|\leq \kappa (|y-y'|+\|\z-\z'\|_{\ell^2}), \mbox{
for every }t,y,y',\z \mbox{ and }\z'.
$$
(iii) $S:={(S_t})_{0\leq t\leq T}$ is a process of $\cS^2$ such that
$ S_T\leq\xi$, $\cP-a.s.$, and whose jump times are inaccessible stopping
times. This in particular implies that for any $t\leq T$,
$S^p_t=S_{t-}$, where $S^p$ is the predictable projection of $S$
(see e.g. \cite{dellacherie}, pp.58
) for more details on those notions.\\

In \cite{renotmani}, the authors have proved the following result
related to existence and uniqueness of the solution of one barrier reflected BSDEs
whose noise is driven by a Lévy process.
\begin{thm} Assume that the triple $(f,\xi,S)$
satisfies Assumptions $(A1)$. Then there exists a unique triple of
processes $(Y,U,K):=((Y_t,U_t,K_t))_{t\leq T}$ with values in
$\R\times \ell^2\times \R^{+}$ such that: \be
\label{finitude}\left\{
     \begin{array}{ll}
    (Y,U,K)\in \cS^2\times \chl\times \cA^2;\\
     Y_t=\xi +\int^T_t f(s,Y_s,U_s)ds+K_T-K_t -
     \sum\limits^{\infty}_{i=1} \int^T_t U^i_s dH^{(i)}_s,\,\forall t\leq T; \\
    Y_t\geq S_t,\,\forall\,\, 0\leq t\leq T \mbox{ and }\int^T_0 (Y_t-S_t)dK_t=0,\,\cP-a.s.
    \end{array}
    \right.\ee
The triple $(Y,U,K)$ is called the solution of the reflected BSDE
associated with $(f,\xi,S)$. $\qed$
\end{thm}

To proceed we need to compare solutions of reflected BSDEs of types (\ref{finitude}). So let
us consider a stochastic process $V=(V_t)_{t\leq T}=(V^i)_{i\geq
1}=((V^i_t)_{t\leq T})_{i\geq 1}$ which belongs to $\cH^2(\ell^2)$ and let $M:=(M_t)_{t\le T}$ be
the stochastic integral defined by:$$\bal \forall t\leq
T,\,\,M_t:=\sum\limits^\infty_{i=1}\int^t_0 V^i_s dH^{(i)}_s.\ea$$ We
next denote by $\varepsilon(M):=(\varepsilon(M)_t)_{t\leq T}$ the process that satisfies: 
$$\forall t\leq T,\,\,\bal \varepsilon(M)_t=1+\int^t_0 \varepsilon(M)_{s-}dM_s.\ea$$By Doléans-Dade's formula we have (see e.g. \cite{philip}): $$\forall t\leq
T,\,\,\varepsilon(M)_t=\exp\Big\{M_t-\frac{1}{2}[M,M]^c_t-\sum\limits_{0\leq
s\leq t} \dt M_s\Big\}\prod_{0\leq s\leq t}\{1+\dt
M_s\}.$$

Let us now introduce the following assumption on the  process $V$.
\ms

\noindent {\bf Assumptions (A2)}: The process $V=(V^i)_{i\geq 1}=((V^i_t)_{t\leq T})_{i\geq 1}$
verifies\begin{equation}\label{hupov1}
\cP -a.s., \forall t\leq T, 
\sum\limits^\infty_{i=1} V^i_t p_i(\dt
L_t)>-1.\end{equation}
and there exists a constant $C$ such
that:
\begin{equation}\label{hupov2}
\sum\limits^\infty_{i=1}|V^i_t|^2\leq C,\,\,d\cP\otimes dt-a.e.
\end{equation}We then have:
\begin{propo}\label{prop31} Assume that Assumption (A2) is fulfilled. Then, $\cP$-a.s., for any $t\in [0,T]$,
$\varepsilon(M)_t>0$ and $\varepsilon(M)$ is a martingale of $\cS^2 $.
\end{propo}
\begin{proof}
First note that for any $t\leq T$, $$\Delta
M_t=\sum\limits^\infty_{i=1} V^i_t\Delta H^{(i)}_t=\sum\limits^\infty_{i=1}V^i_t p_i(\Delta L_t)>-1,$$
therefore for any $t\leq T$,
$\varepsilon (M_t)> 0$. Next by using Doléans-Dade's formula and since \\
$d\langle H^{(i)},H^{(j)}\rangle_s=\delta_{ij}ds$, we have: $\forall t\leq T$,
$$\begin{array}{ll}
\varepsilon (M)^2_t&=\varepsilon(2M+[M,M])_t\\
&=\varepsilon(2\sum\limits^\infty_{i=1}\int^._0 V^i_s
dH^{(i)}_s+\sum\limits^\infty_{i=1}\sum\limits^\infty_{j=1}\int^._0V^i_sV^j_sd[H^{(i)},H^{(j)}]_s)_t\\
&=\varepsilon(2\sum\limits^\infty_{i=1}\int^._0 V^i_s
dH^{(i)}_s+\sum\limits^\infty_{i=1}\int^._0|V^i_s|^2ds+\sum\limits^\infty_{i=1}\sum\limits^\infty_{j=1}\int^._0V^i_sV^j_sd([H^{(i)},H^{(j)}]_s-\langle
H^{(i)},H^{(j)}\rangle_s))_t\\
&=\varepsilon(N)_t
\mbox{exp}\{\sum\limits^\infty_{i=1}
\int^t_0|V^i_s|^2ds\}
\end{array}$$
where $$\bal N_t=2\sum\limits^\infty_{i=1}\int^t_0 V^i_s
dH^{(i)}_s+\sum\limits^\infty_{i=1}\sum\limits^\infty_{j=1}\int^t_0V^i_sV^j_sd([H^{(i)},H^{(j)}]_s-\langle
H^{(i)},H^{(j)}\rangle_s),\,t\leq T,\ea$$ is a local martingale. On the other
hand, the quantity $\sum\limits^\infty_{i=1}\int^T_0|V^i_s|^2ds$ is
bounded and $\varepsilon(N)\ge 0$, then 
$$
\E[(\varepsilon(M)_t)^2]\leq C\E[\varepsilon(N)_0]\leq C,\,\,\forall t\leq T,
$$since $\varepsilon(N)$ is a supermartingale. It follows that $\varepsilon(M)$ is not only a local martingale but also a martingale and then by
Doob's maximal inequality it belongs to $\cS^2$. \end{proof}
\begin{rem} The result of Proposition $3.1$ still holds true if instead of (\ref{hupov2}) we only have
\be \label{weakvi}\bal \sum\limits^\infty_{i=1}\int_0^T|V^i_s|^2ds\leq
C,\,\,\cP-a.s.\qed \ea\ee
\end{rem}

Next for two processes $U^i=(U^i_k)_{k\geq 1}$, $i=1,2$, of
$\cH^2(\ell^2)$ we define their scalar product in $\cH^2(\ell^2)$ which we denote by $\langle
U^1,U^2\rangle^p:=(\langle U^1,U^2\rangle^p_t)_{t\leq T}$ as:
$$
\forall t\leq T,\,\,\langle U^1,U^2\rangle^p_t=\sum_{k\geq
1}U_k^1(t)U_k^2(t).
$$
\begin{propo}:
Let $\xi\in{\cal L}^2$, $\varphi:=(\varphi_s)_{s\leq T}\in H^2$,
$\delta:=(\delta_s)_{s\leq T}$ a uniformly bounded process, and
finally let $V=(V^i)_{i\geq 1}\in \cH^2(\ell^2)$ satisfying (A2). Let
$(Y,U):=(Y_t,U_t)_{t\leq T}\in \cS^2\times\cH^2(\ell^2)$ be the
solution of the following BSDE: \be\label{edsr}\bal \forall t\leq
T,\,\,Y_t=\xi+\int^T_t(\varphi_s+\delta_sY_s+\langle
V,U\rangle^p_s)ds-\sum\limits^\infty_{i=1}\int^T_tU^i_sdH^{(i)}_s.\ea\ee
For $t\leq T$, let $(X^t_s)_{s\in [t,T]}$ be the process defined as
follows: \be\label{Xt}\forall s\in [t,T],\,\, X^t_s=e^{\int^s_t
\delta_rdr}\frac{\varepsilon(M)_s}{\varepsilon(M)_t}. \ee Then for any $t\leq T$, $Y_t$ satisfies:
$$\bal Y_t=\E[X^t_T\xi+\int^T_t X^t_s\varphi_s ds|\cF_t], \,\, \cP-a.s..\ea$$ On the other hand, if $(Y',U')\in
\cS^2\times\cH^2(\ell^2)$ is the solution of the BSDE: \be\bal 
Y'_t=\xi+\int^T_tf(s,Y'_s,U'_s)ds-\sum\limits^\infty_{i=1}\int^T_tU'^i_sdH^{(i)}_s,
\forall t\leq T\ea \ee where $$f(t,Y'_t,U'_t)\geq \varphi_t+\delta_t
Y'_t+\langle V, U'\rangle^p_t,\,d\cP\otimes dt-a.s.$$ then for any
$t\leq T$, $$\bal Y'_t\geq \E[X^t_T\xi+\int^T_t X^t_s\varphi_s
ds|\cF_t],\,\cP-a.s..\ea$$
\end{propo}
\begin{proof} First note that the processes $(Y,U)$ and $(Y',U')$ exist thanks to Theorem $3.1$. Let us now fix $t\in [0,T]$. Since $V$ satisfies (A2) then $\varepsilon(M)>0$ which implies that $(X^t_s)_{s\in [t,T]}$ is defined 
$\omega$ by $\omega$. On the other hand it satisfies
$$
\forall s\in [t,T], \,dX^t_s=X^t_{s-}(\delta_sds+dM_s)
$$
and since $\delta$ is uniformly bounded then as in Proposition
$3.1$, one can show that $\E[\sup_{s\in [t,T]}|X^t_s|^2]<\infty.$
Now by It\^o's formula, for any $s\in [t,T]$, we have
\begin{align*}
-d(Y_s X^t_s)=&-Y_{s-}dX^t_s-X^t_{s-}dY_s-d[Y,X^t]_s\\
=&-X^t_{s-}Y_{s-}\delta_sds-Y_{s-}X^t_{s-}dM_s+X^t_{s-}\varphi_sds+X^t_{s-}\delta_sY_sds\\
&\qq \qq -X^t_{s-}\{\sum_{i\geq 1}
U^i_sdH^{(i)}\}-X^t_{s-}\{\sum\limits^\infty_{i=1}\sum\limits^\infty_{j=1}V^i_sU^j_sd([H^{(i)},H^{(j)}]_s-\langle H^{(i)},H^{(j)}\rangle_s)\}\\
=&X^t_s\varphi_sds-dN_s
\end{align*}
where for any $s\in [t,T]$
$$dN_s=Y_{s-}X^t_{s-}\{\sum\limits^\infty_{i=1}V^i_sdH^{(i)}_s\}+X^t_{s-}\{\sum_{i\geq
1}
U^i_sdH^{(i)}_s\}+X^t_{s-}\{\sum\limits^\infty_{i=1}\sum\limits^\infty_{j=1}V^i_sU^j_sd([H^{(i)},H^{(j)}]_s-\langle
H^{(i)},H^{(j)}\rangle_s\}.$$ Note that since $X^t$ is uniformly
square integrable, $Y\in \cS^2$, $U\in \cH^2(\ell^2)$ and finally
taking into account Assumption (A2) on $V$, we get that $N$ is a
uniformly integrable martingale on $[t,T]$. Therefore taking
conditional expectation to obtain:$$\brl Y_t=\E[X^t_T\xi+\int^T_t
X^t_s\varphi_s ds|\cF_t], \,\cP-a.s.\ea$$ which is the desired result.

We now focus on the second part of the claim. By It\^o's formula we have: $\forall s\in [t,T]$,
\begin{align*}
-d(Y'_s X^t_s)=&-Y'_{s-}dX^t_s-X^t_{s-}dY'_s-d[Y',X^t]_s\\
=&-X^t_{s-}Y'_{s-}\delta_sds-Y'_{s-}X^t_{s-}\{\sum\limits^\infty_{i=1}V^i_sdH^{(i)}_s\}+X^t_{s-}f(s,Y'_s,U'_s)ds-X^t_{s-}\{\sum\limits^\infty_{i=1}U'^i_sdH^{(i)}_s\}\\
&\qq \qq -X^t_{s-}\{\sum\limits^\infty_{i=1}\sum\limits^\infty_{j=1}V^i_sU'^j_sd[H^{(i)},H^{(j)}]_s\}.
\end{align*}
Next since $X^t\geq 0$ and taking into account the inequality which $f$ verifies to obtain $$-d(Y'_sX^t_s)\geq
X^t_s\varphi_sds-dN'_s~~\cP-a.s.,$$where for any $s\in [t,T]$,
$$dN'_s=Y'_{s-}X^t_{s-}\{\sum\limits^\infty_{i=1}V^i_sdH^{(i)}_s\}-X^t_{s-}\{\sum_{i=1}^{\infty}
U'^i_sdH^{(i)}_s\}-X^t_{s-}\{\sum\limits^\infty_{i=1}\sum\limits^\infty_{j=1}V^i_sU'^j_sd([H^{(i)},H^{(j)}]_s-\langle
H^{(i)},H^{(j)}\rangle_s)\}.$$ But once more $N'$ is a
uniformly integrable martingale then by  taking the conditional expectation we
obtain: $$\brl Y'_t\geq \E[X^t_T\xi+\int^T_t X^t_s\varphi_s
ds|\cF_t],\,\,\cP-a.s.\ea
$$which completes the proof.
\end{proof}
We are now ready to give a comparison result of solutions of two
BSDEs of type (\ref{finitude}).
\begin{propo}\label{comp}
For $i=1,2$, let $(f_i,\xi_i)$ be a pair that satisfies Assumption
(A1)-(i),(ii) and let $(Y^i,U^i)\in \cS^2\times\cH^2(\ell^2)$ be the
solution of the following BSDE: $\forall t\leq
T$, $$\brl
Y^i_t=\xi_i+\int^T_tf_i(s,Y^i_s,U^i_s)ds
-\sum\limits^\infty_{j=1}\int^T_tU^{i,j}_s
dH^{(j)}_s.\ea$$ Assume that:

(i) For any $U^1,U^2\in \cH^2(l^2)$,  there exists a process
$V^{U^1,U^2}=(V_j^{U^1,U^2})_{j\geq 1}$ (which may depend on $U^1$ and $U^2$)
satisfying (A2) such that $f_1$ verifies: \be \label{ineq}
f_1(t,Y_t^2,U_t^1)-f_1(t,Y_t^2,U_t^2)\geq \langle
V^{U^1,U^2},(U^1-U^2)\rangle^p_t,\,\, d\cP\otimes dt-a.e. ;\ee

(ii) $\cP-a.s.,\,\,\xi_1\geq \xi_2$ and
\be \label{inegf}f_1(t,Y^2_t,U^2_t)\geq f_2(t,Y^2_t,U^2_t), d\cP\otimes dt-a.e..\ee Then $\cP$-a.s., $Y^1_t\geq Y^2_t,\,\,\forall
t\in[0,T].$
\end{propo}
\begin{proof} Let us set
$\bar{Y}=Y^1-Y^2$, $\bar{U}=U^1-U^2$ and $\bar \xi=\xi^1-\xi^2$,
then $\forall
t\in[0,T]$,
$$\bal\bar{Y}_t=\bar{\xi}+\int^T_t\{
f_1(s,Y^1_s,U^1_s)-f_2(s,Y^2_s,U^2_s)\}ds-\sum\limits^\infty_{j=1}\int^T_t\bar U^j_sdH^{(j)}_s.\ea$$\\
Next let us set
\be \label{difdelta}\forall s\leq T,\,\,\delta_s=(f_1(s,Y^1_s,U^1_s)-f_1(s,Y^2_s,U^1_s))\times (\bar{Y}_s)^{-1}{\mathbbm{1}}_{\{\bar{Y}_s\neq 0\}} \mbox{ and }
\varphi_s=f_1(s,Y^2_s,U^2_s)-f_2(s,Y^2_s,U^2_s).\ee Then by (\ref{inegf}) we have, $\varphi_s\geq 0,~d\cP\otimes dt-a.e.$. On the other hand $(\delta_s)_{s\in [0,T]}$ is bounded since $f_1$ is uniformly
Lipschitz. Finally we have
$$f_1(s,Y^1_s,U^1_s)-f_2(s,Y^2_s,U^2_s)\geq
\varphi_s+\delta_s\bar{Y}_s+\langle V^{U^1,U^2},\bar U
\rangle^p_s,~d\cP\otimes ds-a.e..$$ Therefore thanks to Proposition
$3.2$ we get,
$$\forall t\leq T,\,\,\bar{Y}_t\geq
\E[X^t_T\bar{\xi}+\int^T_tX^t_s\varphi_sds|F_s]\geq 0,\,\,\cP-a.s.$$
where $(X^t_s)_{s\in [t,T]}$ is defined in the same way as in
(\ref{Xt}) with the new processes $\delta$ and $\varphi$ defined in (\ref{difdelta}). As $X^t$, $\bar \xi$
and $\varphi$ are non-negative then for any $t\leq T$, $\bar Y_t\geq
0$ which implies that $\cP-a.s., \forall t\leq T, Y^1_t\geq Y^2_t$
since $Y^1$ and $Y^2$ are RCLL. The proof of the claim is now complete. 
\end{proof}
\begin{rem} Conditions (\ref{ineq}) and
(\ref{inegf}) can be replaced respectively with
\be \label{ineqy2}
f_2(t,Y_t^2,U_t^1)-f_2(t,Y_t^2,U_t^2)\geq \langle
V^{U^1,U^2},(U^1-U^2)\rangle^p_t,\,\, d\cP\otimes dt-a.e.\ee and
\be \label{inegfy2}f_1(t,Y^1_t,U^1_t)\geq f_2(t,Y^1_t,U^1_t), d\cP\otimes dt-a.e.. \ee In this case, with the other properties,
one can show that we have $\cP$-a.s., $Y^1\geq Y^2$. \qed
\end{rem}
\begin{rem} Point (i) of Proposition \ref{comp} is satisfied in the following cases:
\ms

\noindent (i) $f$ does not depend on the component $\z$ ;

\noindent (ii) If $L$ reduces to a Poisson process, we have $H^{(i)}\equiv 0$ for
all $i\geq 2$, then Assumption (A2) reads: (a) $V=(V_t)_{t\in [0,T]}$ is bounded ; (b) for any stopping time $\t$, such
that $\triangle L_\tau\neq 0$, $ V_\tau >-1,\,\,\cP-a.s.$.

\noindent (iii) The generator $f$ satisfies $$f(t,y,\z)=h_1(t,y,\sum_{i\geq
1}\theta^i_t\z^i), \,\forall (t,y,\z)\in \esp\times  \ell^2$$ where
the mapping $\eta \in \R\mapsto h_1(t,y,\eta)$ is non decreasing and uniformly Lipschitz and $((\theta^i_t)_{i\geq 1})_{t\leq T}$ satisfies$$\sum_{i\geq 1}|\theta_t^i|^2\leq C,\,\, dt \otimes d\cP -a.e. \mbox{ and }\cP-a.s.,\forall t\leq T, \, \sum_{i\geq 1}\theta^i_t p_i(\Delta
L_t)\geq 0.\qed$$
\end{rem}
We finally provide a comparison result of solutions of reflected
BSDEs of type (\ref{finitude}) which will be useful in the sequel.
\begin{propo}\label{comparison}For $i=1,2$, let $(f_i,\xi_i,S^i)$ be a triple which satisfies Assumption (A1) and let $(Y^i_t,K^i_t,U^i_t)_{t\le T}$ be the  solution of the RBSDE associated with $(f_i,\xi_i,S^i)$.   Assume that:

(i) $\cP-a.s$, $\xi_1\geq\xi_2$ and $\forall
t\in[0,T],~~f_1(t,y,\z)\geq f_2(t,y,\z)$ and $S^1_t\geq S^2_t$ ;

(ii) $f_1$ verifies  condition (\ref{ineq}).

\noindent Then $\cP$-a.s. for any $t\leq T$, $Y^1_t\geq Y^2_t. $
\end{propo}
\begin{proof}
For $i=1,2$, let us consider the following sequence of processes
$(Y^{i,n},U^{i,n})\in \espyz$, $n\ge 0$, that satisfy:
$$\begin{array}{l}Y^{i,n}_t=\xi_i+\int^T_t
f_i(s,Y^{i,n}_s,U^{i,n}_s)ds+n\int^T_t(Y^{i,n}_s-S^i_s)^-ds-\sum\limits^\infty_{j=1}\int^T_tU^{i,n,j}_sdH^{(j)}_s,\,\,\forall
t\le T\ea$$ and let us denote by
$$f^n_i(s,y,\z):=f_i(s,y,\z)+n(y-S^i_s)^-.$$ For any $n\ge
0$, $f^n_1$ satisfies (\ref{ineq}) and $f_1^n\ge f_2^n$. Therefore
using the comparison result of Proposition \ref{comp}, we deduce
that: $\forall n\ge 0$, \be \label{ineqyin}\cP-a.s., \forall t\leq
T, Y^{1,n}_t\geq Y^{2,n}_t.\ee But since $f_1$ verifies (\ref{ineq})
then we can show that for $i=1,2$, $Y^{i,n}\nearrow Y^i$ in $\cS^2$ since the processes $S^i$ do not have predictable jumps (see e.g. \cite{saidouknine}, Theorem 1.2.a, pp. 5). Thus,  inequality
(\ref{ineqyin}) implies that $\cP$-a.s., $Y^1\geq Y^2$.
\end{proof}

We are now going to make a connection between reflected BSDEs and
their associated IPDEs with obstacle. So let $(t,x)\in \esp$ and let $(X^{t,x}_s)_{s\leq T}$ be the solution of the following standard SDE driven by the Lévy process $L$, i.e., \be
\label{SDE}\bal X^{t,x}_s=x+\int^{t\vee s}_t
b(r,X^{t,x}_r)dr+\int^{t\vee s}_t\sigma(r,X^{t,x}_{r-} )dL_r, \q \forall
 s\leq T, \ea \ee
 where we
assume that the functions $b$ and $\sigma$ are jointly continuous,
 Lipschitz continuous $\wr$ $x$ uniformly in
$t$, i.e., there exists a constant $C\geq 0$ such that for any $t\in
[0,T]$, $x$,$x^\prime\in \R$,\be\label{eqlipb}
|\sigma(t,x)-\sigma(t,x^\prime)|+|b(t,x)-b(t,x^\prime)|\leq
C|x-x^\prime|.\ee As a consequence, the functions $b(t,x)$ and $\sigma(t,x)$ are of linear growth. We additionally assume that $\sigma$ is bounded, i.e., there exists a constant $C_\sigma$ such that \be\label{sigma}\forall (t,x)\in [0,T]\times \R, |\sigma(t,x)|\leq C_\sigma.\ee
Under the above conditions on $b$ and $\sigma$, the
process $X^{t,x}$ exists and is unique (see e.g. \cite{philip}, pp.249),
and satisfies:\be\label{sde1}\forall p\geq 1,\,\E[\sup\limits_{s\leq
 T}|X^{t,x}_s|^p]\leq C(1+|x|^p).\ee
Next let us consider the following functions: $$\begin{array}{l}
h:\,x\in \R\mapsto
h(x)\in \R;\\f:\,(t,x,y,\z)\in [0,T]\times \R^{1+1}\times \ell^2\mapsto
f(t,x,y,\z)\in \R;\\\Psi:\,(t,x)\in [0,T]\times \R\mapsto \Psi(t,x)\in \R,\end{array}$$
which we assume satisfying:
\ms

\noindent {\bf Assumptions (A3):}\\
(i) $h$, $\Psi$ and $f(t,x,0,0)$ are jointly continuous and belong to $\Pi_g$ ; \\
(ii) the mapping $(y,\z)\mapsto f(t,x,y,\z)$ is Lipschitz continuous uniformly in $(t,x)$ ; \\
(iii) For any $x\in \R$, $h(x)\geq \Psi(T,x)$ ; \\
(iv)  The generator $f$ has the following form: $$f(t,x,y,\z)=\underbar h(t,x,y,\sum_{i\geq
1}\theta^i_t\z^i), \,\forall (t,x,y,\z)\in [0,T]\times \R^{1+1}\times
\ell^2$$ where the mapping $\eta \in \R\longmapsto \underbar h(t,x,y,\eta)$ is
non decreasing, and there exists a constant $C>0$, such that
$\forall t\in [0,T]$, $x,y, \eta ,\eta'\in \R$,
$$|\underbar h(t,x,y,\eta)-\underbar h(t,x,y,\eta')\leq C|\eta-\eta'|.$$Moreover $(\theta^i_t)_{i\geq 1}$ satisfies $$\sum_{i\geq 1}|\theta_t^i|^2\leq C,\,\, dt \otimes d\cP -a.e. \mbox{ and }\cP-a.s.,\forall t\leq T, \, \sum_{i\geq 1}\theta^i_t p_i(\Delta
L_t)\geq 0.\qed$$

Next let $(t,x)\in \esp$ be fixed and let us consider the following
reflected BSDE:  \be \label{RBSDE}\left\{
    \begin{array}{ll}

(Y^{t,x},U^{t,x},K^{t,x})\in \cS^2\times \chl\times \cA^2;\\
Y^{t,x}_s=h(X^{t,x}_T)+\int^T_s f(r,X^{t,x}_r,Y^{t,x}_r,U^{t,x}_r)dr+K^{t,x}_T-K^{t,x}_s-\sum\limits^\infty_{i=1}\int^T_s U^{t,x,i}_r dH^{(i)}_r\\\forall s\leq T,\,\,
     Y^{t,x}_s\geq \Psi(s,X^{t,x}_s)\mbox{ and } \int^T_0(Y^{t,x}_s-\Psi(s,X^{t,x}_s))
     dK^{t,x}_s=0,\,\cP-a.s.
    \end{array}
    \right.\ee
Under assumptions (A3)-(i),(ii),(iii), the reflected BSDE
(\ref{RBSDE}) is well-posed and, thanks to Theorem 3.1, has a unique solution $(Y^{t,x},U^{t,x},K^{t,x})$. Moreover the
following estimate holds true: \be
\label{estimateyku}\begin{array}{c}\E\Big[\sup\limits_{0\leq s\leq
T}|Y^{t,x}_s|^2+\int_0^T\{\sum_{i\geq 1}|U^{t,x,i}_s|^2\} ds\Big]\leq
C\E\Big[|h(X^{t,x}_T)|^2+\int^T_0|f(s,X^{t,x}_s,0,0)|^2ds+\sup\limits_{0\leq
s\leq T}|\Psi(s,X^{t,x}_s)|^2\Big ].\ea\ee On the other hand, the quantity
\be\label{UY}u(t,x)=Y^{t,x}_t, \ee is deterministic, continuous and satisfies  $$\forall (t,x)\in \esp, \forall s\in [t,T], \,\,Y^{t,x}_s:=u(s,X^{t,x}_s).$$
Fore more details, one can see e.g. (\cite{renotmani}, pp.1265). Finally note that
under Assumptions (A3) and by (\ref{estimateyku}) the function $u$ belongs also to $\Pi_g$. \ms

We now introduce the following IPDE with obstacle:

\be \label{IPDE}\left\{
    \begin{array}{ll}
     \min\Big \{u(t,x)-\Psi (t,x);-\partial_t u(t,x)-{\cal L}u(t,x)-f(t,x,u(t,x),\Phi (u)(t,x))\Big \}=0,\,(t,x)\in [0,T)\times \R,\\
     u(T,x)=h(x),
    \end{array}
    \right.\ee
where ${\cal L}$ is the generator associated with the process $X^{t,x}$ of (\ref{SDE}) which has the following expression: \be \lb{gener}\begin{array}{l} {\cal
L}u(t,x)=(\E[L_1]\sigma(t,x)+b(t,x))\partial_x
u(t,x)+\frac{1}{2}\sigma(t,x)^2\varpi^2\partial^2_{xx}
u(t,x)\\\qq\qq\qq\qq
+\int_\R[u(t,x+\sigma(t,x)y)-u(t,x)-\partial_xu(t,x)\sigma(t,x)y]\Pi(dy)
\end{array} \ee and
$$\bal \Phi (u)(t,x)=\Big(\frac{1}{c_{1,1}}
\partial_x u(t,x)\sigma(t,x)\mathbbm{1}_{k=1}+\int_\R(u(t,x+\sigma(t,x)y)-u(t,x)-\partial_x u(t,x)y)p_k(y)\Pi(dy)\Big)_{k\ge 1}\ea$$where $c_{1,1}$ is defined in (\ref{formc11}).

We are going to  consider solutions of $(\ref{IPDE})$ in viscosity
sense whose definition is as follows:
\begin{defi}\label{defviscosimple}A continuous function $u:[0,T]\times \R\rightarrow \R$ is said to be a viscosity subsolution (resp. supersolution) of $(\ref{IPDE})$ if:
\ms

\noindent (i) $u(T,x)\leq  h(x)$ (resp.  $u(T,x)\geq  h(x)$) ; 

\noindent (ii) for any $(t,x)\in (0,T)\times \R$ and  for any $\varphi\in \ccp$ such that $\varphi(t,x)=u(t,x)$ and $\varphi-u$ attains its global minimum (resp. maximum) at (t,x),
$$\min\Big \{u(t,x)-\Psi(t,x);-\partial_t \varphi(t,x)-{\cal L}\varphi(t,x)-f(t,x,\varphi(t,x),\Phi (\varphi)(t,x))\Big \}\leq 0\,\, (resp. \,\geq 0).$$
The function $u$ is said to be a viscosity solution of
$(\ref{IPDE})$ if it is both its viscosity subsolution and supersolution.\qed\end{defi} In \cite{renotmani}, Ren-El Otmani (Theorem 5.8, pp.1265) have shown that under Assumption (A3), the function $u$ defined in (\ref{UY}) is a viscosity solution for (\ref{IPDE}).
\subsection{Systems of reflected BSDEs with inter-connected obstacles driven by a L\'evy process and multi-modes switching problem.}
We now introduce the following functions $f_i$, $h_i$ and
$g_{ij}$, $i,j\in A$:
\be \lb{introfct}\begin{array}{ll}
&f_i\,\,: (t,x,(y^i)_{i=1,m},\z)\in \esp \times \R^{m}\times \ell^2\longmapsto f_i(t,x,(y^i)_{i=1,m},\z)\in \R\,,\\
&g_{ij}\,: (t,x)\in \esp\longmapsto g_{ij}(t,x)\in \R\,,\\
&h_i\,: x\in \R\longmapsto h_i(x)\in \R
\end{array}\ee
which we assume satisfying:
\ms

\nd {\bf  Assumptions (A4):}\\
(I)\,\, For any $i\in A$:

(i) The mapping $(t,x)\rightarrow f_i(t,x,\overrightarrow{y},\z)$ is continuous uniformly with respect to $(\overrightarrow{y},\z)$ where $\overrightarrow{y}=(y^i)_{i=1,m}$ ;

(ii) The mapping $(\overrightarrow{y},\z)\mapsto f_i(t,x,\overrightarrow{y},\z)$ is Lipschitz continuous  uniformly $\wr$ $(t,x)$ ;

(iii) $f_i(t,x,0,0)$ is measurable and of polynomial growth ; 

(iv)  For any $U^1,U^2\in \cH^2(l^2)$, $X,Y\in \cS^2$, $i\in A$,  there
exist $V^{U^1,U^2,i}=(V^{U^1,U^2,i}_k)_{k\geq 1}$ (which may depend on $U^1$ and
$U^2$) that satisfy Assumption (A2) and such that :
\be\label{ineqxx} f_i(t,X_t,Y_t,U_t^1)-f_i(t,X_t,Y_t,U_t^2)\geq \langle
V^{U^1,U^2,i},(U^1-U^2)\rangle^p_t,\,\, d\cP\otimes dt-a.e.\ee

(v) For any $i\in A$ and $k\in A_i:=A-\{i\}$, the mapping $y_k\rightarrow f_i(t,x,y_1,\cdots,y_{k-1},y_k,y_{k+1},\cdots,y_m,\z)$ is nondecreasing whenever the other components $(t,x,y_1,\cdots,y_{k-1},y_{k+1},\cdots,y_m,\z)$ are fixed.\\

\nd (II) $\forall i,j\in A$, $g_{ii}\equiv 0$ and for $k\neq j$,
$g_{jk}(t,x)$ is non-negative, continuous with polynomial growth and
satisfy the following non free loop property:
\ms

For any $(t,x)\in
[0,T]\times \R$ and for any sequence of indices $i_1,\cdots,i_k$ such
that $i_1=i_k$ and $card\{i_1,\cdots,i_k\}=k-1$ we have$$g_{i_1
i_2}(t,x)+g_{i_2 i_3}(t,x)+\cdots+g_{i_k i_1}(t,x)>0.$$

\nd (III) $\forall i\in A$, $h_i$ is continuous with polynomial
growth and satisfies the following consistency condition:$$
h_i(x)\geq \max_{j\in A_i}(h_j(x)-g_{ij}(T,x)), \forall x\in
\R.$$

We now describe precisely the switching problem. Let $\Upsilon=((\theta_j)_{j\geq 0},(\alpha_j)_{j\geq 0})$ be an admissible strategy and let $a=(a_s)_{s\in [0,T]}$ be the process defined by
$$\forall s\leq T,\,\,a_s:=\alpha_0\mathbbm{1}_{\{\theta_0\}}(s)+\sum\limits^\infty_{j=1}\alpha_{j-1}\mathbbm{1}_{]\theta_{j-1}\theta_j]}(s),$$
where $\{\theta_j\}_{j\ge 0}$  is an increasing sequence of $\cF_t$-stopping
times with values in [0,T] and for $j\ge 0$, $\alpha_j$
is a random variable $F_{\theta_j}$-measurable with values in
$A=\{1,...,m\}$. If $\cP[\lim_{n}\th_n<T]=0$, then the pair
$\{\theta_j,\alpha_j\}_{j\geq 0}$ (or the process $a$) is called an admissible strategy
of switching. Next  we denote by $(A^a_s)_{s\leq T}$ the switching
cost process associated with an admissible strategy $a$, which is defined as
following:
\be\label{def_A}\forall s<T,\,\, A^a_s=\sum_{j\geq
1}g_{\alpha_{j-1},\alpha_j}(\theta_j,\xt_{\theta_j})\mathbbm{1}_{[{\theta_j\leq
s}]}\mbox{ and }  A^a_T=\lim_{s\rw T}A^a_s \ee where $X^{t,x}$ is the process given in
(\ref{SDE}). For $\eta \leq T$ and $i\in A$, we denote by $${\cal
A}^i_\eta:=\{a
\text{ admissible strategy such that }~ \alpha_0=i,\, \theta_0=\eta
\mbox{ and } E[(A^a_T)^2]<\infty\}.$$ Assume momentarily that for $i\in A$, the function $f_i$ of (\ref{introfct}) does not depend on $\vy$ and $\z$. For $t\leq T$ and a given admissible strategy
$a\in \cA^i_t$, we define the payoff $J_i^a(t,x)$ by:
\begin{align*}
J_i^{a}(t,x):=\E[\int_t^Tf_{a(s)}(s,\xt_s)ds+h_{a(T)}(\xt_T)-A^a_T]
\end{align*}where $f_{a(s)}(\dots)=f_{k}(\dots)$ (resp. $h_{a(T)}(.)=h_k(.)$) if at time $s$ (resp. $T$) $a(s)=k$ (resp. $a(T)=k$) ($k\in A$).
Finally let us define
\be \label{eqcout}J^i(t,x):=\sup_{a\in \cA^i_t}J_i^{a}(t,x),\,\,i=1,...,m.\ee
As a by-product of our main result which is given in Theorem \ref{exuncr} below, we get that
the functions $(J^i(t,x))_{i=1,\dots,m}$ is the unique continuous
viscosity solution of the Hamilton-Jacobi-Bellman system associated with this switching
problem (see Corollary \ref{corhjb}). \qed \ms

Let $(t,x)\in \esp$ and let us consider the following system of
reflected BSDEs with oblique reflection: $\forall j=1,...,m$ \be
\label{SRBSDE} \left\{
    \begin{array}{l}Y^{j}\in \cS^2,~~U^{j}\in \cH^2(\ell^2),~~K^{j} \in \cA^2;\\
      Y^{j}_s=h_j(X^{t,x}_T)+\int^T_s f_j(r,X^{t,x}_r,Y^{1}_r,Y^{2}_r,\cdots ,Y^{m}_r,U^{j}_r)dr
      -\sum\limits^\infty_{i=1}\int^T_s U^{j,i}_r dH^{(i)}_r +K^{j}_T-K^{j}_s, \,\,\forall s\leq T; \\
     \forall s\leq T,~~Y^{j}_s\geq \max\limits_{k\in Aj} \{Y^{k}_s-g_{jk}(s,X^{t,x}_s)\}\mbox{ and }
     \int^T_0\{Y^{j}_s-\max\limits_{k\in A_j}\{Y^{k}_s-g_{jk}(s,X^{t,x}_s)\}\}dK^{j}_s=0.\end{array}
    \right.\ee
Note that the solution of this BSDE depends actually on $(t,x)$ which we will omit for sake of simplicity, as far as there is no confusion. We then have the following result related to existence and uniqueness of the solution of (\ref{SRBSDE}).
\begin{thm} \label{maintheoremsystem} Assume that Assumption (A4)(I)(ii)-(iv), (A4)(II) and (A4)(III) are fulfilled. Then system of reflected BSDE with oblique reflection (\ref{SRBSDE}) has a unique solution.\end{thm}

\begin{proof} The proof follows the same lines as in \cite{chassagneux} and \cite{saidmorlais}. It will be given in two steps.
\ms

\nd \underline{Step 1}: We will first assume that the functions $f_i$, $i\in A$, verify (A4)(I)(ii)-(v). The other assumptions remain fixed.
\ms

Let us introduce the following standard BSDEs : 
\be \label{ini1}\left\{
\begin{array}{l}
\bar{Y}\in \cS^2,~~\bar U\in \cH^2(\ell^2);\\
\bar{Y}_s=\max\limits_{j=1,m}h_j(X^{t,x}_T)+\int^T_s \max\limits_{j=1,m}f_j(r,X^{t,x}_r,\bar{Y}_r,\cdots,\bar{Y}_r,\bar{U}_r)dr-\sum\limits^\infty_{i=1}\int^T_s \bar{U}^{i}_r dH^{(i)}_r,\,\,\forall s\leq T,\end{array}\right.\ee
and
\be \label{ini2}\left\{
\begin{array}{l}
\underbar{Y}\in \cS^2,~~\underbar U\in \cH^2(\ell^2);\\
\underbar{Y}_s=\min\limits_{j=1,m}h_j(X^{t,x}_T)+\int^T_s \min\limits_{j=1,m}f_j(r,X^{t,x}_r,\underbar{Y}_r,\cdots,\underbar{Y}_r,\underbar{U}_r)dr-\sum\limits^\infty_{i=1}\int^T_s \underbar{U}^{i}_r dH^{(i)}_r,\,\,\forall s\leq T.\end{array}\right.\ee
Note that thanks to Theorem 1 in \cite{nualart}, each one of the above BSDEs has a unique solution. Next for $j\in A$ and $n\geq 1$, let us define
$(Y^{j,n},U^{j,n},K^{j,n})$ by: \be \label{SQ}\left\{
    \begin{array}{ll}
    Y^{j,n}\in \cS^2,~~U^{j,n}\in \cH^2(\ell^2),~~K^{j,n} \in
    \cA^2\,;\\
    Y^{j,0}=\underbar{Y}\\
      Y^{j,n}_s=h_j(X^{t,x}_T)+\int^T_s f_j(r,X^{t,x}_r,Y^{1,n-1}_r,\cdots,Y^{j-1,n-1}_r,Y^{j,n}_r,Y^{j+1,n-1}_r,\cdots,Y^{m,n-1},U^{j,n}_r)dr\\\qq\qq\qq-\sum\limits^\infty_{i=1}\int^T_s U^{j,n,i}_r dH^{(i)}_r
       +K^{j,n}_T-K^{j,n}_s,\q \forall s\leq T;\\
     Y^{j,n}_s\geq \max\limits_{k\in A_j} (Y^{k,n-1}_r-g_{jk}(r,X^{t,x}_r)),~\forall s\leq T;~~\int^T_0 [Y^{j,n}_r-\max\limits_{k\in A_j}(Y^{k,n-1}_r-g_{jk}(r,X^{t,x}_r))]dK^{j,n}_r=0
     .
    \end{array}
    \right.\ee
By induction we can show that system (\ref{SQ}) has a unique
solution for any fixed $n\geq 1$ since when $n$ is fixed, (\ref{SQ}) reduces to $m$ decoupled reflected BSDEs of the form (\ref{RBSDE}). On the other hand it is  easily seen that $(\bar{Y}, \bar{U}, 0)$ is also a solution of :$$ \left\{
    \begin{array}{ll}
      \bar{Y}_s=\max\limits_{j=1,m}h_j(X^{t,x}_T)+\int^T_s \max\limits_{j=1,m}f_j(r,X^{t,x}_r,\bar{Y}_r,\cdots,\bar{Y}_r,\bar{U}_r)dr-\sum\limits^\infty_{i=1}\int^T_s \bar{U}^{i}_r dH^{(i)}_r+\bar{K}_T-\bar{K}_s,\forall s\leq T;\\
     \bar{Y}_s\geq \max\limits_{k\in A_j} (\bar{Y}_s-g_{jk}(s,X^{t,x}_s)),~\forall s\leq T;~~\int^T_0 [\bar{Y}_r-\max\limits_{k\in A_j}
     (\bar{Y}_s-g_{jk}(s,X^{t,x}_s))]d\bar{K}_r=0.
    \end{array}
    \right.$$
Next since for any $i\in A$, $f_i$ verifies Assumption A4(I)(ii)-(v), by Proposition \ref{comparison} and an induction argument, we get that $\cP$-a.s. for any $j,n$ and $s\leq T$, $Y^{j,n-1}_s\leq
Y^{j,n}_s\leq \bar{Y}_s.$ Then the sequence $(Y^{j,n})_{n\geq 0}$, has a limit which we denote by  $Y^j$, for any $j\in A$. By the monotonic limit theorem in \cite{fan},
$Y^j\in \cS^2$ and there exist $U^j\in \cH^2(\ell^2)$ and $K^j$ a non-decreasing process of $\cS^2$ such that: $\forall s\leq T$, \be
\label{eqlimcor}\left\{
    \begin{array}{ll}
      Y^{j}_s=h_j(X^{t,x}_T)+\int^T_s f_j(r,X^{t,x}_r,\overrightarrow{Y_r},U^{j}_r)dr-\sum\limits^\infty_{i=1}\int^T_s U^{j,i}_r dH^{(i)}_r +K^{j}_T-K^{j}_s,\\
     Y^{j}_s\geq \max\limits_{k\in A_j}
     (Y^{k}_{s}-g_{jk}(s,X^{t,x}_{s})),
    \end{array}
    \right.\ee
where for any $j\in A$, $U^j$ is the weak limit of $(U^{j,n})_{n\geq 1}$ in $\cH^2(\ell^2)$ and for any stopping time $\tau$, $K^j_\tau$ is the weak limit of $K^{j,n}_\tau$ in $L^2(\Omega, {\cal F}_\tau,\cP)$. Finally note that $K^j$ is predictable since the processes $K^{n,j}$ are so, for any $n\geq 1$.

Let us now consider the following RBSDE: \be \label{eqsmall}\left\{
    \begin{array}{ll}\hat{Y}^{j}\in \cS^2,~~\hat{U}^{j}\in \cH^2(\ell^2),~~\hat{K}^{j} \in
    \cS^2, \mbox { non-decreasing and }\hat{K}^{j}_0=0 ;\\
      \hat{Y}^{j}_s=h_j(X^{t,x}_T)+\int^T_s f_j(r,X^{t,x}_r,Y^1_r,\cdots,{Y}^{j-1}_r,\hat{Y}^{j}_r,{Y}^{j+1}_r,\cdots,Y^m_r,\hat{U}^{j}_r)dr-\sum\limits^\infty_{i=1}\int^T_s \hat{U}^{j,i}_r dH^{(i)}_r +\hat{K}^{j}_T-\hat{K}^{j}_s, \forall s\leq T;\\
     \hat{Y}^{j}_s\geq \max\limits_{k\in A_j} (Y^{k}_{s}-g_{jk}(s,X^{t,x}_{s})),~\forall s\leq T; ~~\int^T_0[\hat{Y}^j_{r-}-\max\limits_{k\in A_j} (Y^{k}_{r-}-g_{jk}(r,X^{t,x}_{r-}))]d\hat{K}^j_r= 0.
    \end{array}
    \right.\ee
    According to  Theorem 3.3 in  \cite{Aman}, this equation has a unique solution.
By Tanaka-Meyer's formula (see e.g.\cite{philip}, Theorem 68,
pp. 216), for all $j\in A$:
\begin{align*}(\hat{Y}^j_T-Y^j_T)^+=&(\hat{Y}^j_s-Y^j_s)^+ +\int^T_s
\mathbbm{1}_{\{\hat{Y}^j_{r-}-Y^j_{r-}>0\}}d(\hat{Y}^j_r-Y^j_r)\\
&+\sum\limits_{s<r\leq
T}[\mathbbm{1}_{\{\hat{Y}^j_{r-}-Y^j_{r-}>0\}}(\hat{Y}^j_r-Y^j_r)^-
+\mathbbm{1}_{\{\hat{Y}^j_{r-}-Y^j_{r-}\leq0\}}(\hat{Y}^j_r-Y^j_r)^+]+\frac{1}{2}L^0_t(\hat{Y}^j-Y^j)\end{align*}
where the process $(L^0_t(\hat{Y}^j-Y^j))_{t\leq T}$ is the local
time of the semi martingale $(\hat{Y}^j_s-Y^j_s)_{0\leq s\leq T}$ at
0 which is a nonnegative process. Then we have
\begin{align*}
(\hat{Y}^j_T-Y^j_T)^+\geq &(\hat{Y}^j_s-Y^j_s)^+ +\int^T_s
\mathbbm{1}_{\{\hat{Y}^j_{r-}-Y^j_{r-}>0\}}d(\hat{Y}^j_r-Y^j_r)\\
=&(\hat{Y}^j_s-Y^j_s)^+ -\int^T_s
\mathbbm{1}_{\{\hat{Y}^j_{r-}-Y^j_{r-}>0\}}[f_j(r,X^{t,x}_r,Y^1_r,\cdots,Y^{j-1}_r,\hat{Y}^{j}_r,Y^{j+1}_r\cdots,Y^m_r,\hat{U}^j_r)\\
&-f_j(r,X^{t,x}_r,Y^1_r,\cdots,{Y}^{j}_r,\cdots,Y^m_r,{U}^j_r)]dr-\int^T_s
\mathbbm{1}_{\{\hat{Y}^j_{r-}-Y^j_{r-}>0\}}d(\hat{K}^j_r-K^j_r)\\
&+\sum\limits^\infty_{i=1}\int^T_s\mathbbm{1}_{\{\hat{Y}^j_{r-}-Y^j_{r-}>0\}}
(\hat{U}^{j,i}_r-U^{j,i}_r)dH^{(i)}_r.
\end{align*}
First note that by $(\ref{eqsmall})$, $\int^T_s
\mathbbm{1}_{\{\hat{Y}^j_{r-}-Y^j_{r-}>0\}}d(\hat{K}^j_r-K^j_r)\leq
0$. Now by Assumption (A4)(I)(iv), we obtain:
\begin{align*}
(\hat{Y}^j_s-Y^j_s)^+&\leq \int^T_s
\mathbbm{1}_{\{\hat{Y}^j_{r-}-Y^j_{r-}>0\}}[f_j(r,X^{t,x}_r,Y^1_r,\cdots,\hat{Y}^{j}_r,\cdots,Y^m_r,\hat{U}^j_r)-f_j(r,X^{t,x}_r,Y^1_r,\cdots,{Y}^{j}_r,\cdots,Y^m_r,\hat{U}^j_r)\\
&+f_j(r,X^{t,x}_r,Y^1_r,\cdots,{Y}^{j}_r,\cdots,Y^m_r,\hat{U}^j_r)-f_j(r,X^{t,x}_r,Y^1_r,\cdots,{Y}^{j}_r,\cdots,Y^m_r,{U}^j_r)]dr\\
&-\sum\limits^\infty_{i=1}\int^T_s\mathbbm{1}_{\{\hat{Y}^j_{r-}-Y^j_{r-}>0\}}
(\hat{U}^{j,i}_r-U^{j,i}_r)dH^{(i)}_r\\
&\leq \int^T_s
\mathbbm{1}_{\{\hat{Y}^j_{r-}-Y^j_{r-}>0\}}C(\hat{Y}^j_{r-}-Y^j_{r-})^+dr+
\sum\limits_{i=1}^{\infty}\int^T_s
\mathbbm{1}_{\{\hat{Y}^j_{r-}-Y^j_{r-}>0\}}V^{U^j,\hat{U}^j,j}_i(r)(\hat{U}^{j,i}_r-U^{j,i}_r)dr\\
&\qq -\sum\limits^\infty_{i=1}\int^T_s\mathbbm{1}_{\{\hat{Y}^j_{r-}-Y^j_{r-}>0\}}
(\hat{U}^{j,i}_r-U^{j,i}_r)dH^{(i)}_r.
\end{align*}
Next for $t\leq T$, let us set  $M_t=\sum\limits_{i=1}^\infty \int^t_0
V^{U^j,\hat U^j,j}_i(r) dH^{(i)}_r$ and
$Z_t=\sum\limits_{i=1}^{\infty}\int^t_0
\mathbbm{1}_{\{\hat{Y}^j_{r-}-Y^j_{r-}>0\}}(\hat{U}^{j,i}_r-U^{j,i}_r)dH^{(i)}_r$ ($M$ and $Z$ depend on $j$ but this is irrelevant). By Proposition \ref{prop31},  $\varepsilon(M)\in\cS^2 $, $\varepsilon(M)>0$ and $\E[\varepsilon(M)_T]=1$. Then using Girsanov's Theorem (\cite{philip},
pp.136), under the probability measure $d\tilde{\cP}:=\varepsilon(M)_T
d\cP$, we obtain that the process $(\tilde{Z}_t=Z_t-<M,Z>_t)_{t\leq T}$ is a martingale and then $$\E_{\tilde{\cP}}[\sum\limits_{i=1}^{\infty}\int^T_s
\mathbbm{1}_{\{\hat{Y}^j_{r-}-Y^j_{r-}>0\}}V^{U^j,\hat U^j,j}_i(r)(\hat{U}^{j,i}_r-U^{j,i}_r)dr
-\sum\limits^\infty_{i=1}\int^T_s\mathbbm{1}_{\{\hat{Y}^j_{r-}-Y^j_{r-}>0\}}
(\hat{U}^{j,i}_r-U^{j,i}_r)dH^{(i)}_r]=-\E_{\tilde{\cP}}(\tilde{Z}_T-\tilde{Z}_s)=0.$$
Thus for any $s\le T$, $$\brl \E_{\tilde{\cP}}(\hat{Y}^j_s-Y^j_s)^+\leq
\E_{\tilde{\cP}}[\int^T_s C(\hat{Y}^j_{r}-Y^j_{r})^+dr]\ea$$ and finally
by Gronwall's Lemma, $\forall j\in A$, $\forall s\leq T$,
$(\hat{Y}^j_s-Y^j_s)^+=0~~\tilde{\cP}-a.s.$ and then also $\cP-a.s.$
since those probabilities are equivalent. It implies that $\cP$-a.s.,
$\hat{Y}^j\leq Y^j$ for any $j\in A$. On the other hand, since
$\forall n\geq 1$, $\forall j\in A$, $Y^{j,n-1}\leq Y^j$, then we have
$$\forall s\leq T,\,\,\max\limits_{k\in A_j} (Y^{k,n-1}_{s}-g_{jk}(s,X^{t,x}_{s}))\leq
\max\limits_{k\in A_j} (Y^{k}_{s}-g_{jk}(s,X^{t,x}_{s})).$$
Therefore by comparison, we obtain $Y^{j,n}\leq \hat{Y}^j$, and then $Y^j\leq \hat{Y}^j$ which implies $Y^j= \hat{Y}^j$, $\forall j\in A$.\\

Next by It\^o's formula applied to $(Y^j-\hat{Y}^j)^2$ we obtain: $\forall s\in [0,T]$,
$$(Y^j_s-\hat{Y}^j_s)^2=(Y^j_0-\hat{Y}^j_0)^2+2\int^s_0(Y^j_{r-}-\hat{Y}^j_{r-})d(Y^j_r-\hat{Y}^j_r)+\sum^\infty_{i=1}\sum^\infty_{k=1}\int^s_0(U^{j,i}_r-\hat{U}^{j,i}_r)(U^{j,k}_r-\hat{U}^{j,k}_r)d[H^{(i)},H^{(k)}]_r.$$
As $Y^j= \hat{Y}^j$ and taking expectation in both-hand sides of the previous equality to obtain
$$\bal
\E[\int_0^T\sum_{i\geq 1}(U^{j,i}_r-\hat{U}^{j,i}_r)^2dr]=0.\ea$$
It implies that $U^j=\hat U^j$, $dt\otimes d\cP$ and finally $K^j=\hat K^j$ for any $j\in A$, i.e. $(Y^j,U^j,K^j)_{j\in A}$ verify (\ref{eqsmall}).\\

Next we will  show that the predictable process $K^j$ does not have
jumps. First note that since $K^j$ is predictable then its jumping times are also predictable. So assume there exist $j_1\in A$ and a predictable stopping
time $\tau$ such that  $\dt
K^{j_1}_\tau=\dt \hat{K}^{j_1}_\tau>0$. As $Y^j$ verifies (\ref{eqsmall}) and since the martingale part in this latter equation has only inaccessible jump times then 
$\dt Y^{j_1}_{\tau}=-\dt
K^{j_1}_\tau=-\dt \hat{K}^{j_1}_\tau<0$. By the
second equality in (\ref{eqsmall}) we have \be
\label{eqbord1}{Y}^{j_1}_{\tau-}=\max\limits_{k\in A_{j_1}}
(Y^{k}_{\tau-}-g_{{j_1}k}(\tau,X^{t,x}_{\tau-})).\ee Now let $j_2\in
A_{j_1}$ be the optimal index in (\ref{eqbord1}), i.e.,
  $$Y^{j_2}_{\tau-}-g_{j_1,j_2}(\tau,X^{t,x}_{\tau})=Y^{j_1}_{\tau-}> {Y}^{j_1}_\tau\geq Y^{j_2}_{\tau}-g_{j_1,j_2}(\tau,X^{t,x}_{\tau}).$$
       Note that
       $g_{j_1,j_2}(\tau,X^{t,x}_{\tau-})=g_{j_1,j_2}(\tau,X^{t,x}_{\tau})$
       since the stopping time $\tau$ is predictable, and the
       process $(X^{t,x}_s)_{t\leq s\leq T}$ does not have
       predictable jump times.
Thus $\dt Y^{j_2}_\tau<0$ and once more we have, \be
\label{eqbord2}{Y}^{j_2}_{\tau-}=\max\limits_{k\in A_{j_2}}
(Y^{k}_{\tau-}-g_{{j_2}k}(\tau,X^{t,x}_{\tau-})).\ee We can now
repeat the same argument as many times as necessary, to deduce the existence of a loop $\ell_1, ...,\ell_{p-1},
\ell_p=\ell_1$ ($p \ge 2$) and $\ell_2\neq \ell_1$ such that
$$Y^{\ell_1}_{\tau-}=Y^{\ell_2}_{\tau-}-g_{\ell_1\ell_2}(\tau, X^{t,x}_{\tau-}),\cdots ,Y^{\ell_{p-1}}_{\tau-}=Y^{\ell_p}_{\tau-}-g_{\ell_{p-1}\ell_p}(\tau, X^{t,x}_{\tau-})$$ which implies that
$$
g_{\ell_1\ell_2}(\tau,
X^{t,x}_{\tau-})+\dots+g_{\ell_{p-1}\ell_p}(\tau,
X^{t,x}_{\tau-})=0
$$which is contradictory  with Assumption (A4)(II). It implies that $\Delta K^{j_1}_\tau=0$ and then $K^{j_1}$ is continuous since it is predictable. As $j$ is arbitrary in $A$, then the processes $K^j$  are continuous and taking into account (\ref{eqsmall}), we deduce that the triples $(Y^j,U^j,K^j)_{j\in A}$, is a solution for system (\ref{SRBSDE}). \qed
\ms

\nd \underline{Step 2}:
We now deal with the general case i.e. we assume that $f_i$, $i\in A$, do no longer satisfy the monotonicity assumption (A4)(I)(v) but
(A4)(I)(ii)-(iv) solely.
\ms

Let $j\in A$ and $t_0\in [0,T]$ be fixed. We should  stress here that we do not need to take $t_0=t$ since the result is valid for general stochastic process and not only of Markovian type as $X^{t,x}$. For $a\in \cA^j_{t_0}$ and  $\Gamma:=((\Gamma^l_s)_{s\in [0,T]})_{l\in A}\in [H^2]^m:=H^2\times \dots \times H^2$ ($m$ times), we introduce the unique
solution of the switched BSDE which is defined by: $\forall s\in [t_0,T]$,
\be\label{bsderep}
V^{a}_s=h_{a(T)}(X^{t,x}_T)+\int^T_s f_{a(r)}(r,X^{t,x}_r,\overrightarrow{\Gamma_r},N^a_r)dr-\sum\limits^\infty_{i=1}\int^T_s
N^{a,i}_rdH^{(i)}_r-A^a_T+A^a_s
\ee
where $V^{a}\in \cS^2$ and $N^a\in \cH^2(\ell^2)$ ($\overrightarrow{\Gamma_r}=(\Gamma^i_r)_{i\in A}$). First note that the solution of this equation exists and is unique since in setting, for $s\in [t_0,T]$, $\tilde V^{a}_s=V^{a}_s-A^a_s$ and $\tilde h^a_T=h_{a(T)}(X^{t,x}_T)-A^a_T$ this equation becomes standard and has a unique solution by Nualart et al.'s result (see \cite{nualart}, Theorem 1, pp.765). Moreover (see Appendix, Proposition \ref{repres}) we have the following link between the BSDEs (\ref{SRBSDE}) and (\ref{bsderep}),   \be\label{lienentreyetv}Y^j_{t_0}=\ess_{a\in {\cal A}^j_{t_0}} (V^a_{t_0}-A^a_{t_0})=V^{a^*}_{t_0}-A^{a^*}_{t_0}\ee
for some $a^*\in \cA^j_{t_0}$. Next let us introduce the following mapping $\Theta$ defined on $[H^2]^m$ by \be\label{defteta}\begin{array}{l}
\Theta:[H^2]^m\rightarrow [H^2]^m\\
\Gamma=(\Gamma^j)_{j\in A}\mapsto (Y^j)_{j\in A}\end{array}\ee
where $(Y^j,U^j,K^j)_{j\in A}$ is the unique solution of the following system of RBSDEs:\be\label{theta} \left\{
    \begin{array}{ll}
      Y^j_s=h_j(X^{t,x}_T)+\int^T_s f_j(r,X^{t,x}_r,\overrightarrow{\Gamma_r},U^j_r)dr-\sum\limits^\infty_{i=1}\int^T_s U^{j,i}_r dH^{(i)}_r +K^j_T-K^j_s,~\forall s\leq T;\\
      Y^j_s\geq \max\limits_{k\in A_j} \{Y^k_s-g_{jk}(s,X^{t,x}_s)\},~\forall s\leq T;~~\int^T_0[Y^j_s-\max\limits_{k\in A_j}\{Y^k_s-g_{jk}(s,X^{t,x}_s)\}]dK^j_s=0.
    \end{array}
    \right.\ee
By the result proved in Step 1, $\Theta$ is well-defined. Next for $\eta\in H^2$ let us
define $\|\cdot\|_{2,\beta}$ by$$\bal\|
\eta\|_{2,\beta}:=(\E[\int^T_0 e^{\beta
s}|\eta_s|^2ds])^{\frac{1}{2}},\ea$$ which is a norm of $H^2$, equivalent to $\|.\|$ and $(H^2,\|\cdot\|_{2,\beta})$ is a Banach space. Let now
$\Gamma^1$ and $\Gamma^2$ be two processes of $[H^2]^m$ and for $k=1,2$, let
$(Y^{k,j},U^{k,j},K^{k,j})_{j\in A}=\Theta(\Gamma^k)$, i.e., that satisfy: $\forall s\leq T$,
$$\left\{
    \begin{array}{ll}
      Y^{k,j}_s=h_j(X^{t,x}_T)+\int^T_s f_j(r,X^{t,x}_r,\overrightarrow{\Gamma_r^k},U^{k,j}_r)dr-
      \sum\limits^\infty_{i=1}\int^T_s U^{k,j,i}_r dH^{(i)}_r +K^{k,j}_T-K^{k,j}_s;\\
      Y^{k,j}_s\geq \max\limits_{q\in A_j} \{Y^{k,q}_s-g_{jq}(s,\xt_s)\};~~\int^T_0[Y^{k,j}_s-\max\limits_{q\in
      A_j}\{Y^{k,q}_s-g_{jq}(s,\xt_s)\}]dK^{k,j}_s=0.
    \end{array}
    \right.$$
Next let us define $(\hat{Y}^j)_{j\in A}$ through the following system of reflected BSDEs with oblique reflection: $\forall s\leq T$,
$$\left\{
    \begin{array}{ll}
      \hat{Y}^j_s=h_j(X^{t,x}_T)+\int^T_s f_j(r,X^{t,x}_r,\overrightarrow{\Gamma_r^1},\hat{U}^j_r)\vee f_j(r,X^{t,x}_r,\overrightarrow{\Gamma_r^2},\hat{U}^j_r)dr-\sum\limits^\infty_{i=1}\int^T_s \hat{U}^{j,i}_r dH^{(i)}_r +\hat{K}^j_T-\hat{K}^j_s\\
      \hat{Y}^j_s\geq \max\limits_{q\in A_j} \{\hat{Y}^q_s-g_{jq}(s,X^{t,x}_s)\};~~\int^T_0[\hat{Y}^j_s-\max\limits_{q\in
      A_j}\{\hat{Y}^q_s-g_{jq}(s,X^{t,x}_s)\}]d\hat{K}^j_s=0.
    \end{array}
    \right.$$ Recall once more that $a\in \cA^j_{t_0}$ and let us define $V^{k,a}$, $k=1,2$, and $\hat V^a$, via BSDEs,  by
    $$\bal
\hat{V}^a_s=h_{a(T)}(X^{t,x}_T)+\int^T_s f_{a(r)}(r,X^{t,x}_r,\overrightarrow{\Gamma_r^1},\hat{N}^a_r)\vee f_{a(r)}(r,X^{t,x}_r,\overrightarrow{\Gamma_r^2},\hat{N}^a_r)dr-\sum\limits^\infty_{i=1}\int^T_s\hat{N}^{a,i}_rdH^{(i)}_r
-A^a_T+A^a_s,\,\, s\le T, 
\ea$$
and for $k=1,2$, $$\bal V^{k,a}_s=h_{a(T)}(X^{t,x}_T)+\int^T_s
f_{a(r)}(r,X^{t,x}_r,\overrightarrow{\Gamma_r^k},N^{k,a}_r)dr-A^a_T+A^a_s-\sum\limits^\infty_{i=1}\int^T_s
N^{k,a,i}_rdH^{(i)}_r,\,\,s\le T.\ea$$ By Proposition \ref{app1} in Appendix, we have:
\be\label{3.8} Y^{k,j}_{t_0}=\ess_{a\in{\cal A}^j_{t_0}}(
V^{k,a}_{t_0}-A^a_{t_0}),~k=1,2 \mbox{ and } \hat{Y}^{j}_{t_0}=\ess_{a\in{\cal
A}^j_{t_0}} (\hat{V}^{a}_{t_0}-A^a_{t_0}):=\hat{V}^{{a}^*}_{t_0}-A^{a^*}_{t_0}. \ee

 In addition for $s\in [t_0,T]$, $f_{a(s)}$ verifies the inequality (\ref{ineqxx}) of Assumption (A4)(I)(iv). Actually let us set $a_s=\alpha_0\mathbbm{1}_{\{\theta_0\}}(s)+\sum\limits^\infty_{j=1}\alpha_{j-1}\mathbbm{1}_{]\theta_{j-1}\theta_j]}(s)$, $s\in [t_0,T]$, and let $U^1,U^2\in \cH^2(l^2)$, $X,Y\in \cS^2$. For any
 $s\in [t_0,T]$ we have:$$
\begin{array}{ll}
&f_{a(s)}(s,X_s,Y_s,U_s^1)-f_{a(s)}(s,X_s,Y_s,U_s^2)\\
&=[f_{\a_0}(s,X_s,Y_s,U_s^1)-f_{\a_0}(s,X_s,Y_s,U_s^2)]\mathbbm{1}_{\{\theta_0\le s\le \theta_1\}}+
\sum\limits_{j\ge 2}[f_{\a_{j-1}}(s,X_s,Y_s,U_s^1)-f_{\a_{j-1}}(s,X_s,Y_s,U_s^2)]\mathbbm{1}_{]\theta_{j-1},\theta_j]}(s)\\
&\geq \langle
V^{U^1,U^2,\a_0},(U^1-U^2)\rangle^p_s\mathbbm{1}_{\{\theta_0\le s\le \theta_1\}}+\sum\limits_{j\ge 2}\langle
V^{U^1,U^2,\a_{j-1}},(U^1-U^2)\rangle^p_s\mathbbm{1}_{]\theta_{j-1},\theta_j]}(s)=:\langle V^{U^1,U^2,a},(U^1-U^2)\rangle^p_s.
\end{array}$$
where for any $s\in [t_0,T]$, $$V^{U^1,U^2,a}_s:=(V^{U^1,U^2,a,i}_s)_{i\ge 1}=V^{U^1,U^2,\a_0}_s\mathbbm{1}_{\{\theta_0\le s\le \theta_1\}}+\sum\limits_{j\ge 2}
V^{U^1,U^2,\a_{j-1}}_s\mathbbm{1}_{]\theta_{j-1},\theta_j]}(s).$$
But on $[t_0,T]\times \Omega$,
$$\bal \cP\{\omega, \exists s\leq T, \mbox{ such that }\sum\limits_{i=1}^\infty
V^{U^1,U^2,a,i}_s(\omega)p_i(\Delta
L_s(\omega))\leq -1\}\\\qq\leq \sum_{j\in A}
\cP\{\omega, \exists s\leq T, \mbox{ such that }\sum\limits_{i=1}^\infty
V^{U^1,U^2,j,i}_s(\omega)p_i(\Delta L_s(\omega))\leq
-1\}=0\ea $$
which implies that $$\cP-a.s.,\,\forall s\in [t_0,T],\,\, 
\sum\limits_{i=1}^\infty
V^{U^1,U^2,a,i}_s(\omega)p_i(\Delta
L_s(\omega))>-1.
$$
On the other hand, on $[t_0,T]\times \Omega$,
$$\sum\limits^\infty_{i=1}|V^{U^1,U^2,a,i}_s|^2\leq \sum_{\ell\in A}\sum\limits^\infty_{i=1}|V^{U^1,U^2,\ell,i}_s|^2\leq C, \,\,ds\otimes d\cP-a.e.$$ 
Thus the process $V^{U^1,U^2,a}$ verifies Assumption (A2) and $f_{a(s)}$ satisfies Assumption (A4)(I)(iv) on $[t_0,T]$.\\

Consequently, by the comparison result of Proposition \ref{comp}, for any strategy
$a\in{\cal A}^j_{t_0}$, $\cP$-a.s. for any $s\in [t_0,T]$, $\hat{V}^a_s\geq V^{1,a}_s\vee V^{2,a}_s$. This combined with (\ref{3.8}) leads to  $
Y^{1,j}_{t_0}\vee Y^{2,j}_{t_0}\leq \hat{Y}^j_{t_0} =\hat{V}^{{a}^*}_{t_0}-A^{{a}^*}_{t_0}$. We then deduce $$V^{1,{a}^*}_{t_0}-A^{{a}^*}_{t_0}\leq Y^{1,j}_{t_0}\leq \hat{V}^{{a}^*}_{t_0}-A^{{a}^*}_{t_0} \mbox{ and }V^{2,{a}^*}_{t_0}-A^{{a}^*}_{t_0}\leq Y^{2,j}_{t_0}\leq \hat{V}^{{a}^*}_{t_0}-A^{{a}^*}_{t_0}$$ which implies  \be\label{add1}|Y^{1,j}_{t_0}-Y^{2,j}_{t_0}|\leq
|\hat{V}^{a^*}_{t_0}-V^{1,a^*}_{t_0}|+|\hat{V}^{a^*}_{t_0}-V^{2,a^*}_{t_0}|.\ee

Next we first estimate the quantity $|\hat{V}^{a^*}_{t_0}-V^{1,a^*}_{t_0}|$. For $s\in [t_0,T]$ let us set $\Delta V^{a^*}_s:=\hat{V}_s^{a^*}-V^{1,a^*}_s$ and
$\Delta N^{a^*}_s:=\hat{N}_s^{a^*}-N^{1,a^*}_s$. Applying It\^o's Formula to the process $e^{\beta s}|\Delta
V^{a^*}_s|^2$ we obtain: $\forall s\in [t_0,T]$,
$$\begin{array}{ll}
&e^{\beta s}|\Delta V^{a^*}_s|^2+\int^T_s e^{\beta r}\Vert\Delta N^{a^*}_r\Vert^2dr\\\\
&\qq =-\int^T_s\beta e^{\beta r}|\Delta V^{a^*}_{r-}|^2 dr-2\sum\limits^\infty_{i=1}\int^T_s e^{\beta r}\Delta V^{a^*}_{r-} \Delta N^{a^*,i}_rd H^{(i)}_r\\\\
&\qq \qq +2\int^T_s e^{\beta r}\Delta V^{a^*}_{r-}[f_{a^*(r)}(r,X^{t,x}_r,\overrightarrow{\Gamma_r^1},\hat{N}^{a^*}_r)\vee f_{a^*(r)}(r,X^{t,x}_r,\overrightarrow{\Gamma_r^2},\hat{N}^{a^*}_r)-f_{a^*(r)}(r,X^{t,x}_r,\overrightarrow{\Gamma_r^1},\hat{N}^{1,a^*}_r)]dr\\\\
&\qq \qq -\sum\limits^\infty_{i=1}\sum\limits^\infty_{l=1}\int^T_s  e^{\beta
r}\Delta N^{a^*,i}_r \Delta
N^{a^*,l}_rd([H^{(i)},H^{(l)}]_r-\langle H^{(i)},H^{(l)}\rangle_r).
\end{array}$$
By the Lipschitz property of $f_j$, $j\in A$, and then of $f_{a*}$ and the fact that for any $x,y\in \R$, $|x\vee y-y|\leq
|x-y|$ we have: $\forall s\in [t_0,T]$,
\be \label{eqnxx1}\begin{array}{l}
|f_{a^*(r)}(r,X^{t,x}_r,\overrightarrow{\Gamma_r^1},\hat{N}^{a^*}_r)\vee f_{a^*(r)}(r,X^{t,x}_r,\overrightarrow{\Gamma_r^2},\hat{N}^{a^*}_r)-f_{a^*(r)}(r,X^{t,x}_r,\overrightarrow{\Gamma_r^1},\hat{N}^{1,a^*}_r)|\\\\\qq
\le |f_{a^*(r)}(r,X^{t,x}_r,\overrightarrow{\Gamma_r^1},\hat{N}^{a^*}_r)\vee f_{a^*(r)}(r,X^{t,x}_r,\overrightarrow{\Gamma_r^2},\hat{N}^{a^*}_r)-f_{a^*(r)}(r,X^{t,x}_r,\overrightarrow{\Gamma_r^1},\hat{N}^{a^*}_r)|\\\qq\qq+
|f_{a^*(r)}(r,X^{t,x}_r,\overrightarrow{\Gamma_r^1},\hat{N}^{a^*}_r)-f_{a^*(r)}(r,X^{t,x}_r,\overrightarrow{\Gamma_r^1}_r,\hat{N}^{1,a^*}_r)|\\\\\qq
\leq L(|\overrightarrow{\Gamma_r^1}-\overrightarrow{\Gamma_r^2}|+\|\hat{N}^{a^*}_r-N^{1,a^*}_r\|).\ea
\ee
The inequality $2xy\leq
\frac{1}{\beta} x^2+\beta y^2$ (for any $\beta>0$ and $x,y\in \R$) and (\ref{eqnxx1}) yield: $\forall s\in [t_0,T]$,$$
\begin{array}{ll}
e^{\beta s}|\Delta{V}^{a^*}_s|^2 &\leq -\int^T_s e^{\beta
r}\Vert\Delta N^{a^*}_r\Vert^2dr-\int^T_s\beta e^{\beta
r}|\dt V^{a^*}_{r-}|^2
ds-2\sum\limits^\infty_{i=1}\int^T_s e^{\beta r}\Delta
V^{a^*}_{r-} \Delta N^{a^*,i}_rd H^{(i)}_r\\
&\qq +2L\int^T_s e^{\beta r}|\Delta V^{a^*}_{r-}|(|\overrightarrow{\Gamma_r^1}-\overrightarrow{\Gamma_r^2}|+\|\Delta N^{a^*}_r\|)dr\\
&\qq -\sum\limits^\infty_{i=1}\sum\limits^\infty_{l=1}\int^T_s  e^{\beta
r}\Delta N^{a^*,i}_r \Delta
N^{a^*,l}_rd([H^{(i)},H^{(l)}]_r-\langle H^{(i)},H^{(l)}\rangle_r)\\\\
{}&\leq-\int^T_s e^{\beta r}\Vert\Delta
N^{a^*}_r\Vert^2dr-\int^T_s\beta e^{\beta r}|\Delta
V^{a^*}_{r-}|^2 ds-2\sum\limits^\infty_{i=1}\int^T_s e^{\beta
r}\Delta
V^{a^*}_{r-} \Delta N^{a^*,i}_rd H^{(i)}_r\\
&\qq +\int^T_s\beta e^{\beta r}|\dt V^{a^*}_{r-}|^2
ds+\frac{L^2}{\beta}\int^T_s e^{\beta r}(|\overrightarrow{\Gamma_r^1}-\overrightarrow{\Gamma_r^2}|+\|\Delta
N^{a^*}_{r}\|)^2dr\\
&\qq -\sum\limits^\infty_{i=1}\sum\limits^\infty_{l=1}\int^T_s  e^{\beta
r}\Delta N^{a^*,i}_r \Delta
N^{a^*,l}_rd([H^{(i)},H^{(l)}]_r-\langle H^{(i)},H^{(l)}\rangle_r)\\\\
{}&\leq \q \frac{2L^2}{\beta}\int^T_s e^{\beta r}|\overrightarrow{\Gamma_r^1}-\overrightarrow{\Gamma_r^2}|^2dr-2\sum^\infty_{i=1}\int^T_s e^{\beta r}\Delta V^{a^*}_{r-} N^{a^*,i}_r dH^{(i)}_r\\
&\qq -\sum^\infty_{i=1}\sum^\infty_{l=1}\int^T_s  e^{\beta
r}\Delta N^{a^*,i}_r \Delta
N^{a^*,l}_rd([H^{(i)},H^{(l)}]_r-\langle H^{(i)},H^{(l)}\rangle_r),
\end{array}$$
for $\beta \geq 2L^2$. We deduce, in taking expectation, $$\forall s\in [t_0,T],\,\, \E[e^{\beta s}|\Delta\hat{V}^{a^*}_s|^2]\leq
\frac{2L^2}{\beta}\E[\int^T_s e^{\beta r}|\overrightarrow{\Gamma_r^1}-\overrightarrow{\Gamma_r^2}|^2dr].$$
Similarly, we get also $\forall s\in [t_0,T]$,  $$\E[e^{\beta
s}|\hat{V}^{a^*}_s-V^{2,a^*}_s|^2]\leq \frac{2L^2}{\beta}\E[\int^T_s
e^{\beta r}|\overrightarrow{\Gamma_r^1}-\overrightarrow{\Gamma_r^2}|^2dr].$$
Therefore by (\ref{add1}) we obtain: \be\label{add12}\E[e^{\beta t_0}|Y^{1,j}_{t_0}-Y^{2,j}_{t_0}|^2]\leq
\frac{8L^2}{\beta}\|\Gamma^1-\Gamma^2\|^2_{2,\beta}.\ee
As $t_0$ is arbitrary in $[0,T]$ then by integration $\wr$ $t_0$ we get \be
\label{map}\Vert\Theta(\Gamma^1)-\Theta(\Gamma^2)\Vert_{2,\beta}\leq
\sqrt{\frac{8L^2Tm}{\beta}}\Vert\Gamma^1-\Gamma^2\Vert_{2,\beta}.\ee
Henceforth for $\beta$ large enough, $\Theta$ is contraction on the Banach
space $([H^2]^m,\Vert .\Vert_{2,\beta})$, then it has a fixed point $(Y^j)_{j\in A}$ which has a version which is the unique solution of system of RBSDE (\ref{SRBSDE}).
\end{proof}
\begin{rem}
As a consequence of (\ref{add12}), there exists a constant $C> 0$,
such that $\forall j\in A$, $s\leq T$,
\be\label{add5}\E[|Y^{1,j}_s-Y^{2,j}_s|^2]\leq
C\Vert (Y^{1,j})_{j\in A}-(Y^{2,j})_{j\in A}\Vert_{2,\beta}^2.\ee
This estimate will be useful later. \qed
\end{rem}
\begin{cor}\label{croipoly}
Under Assumptions (A4), there exist  deterministic lower semi-continuous functions $(u^j(t,x))_{j\in A}$ of polynomial growth such that
$$
\forall (t,x)\in \esp, \,\, \forall s\in [t,T], \,\,Y^j_s=u^j(s,X^{t,x}_s),\,\,\forall j\in A.$$
\end{cor}
\begin{proof} This is a direct consequence of the construction by induction of the solution $(Y^j,U^j,K^j)_{j\in A}$ given in Step 1. Actually by Ren et al.'s result \cite{renotmani}, there exist  deterministic continuous functions of polynomial growth $\bar u(t,x)$, $\underbar u(t,x)$ and $u^{j,n}(t,x)$, $n\geq 0$ and $j\in A$, such that $\forall (t,x)\in \esp, \,\, \forall s\in [t,T]$\\
(a) $$
\bar Y_s=\bar u(s,X^{t,x}_s) \mbox { and }\underbar Y_s=\underbar u(s,X^{t,x}_s).$$
(b)
$$
Y^{j,n}_s=u^{j,n}(s,X^{t,x}_s), \forall j\in A, $$
and
$$
\underbar Y\leq Y^{j,n}\le Y^{j,n+1}\le \bar Y.
$$
This yields, for any $n\geq 0$ and $(t,x)\in \esp$,
$$
\underbar u(t,x)\leq u^n(t,x)\leq  u^{n+1}(t,x)\leq \bar u(t,x).
$$Thus $u_j(t,x):=\lim_{n\rw \infty}u^{j,n}(t,x)$, $j\in A$, verify the required properties since $(Y^{j,n})_n$ converges to $Y^j$, $j\in A$, in $\cS^2$.
\end{proof}

We now give a comparison result for solutions of systems $(\ref{SRBSDE})$. The induction argument allows to compare the solution of the approximating schemes, by Proposition \ref{comp}, and then to deduce the same property for the limiting processes.\\

\begin{rem}\label{remarquecomparaison}Let $(\bar Y^{j},\bar U^{j},\bar K^{j})_{j\in A}$ be a solution of the system of RBSDEs $(\ref{SRBSDE})$ associated with\\
$((\bar f_j)_{j\in A},(\bar g_{jk})_{j,k\in A},(\bar h_j)_{j\in A})$ which
satisfy
 ${\bf (A4)}$. If for any
$j,k\in A$, $$f_j\leq \bar f_j,\, h_j\leq \bar h_j,\, g_{jk}\geq \bar g_{jk}$$ then
for any $j\in A$, $Y^j\leq \bar Y^{j}.\qed $
\end{rem}
\section{Existence and uniqueness of the solution for the system of IPDEs with inter-connected obstacles}
This section focuses on the main result of this paper which is the proof of existence and uniqueness of a solution for the system of IPDEs introduced in the begining of this paper (\ref{edpintro}). For this objective we use its link with the system of RBSDEs (\ref{SRBSDE}). However we are led to make, hereafter, the following additional assumption because, basically, the hypothesis (A4)-(iv) is either artificial in this deterministic setting or not easy to verify.
\bs

\nd {\bf  Assumption (A5)}: For any $i\in A$, $f_i$ does not
depend on the variable $\z \in \ell^2$. \qed
\bs

So we are going to consider the following system of IPDEs: $\forall
i\in A$, \be \label{FIPDE}\left\{
    \begin{array}{ll}
      min\{u_i(t,x)-\max\limits_{j\in A_i}(u_j(t,x)-g_{ij}(t,x));\\\qq\qq
      -\partial_t u_i(t,x)-{\cal L} u_i(t,x)-f_i(t,x,u_1(t,x),\cdots, u_m(t,x))\}=0, \,(t,x)\in \esp ;\\
      u_i(T,x)=h_i(x)
    \end{array}
    \right.\ee
where $$ {\cal L} u(t,x)= {\cal L}^1 u(t,x)+{\cal I}(t,x,u)$$ with
\be \label{operateur}\begin{array}{l}  {\cal L}^1 u(t,x):=(\E[L_1]\sigma(t,x)
+b(t,x))\partial_x
u(t,x)+\frac{1}{2}\sigma(t,x)^2\varpi^2
D^2_{xx}u(t,x)\mbox{ and }
\\\\{\cal I}(t,x,u):=
\int_\R[u(t,x+\sigma(t,x)y)-u(t,x)-\partial_xu(t,x)\sigma(t,x)y]\Pi(dy).\end{array}\ee
Note that for any $\phi \in \ccp$ and
$(t,x)\in \esp$, the non-local term
\be\label{defiint}\begin{array}{l}\I(t,x,\phi):=\int_\R[\phi(t,x+\sigma(t,x)y)-
\phi(t,x)-\partial_x \phi(t,x)\sigma(t,x)y]\Pi(dy)\end{array}\ee is
well-defined. Actually let $\delta>0$ and let us define, for any
$q\in \R$,
\be\label{defiint2}\begin{array}{l}\I^{1,\delta}(t,x,\phi):=\int_{|y|\leq
\delta}[\phi(t,x+\sigma(t,x)y)- \phi(t,x)-\partial_x
\phi(t,x)\sigma(t,x)y]\Pi(dy),\end{array}\ee
\be\label{defiint3}\begin{array}{l}\I^{2,\delta}(t,x,q,u):=\int_{|y|>\delta}[u(t,x+\sigma(t,x)y)-
u(t,x)-q\sigma(t,x)y]\Pi(dy).\end{array}\ee By Taylor's expansion we have $$\begin{array}{l}
\phi(t,x+\sigma(t,x)y)- \phi(t,x)-\partial_x
\phi(t,x)\sigma(t,x)y=\int_0^y\sigma(t,x)^2D^2_{xx}\phi(t,x+\sigma(t,x)r)(y-r)dr.\end{array}
$$
But there exists a constant $C_{tx}$ such that for any $|r|\leq
\delta$, $|D^2_{xx}\phi(t,x+\sigma(t,x)r)|\leq C_{tx}$ since $\phi$ belongs to $\cC^{1,2}$ and $\sigma$ is bounded. Therefore for $|y|\le \d$, 
$$
|\phi(t,x+\sigma(t,x)y)-
\phi(t,x)-\partial_x
\phi(t,x)\sigma(t,x)y|\leq C_{tx} |y|^2
$$
which implies that $\I^{1,\delta}(t,x,\phi)\in \R$. Next for any $(t,x)$,
$\I^{2,\delta}(t,x,D_x\phi(t,x),\phi)\in \R$ since $\Pi$ integrates any power function outside $[-\epsilon,\epsilon]$. Henceforth $\I(t,x,\phi)$ is well-defined. \qed
\ms

We are now going to give the
definition of a viscosity solution of (\ref{FIPDE}). First for a locally bounded function $u$: $(t,x)\in [0,T]\times \R\rightarrow u(t,x)\in \R$, we define its lower semi-continuous (lsc for short) envelope $u_*$ and upper semi-continuous (usc for short) envelope $u^*$ as following:$$u_*(t,x)=\varliminf\limits_{(t',x')\rightarrow (t,x),~t'<T}u(t',x'), ~~~ u^*(t,x)=\varlimsup\limits_{(t',x')\rightarrow (t,x),~t'<T}u(t',x')$$

\begin{defi}\label{defvisco} A  function $(u_1,\cdots,u_m):[0,T]\times \R\rightarrow \R^m$ which belongs to $\pl$ such that for any $i\in A$, $u_i$ is usc (resp. lsc), is said to be a viscosity subsolution  (resp. supersolution) of (\ref{FIPDE}) if  for any $i\in A$, $\varphi\in \ccp$, $u_i(T,x)\leq h_i(x)$ (resp. $u_i(T,x)\geq h_i(x)$) and if  $(t_0,x_0)\in (0,T)\times \R$ is a global maximum (resp. minimum) point of $u_i-\varphi$,
\begin{align*}
min&\Big \{u_i(t_0,x_0)-\max\limits_{j\in A_i}\{u_j(t_0,x_0)-g_{ij}(t_0,x_0)\}\,;\,-\partial_t\varphi(t_0,x_0)-{\cal L}\varphi(t_0,x_0)\\
&-f_i(t_0,x_0,u_1(t_0,x_0),\cdots,u_{i-1}(t_0,x_0),u_i(t_0,x_0),\cdots,u_m(t_0,x_0))\Big \}\leq
0~~(resp.\geq 0).
\end{align*}
The function $(u_i)^m_{i=1}$ is called a viscosity
solution of (\ref{FIPDE}) if $(u_{i*})^m_{i=1}$ and $({u_{i}^*})^m_{i=1}$ are respectively viscosity supersolution and subsolution of (\ref{FIPDE}).\end{defi}

The following result is needed later.

\begin{lem}\label{modifsursol} Let $(u_i)^m_{i=1}$ be a supersolution of (\ref{FIPDE}) which belongs to $\pl$, i.e. for some $\gamma >0$ and $C>0$,
$$
|u_i(t,x)|\leq C(1+|x|^ \gamma),\,\forall (t,x)\in \esp \mbox{ and }i\in A.
$$
Then there exists $\lambda_0>0$ such that for any $\lambda\geq\lambda_0$ and $\theta >0$, $\overrightarrow{v}(t,x)=(u_i(t,x)+\theta e^{-\lambda t}(1+|x|^{2\gamma+2}))^m_{i=1}$ is supersolution of (4.1).
\end{lem}
\begin{proof} As usual wlog we assume that the functions $(u_i)_{i=1,m}$ are $lsc$ and we use Definition \ref{defvisco}. Let $i\in A$ be fixed and  $\varphi^i\in \ccp$ such that $\varphi^i(s,y)-(u_i(s,y)+\theta e^{-\lambda s}(1+|y|^{2\gamma+2}))$ has a global maximum in $(t,x)\in(0,T)\times \R$ and $\varphi^i(t,x)= u_i(t,x)+\theta e^{-\lambda t}(1+|x|^{2\gamma+2})$. By Definition \ref{defvisco} we have:
\begin{align*}
min&\Big\{u_i(t,x)+\theta e^{-\lambda t}(1+|x|^{2\gamma+2})-\max\limits_{j\in A_i}(-g_{ij}(t,x)+(u_j(t,x)+\theta e^{-\lambda t}(1+|x|^{2\gamma+2})));\\
&-\partial_t(\varphi^i(t,x)-\theta e^{-\lambda t}(1+|x|^{2\gamma+2}))-\frac{1}{2}\sigma(t,x)^2\varpi^2D^2_{xx}(\varphi^i(t,x)-\theta e^{-\lambda t}(1+|x|^{2\gamma+2}))\\
&-(\sigma(t,x)\E(L_1)+b(t,x))D_x(\varphi^i(t,x)-\theta e^{-\lambda t}(1+|x|^{2\gamma+2}))-\int_{\R}[\varphi^i(t,x+\sigma(t,x)y)\\
&-\theta e^{-\lambda t}|x+\sigma(t,x)y|^{2\gamma+2}-(\varphi^i(t,x)-\theta e^{-\lambda t}|x|^{2\gamma+2})-D_x(\varphi^i(t,x)-\theta e^{-\lambda t}|x|^{2\gamma+2})\sigma(t,x)y]\Pi(dy)\\
&-f_i(t,x,\overrightarrow{u})\Big\}\geq 0.
\end{align*}
Then\be\label{ajousursol}
\begin{array}{ll}
&-\partial_t\varphi^i(t,x)-{\cal L}\varphi^i(t,x)-f_i(t,x,\overrightarrow{v}(t,x))\\\\
&\geq\theta\lambda e^{-\lambda t}(1+|x|^{2\gamma+2})-\frac{1}{2}\theta e^{-\lambda t} \sigma(t,x)^2\varpi^2D^2_{xx}|x|^{2\gamma+2}-(\sigma(t,x)\E(L_1)+b(t,x))D_x(\theta e^{-\lambda t}|x|^{2\gamma+2})\\
&\q -\int_{\R}(\theta e^{-\lambda t}|x+\sigma(t,x)y|^{2\gamma+2}-\theta
e^{-\lambda t}|x|^{2\gamma+2}-\theta e^{-\lambda
t}D_x|x|^{2\gamma+2}\sigma(t,x)y)\Pi(dy)+f_i(t,x,\overrightarrow{u}(t,x))\\
&\q -f_i(t,x,\overrightarrow{v}(t,x))\\\\
&\geq\theta e^{-\lambda t}\Big \{\lambda (1+|x|^{2\gamma+2})-\frac{1}{2}\sigma(t,x)^2\varpi^2D^2_{xx}|x|^{2\gamma+2}-(\sigma(t,x)\E(L_1)+b(t,x))D_x|x|^{2\gamma+2}\\
&\q -\int_{\R}(|x+\sigma(t,x)y|^{2\gamma+2}-|x|^{2\gamma+2}-D_x|x|^{2\gamma+2}\sigma(t,x)y)\Pi(dy)+\sum\limits^m_{k=1}C^{k,i}_{t,x,\theta,\lambda}(1+|x|^{2\gamma+2})\Big\}
\end{array}\ee
where $C^{k,i}_{t,x,\theta,\lambda}$ is bounded by the Lipschiz
constant of $f_i$ with respect to $(y^i)_{i=1,\cdots,m}$ which is independent of $\theta$. But, since $\phi(y)=|y|^{2\gamma+2}\in
\ccp$, then the non-local term is well-defined. Now let us set $\psi(\rho):=\phi(x+\rho\sigma(t,x)y)$, for $\rho, x,y\in \R$. First note that for any $t,x,y$ we have
\begin{align*}
|x+\sigma(t,x)y|^{2\gamma+2}&-|x|^{2\gamma+2}-D_x|x|^{2\gamma+2}\sigma(t,x)y|
=|\psi(1)-\psi(0)-D_\rho\psi(0)|=|\int^1_0(1-\rho)\psi^{(2)}(\rho)d\rho|\\
\leq\q & C |y|^{2}(|x|^{2\gamma}+|y|^{2\gamma}).
\end{align*}
Therefore by (\ref{levy1}) we have 
$$\begin{array}{ll}
\int_{\R}||x+\sigma(t,x)y|^{2\gamma+2}-|x|^{2\gamma+2}-D_x|x|^{2\gamma+2}\sigma(t,x)y|\Pi(dy)
&\le C\int_{\R}|y|^2(|x|^{2\gamma}+|y|^{2\gamma})\Pi(dy)\\&\leq C (1+|x|^{2\gamma}).
\end{array}$$
It follows that there exists a constant $\lambda_0\in \R^+$ which does not
depend on $\theta$ such that if $\lambda\geq\lambda_0$ then the right-hand side of (\ref{ajousursol}) is non-negative for any $i\in A$. Thus $\vec{v}$ is a viscosity supersolution of (\ref{FIPDE}), which is the desired result.
\end{proof}
\begin{rem} \label{modifsub}In the same way one can show that if $(u_i)^m_{i=1}$ is a viscosity subsolution of (\ref{FIPDE}) which belongs to $\pl$, i.e. for some $\gamma >0$ and $C>0$,
$$
|u_i(t,x)|\leq C(1+|x|^\gamma),\,\forall (t,x)\in \esp \mbox{ and }i\in A.
$$
Then there exists $\lambda_0>0$ such that for any $\lambda\geq\lambda_0$ and $\theta >0$, $\overrightarrow{v}(t,x)=(u_i(t,x)-\theta e^{-\lambda t}(1+|x|^{2\gamma+2}))^m_{i=1}$ is subsolution of (\ref{FIPDE}). \qed
\end{rem}
\subsection{Existence of the viscosity solution of system (\ref{FIPDE})}
In this section we deal with the issue of existence of the viscosity
solution of $(\ref{FIPDE})$. Recall  that $(Y^j,U^j,K^j)_{j\in A}$ is the unique solution of
$(\ref{SRBSDE})$ and let $(u_j(t,x))_{j\in A}$ be the functions defined in Corollary  \ref{croipoly}.
\begin{thm}Assume Assumptions ${\bf (A4)}$ and ${\bf (A5)}$ and (\ref{eqlipb}), (\ref{sigma}) as well, then $(u_j(t,x))_{j\in A}$ is a viscosity solution of $(\ref{FIPDE})$.\end{thm}
\begin{proof} The proof will be divided
into two steps.

\nd \underline{Step 1}: We first show that $(u_j)^m_{j=1}$ is a
supersolution of (\ref{FIPDE}). We will use Definition \ref{defvisco}. Note that for all $j\in A$, as $u_j$ is lsc,
we then have $u_{j_*}=u_j$. Next let us set $u^n_j(t,x)=Y^{j,n,t,x}_t$, where
$(Y^{j,n,t,x}; U^{j,n,t,x},K^{j,n,t,x})_{j\in A}$ is the unique solution of (\ref{SQ}). As pointed out in Corollary \ref{croipoly}, for any $n\geq 0$, $(t,x)\in \esp$ and $s\in [t,T]$,
$$
Y^{j,n,t,x}_s=u^n_j(s,X^{t,x}_s)\mbox{ and }u^n_j(t,x)\nearrow u_j(t,x).$$ Additionally by induction for any $n\ge 0$, $(u^n_j)_{j\in A}$, are continuous, belong  to $\pl$ and by Ren et al.'s result (\cite{renotmani}, Theorem 5.8) verify in viscosity sense the following system $(n\geq 1)$: $\forall j\in A$,
    \be\label{sq2}\left\{
    \begin{array}{ll}
      min\Big\{u^n_j(t,x)-\max\limits_{k\in A_j}(u_j^{n-1}(t,x)-g_{jk}(t,x));\\
      -\partial_t u_j^{n}(t,x)-{\cal L}
      u^n_j(t,x)-f_j(t,x,(u_1^{n-1},\cdots,u_{j-1}^{n-1},u_j^{n},u_{j+1}^{n-1},\cdots,u_m^{n-1})(t,x))\Big\}=0;\\
      u_j^{n}(T,x)=h_j(x).
    \end{array}
    \right.\ee
First note that for any $j\in A$, $u_j$ verifies
$$
u_j(T,x)=h_j(x) \mbox{ and }u_j(t,x)\geq \max_{k\in A_j}\{u_k(t,x)-g_{jk}(t,x)\}, \,\,\forall (t,x)\in \esp .
$$
Now let $j\in A$, $(t,x)\in (0,T)\times \R$ and $\phi$ a function which belongs to $\ccp$ such that $u_j-\phi$ has a global minimum in $(t,x)$ on $\esp$ (wlog we assume it strict and that $u_j(t,x)=\phi(t,x)$). Next let $\d>0$ and for $n\geq 0$ let $(t_n,x_n)$ be the global minimum of $u_j^n-\phi$ on $ [0,T]\times B'(x,2\d C_\sigma)$ ($C_\sigma$ is the constant of boundedness of the diffusion coefficient $\sigma$ which appears in (\ref{sigma}) and $B'$ stands for the closure of the ball $B$). Therefore
$$
(t_n,x_n)\rw_n(t,x) \mbox{ and }u_j^n(t_n,x_n)\rw_n u(t,x).$$
Actually let us consider a convergent subsequence of $(t_n,x_n)$, which we still denote by $(t_n,x_n)$, and let set $(t^*,x^*)$ its limit. Then
\be\label{eqtnxn}
u_j^n(t_n,x_n)-\phi(t_n,x_n)\leq u_j^n(t,x)-\phi(t,x).
\ee
Taking the limit wrt $n$ and since $u_{j*}=u_j$ is lsc to obtain
$$
u_j(t^*,x^*)-\phi(t^*,x^*)\leq u_j(t,x)-\phi(t,x).
$$As the minimum $(t,x)$ of $u_j-\phi$ on $\esp$ is strict then $(t^*,x^*)=(t,x)$. It follows that the sequence $((t_n,x_n))_n$ converges to $(t,x)$. Going back now to (\ref{eqtnxn}) and sending $n$ to infinity to obtain
$$
u_{j*}(t,x)=u_j(t,x)\leq \liminf_nu_j^n(t_n,x_n)\leq
\limsup_nu_j^n(t_n,x_n)\leq u_j(t,x)
$$
which implies that $u_j^n(t_n,x_n)\rw_n u_j(t,x)$.

Now for $n$ large enough $(t_n,x_n)\in (0,T)\times B(x,2C_\sigma\delta)$ and it is the global minimum of $u^n_j-\phi$ in
$[0,T]\times B(x_n,C_\sigma\delta)$. As $u_j^n$ is a supersolution of (\ref{sq2}), then by Definition \ref{defviscoloc22} in Appendix we have
\be\label{eqlim}\begin{array}{l}
 -\partial_t\phi(t_n,x_n)-{\cal L}^1
    \phi(t_n,x_n)-\I^{1,\delta}(t_n,x_n,\phi)\ge 
     \I^{2,\delta}(t_n,x_n,D_x\phi(t_n,x_n),u_j^{n})\\\qq\qq +f_j(t_n,x_n,u_1^{n-1}(t_n,x_n),\cdots,u_{j-1}^{n-1}(t_n,x_n),u_j^{n}(t_n,x_n),u_{j+1}^{n-1}(t_n,x_n),\cdots,u_m^{n-1}(t_n,x_n)).\end{array}
\ee
But there exists a subsequence of $\{n\}$ such that:

(i) for any $k\in A_j$, $(u_k^{n-1}(t_n,x_n))_n$ is convergent and then $\lim_nu_k^{n-1}(t_n,x_n)\geq u_{k*}(t,x)=u_k(t,x)$ ;

(ii) $(\I^{1,\delta}(t_n,x_n,\phi))_n\rw_n
\I^{1,\delta}(t,x,\phi)$.
\ms

\nd Next by Fatou's Lemma and since $u_{j*}=u_j$ and $u_j\geq \phi$ we have 
\be \lb{ineqfatou}\begin{array}{ll}
\liminf_{n\rw \infty}\I^{2,\delta}(t_n,x_n,D_x\phi(t_n,x_n),u_j^{n})&\geq \I^{2,\delta}(t,x,D_x\phi(t,x),u_j)\\
{}&\geq \I^{2,\delta}(t,x,D_x\phi(t,x),\phi).
\ea\ee
\noindent Taking the $\liminf$ wrt to $n$ (through the previous subsequence) in each hand-side of (\ref{eqlim}), using the fact that $f_j$ is continuous and verifies (A4)(I)(v) and finally by (\ref{ineqfatou}) to obtain:
$$\begin{array}{l}
 -\partial_t\phi(t,x)-{\cal L}^1
    \phi(t,x)-\I^{1,\delta}(t,x,\phi)
    \geq \\\qq\qq \qq \I^{2,\delta}(t,x,D_x\phi(t,x),u_{j})+f_j(t,x,u_1(t,x),\cdots,u_{j-1}(t,x),u_{j}(t,x),u_{j+1}(t,x),\cdots,u_{m}(t,x)).\end{array}
$$As $u_j\geq \phi$ and since $\I(\dots)=\I^{1,\delta}(\dots)+\I^{2,\delta}(\dots)$ we then obtain from the previous inequality,  
$$\begin{array}{l}
 -\partial_t\phi(t,x)-{\cal L}^1
    \phi(t,x)
    \geq \\\qq\qq \qq \I(t,x,\phi)+f_j(t,x,u_1(t,x),\cdots,u_{j-1}(t,x),u_{j}(t,x),u_{j+1}(t,x),\cdots,u_{m}(t,x))\end{array}
$$which means that $u_j$ is a viscosity supersolution of
$$\left\{
    \begin{array}{ll}
      min\{u_j(t,x)-\max\limits_{k\in A_j}(u_k(t,x)-g_{jk}(t,x));\\\qq\qq
      -\partial_t u_j(t,x)-{\cal L} u_j(t,x)-f_j(t,x,u_1(t,x),\cdots, u_m(t,x))\}=0\,\,;\\
      u_j(T,x)=h_j(x).
    \end{array}
    \right.$$
As $j$ is arbitrary then $(u_j)_{j\in A}$ is a viscosity supersolution of (\ref{FIPDE}). \qed
\ms

\nd \underline{Step 2}: We will now show that $(u^*_j)_{j\in A}$ is
a subsolution of (\ref{FIPDE}). As a first step we are going to show that
$$\forall j\in A,\,\,min\{u^*_j(T,x)-h_j(x); ~~u^*_j(T,x)-\max\limits_{k\in
A_j}(u^*_k(T,x)-g_{jk}(T,x))\}=0.$$ By definition of $u^*_j$ and
since $u^n_j\nearrow u_j $, we have
$$min\{u^*_j(T,x)-h_j(x); ~~u^*_j(T,x)-\max\limits_{k\in
A_j}(u^*_k(T,x)-g_{jk}(T,x))\}\geq 0$$ Next suppose that for some
$x_0\in \R$, $\exists j>0$, s.t.$$min\{u^*_j(T,x_0)-h_j(x_0);
~~u^*_j(T,x_0)-\max\limits_{k\in
A_j}(u^*_k(T,x_0)-g_{jk}(T,x_0))\}=2\epsilon .$$
We will show that leads to a contradiction. Let $(t_k,x_k)_{k\geq 1}\rightarrow (T,x_0)$ and
$u_j(t_k,x_k)\rightarrow u^*_j(T,x_0)$. We can find a sequence of functions
$(v^n)_{n\geq 0}\in \cC^{1,2}([0,T]\times \R)$ of compact support such that $v^n\rightarrow u^*_j$, since $u^*_j$ is usc.
On some neighborhood $B_n$ of $(T,x_0)$ we have,
\be\label{add3}\forall (t,x)\in B_n,\,\,min\{v^n(t,x)-h_j(x);~~v^n(t,x)-\max\limits_{k\in
A_j}(u^*_k(t,x)-g_{jk}(t,x))\}\geq \epsilon .\ee Let us denote by
$B^n_k := [t_k,T]\times B(x_k,\delta^k_n)$, for some
$\delta^k_n\in]0,1]$ small enough such that $B^n_k\subset B_n$. Since
$u^*_j$ is of polynomial growth, there exists $c>0$, such that
$|u^*_j|\leq c$ on $B_n$. We can then assume $v^n\geq-2c$ on $B_n$.
Define $$
V^n_k(t,x):=v^n(t,x)+\frac{4c|x-x_k|^2}{{\delta^n_k}^2}+\sqrt{T-t}$$
Note that $V^n_k(t,x)\geq v^n(t,x)$ and \be\label{add2}(u^*_j-V^n_k)(t,x)\leq-c
~~\forall (t,x)\in [t_k,T]\times \partial B(x_k,\delta^n_k).\ee On
the other hand, an easy calculation yields 
$$\begin{array}{ll}
&-\{\partial_t V^n_k(t,x)+{\cal L}V^n_k(t,x)\}
\\
&\qq=-\Big\{\partial_t v^n(t,x)+\partial_t((T-t)^{\frac{1}{2}})+\{\E(L_1)\sigma(t,x)+b(t,x)\}\{\partial_x v^n(t,x)+\frac{8c(x-x_k)}{(\delta^n_k)^2}\}\\
&\qq \qq+\frac{1}{2}\sigma(t,x)^2\varpi^2(D_{xx}^2v^n(t,x)+\frac{8c}{(\delta^n_k)^2})+\int_\R[v^n(t,x+\sigma(t,x)y)-v^n(t,x)-\partial_x
v^n(t,x)\sigma(t,x)y]\Pi(dy)\\
&\qq \qq+\int_\R[\frac{4c|x-x_k+\sigma(t,x)y|^2}{(\delta^n_k)^2}-\frac{4c|x-x_k|^2}{(\delta^n_k)^2}-\frac{8c(x-x_k)}{(\delta^n_k)^2}\sigma(t,x)y]\Pi(dy)\Big \}.
\end{array}$$
Note that
 $\Phi(x):=\frac{4c|x-x_k|^2}{(\delta^n_k)^2}\in C^2\cap\Pi_g$ and $v^n\in \cC^{1,2}$ and of compact support, then the two non-local terms are
 bounded and $\partial_t v^n$, $\partial_x v^n$, $D_{xx}^2
v^n$ are so. Since
$\partial_t(\sqrt{T-t})\rightarrow -\infty$, when $t\rw T$, then we can choose $t_k$
large enough in front of $\delta_k$ and the derivatives of $v^n$ to
ensure that \be\label{add14}-(\partial_t V^n_k(t,x)+{\cal
L}V^n_k(t,x))\geq 0, ~~ \forall(t,x)\in B^k_n .\ee Consider now the
stopping time $\theta^k_n:=inf\{s\geq t_k, (s,X^{t_k,x_k}_s)\in
{B^k_n}^c\}\wedge T$, where ${B^k_n}^c$ is the complement of
$B^k_n$ and $\theta_k:= inf\{s\geq
t_k,u_j(s,X^{t_k,x_k}_s)=\max\limits_{l\in
A_j}(u_l(s,X^{t_k,x_k}_s)-g_{jl}(s,X^{t_k,x_k}_s))\}\wedge T$.
Applying It\^o's formula with $V^n_k(t,x)$ on
$[t_k,\theta^k_n\wedge\theta_k]$ yields:
\be\label{eqvnk}\begin{array}{ll}
V^n_k(t_k,x_k)&=V^n_k(\theta^k_n\wedge\theta_k,X^{t_k,x_k}_{\theta^k_n\wedge\theta_k})-\int^{\theta^k_n\wedge\theta_k}_{t_k}[b(r,X^{t_k,x_k}_{r})\partial_x V^n_k(r,X^{t_k,x_k}_r)+\partial_t V^n_k(t,x)(r,X^{t_k,x_k}_r)]dr\\\\
&-\int^{\theta^k_n\wedge\theta_k}_{t_k}
\sigma(r,X^{t_k,x_k}_{r_-})
\partial_x V^n_k(r,X^{t_k,x_k}_{r_-})dL_r-\frac{1}{2}\int^{\theta^k_n\wedge\theta_k}_{t_k}
\sigma^2(r,X^{t_k,x_k}_{r})\varpi^2
\partial^2_{xx} V^n_k
(r,X^{t_k,x_k}_r)dr\\\\
&-\sum\limits_{t_k<
r\leq\theta^k_n\wedge\theta_k}\{V^n_k(r,X^{t_k,x_k}_r)-V^n_k(r,X^{t_k,x_k}_{r_-})-
\sigma(r,X^{t_k,x_k}_{r_-})\partial_x V^n_k(r,X^{t_k,x_k}_{r_-})\dt L_r\}.
\end{array}\ee
Next let us deal with the last term of (\ref{eqvnk}) and let us set
$$h(s,y)=V^n_k(s,X^{t_k,x_k}_{s-}+\sigma(s,X^{t_k.x_k}_{s-})y)-V^n_k(s,X^{t_k,x_k}_{s-})-\partial_x V^n_k(s,X^{t_k.x_k}_{s-})
\sigma(s,X^{t_k.x_k}_{s-})y.$$
By the mean value
theorem we have
$$
h(s,y)=\frac{1}{2}\partial^2_{xx}v^n(s,X^{t_k,x_k}_{s-}+\bar X \sigma (s,X^{t_k.x_k}_{s-})y)(\sigma (s,X^{t_k.x_k}_{s-})y)^2+
\frac{4c}{{\delta^n_k}^2}(\sigma (s,X^{t_k.x_k}_{s-})y)^2
$$
where $\bar X$ is a stochastic processes which is valued in $(0,1)$. As $v^n$ is of compact support and $\sigma$ is bounded then
$$\bal
\E[\int_0^T\!\!\int_\R|h(s,y)|\Pi(dy)ds]<\infty.\ea
$$
It follows that
$$
\E[\sum\limits_{t_k<r\leq\theta^k_n\wedge\theta_k}\{V^n_k(r,X^{t_k,x_k}_r)-V^n_k(r,X^{t_k,x_k}_{r_-})-\sigma(r,X^{t_k,x_k}_{r_-})\partial_x V^n_k(r,X^{t_k,x_k}_{r_-})\dt L_r\}]=
\E[\int_{t_k}^{\theta^k_n\wedge\theta_k}
\int_\R h(s,y)\Pi(dy)ds]<\infty.
$$
Next going back to (\ref{eqvnk}), taking expectation and taking into account of (\ref{add14}), (\ref{add2}) and 
(\ref{add3}) to obtain
\begin{align*}
V^n_k(t_k,x_k)&=\E[V^n_k(\theta^k_n\wedge\theta_k,X^{t_k,x_k}_{\theta^k_n\wedge\theta_k})-\int^{\theta^k_n\wedge\theta_k}_{t_k}(\partial_t V^n_k(r,X^{t_k,x_k}_r)+{\cal L}V^n_k(r,X^{t_k,x_k}_{r}))dr]\\
&\geq \E[V^n_k(\theta^k_n,X^{t_k,x_k}_{\theta^k_n})\mathbbm{1}_{\{\theta^k_n\leq\theta_k\}}+V^n_k(\theta_k,X^{t_k,x_k}_{\theta_k})\mathbbm{1}_{\{\theta^k_n>\theta_k\}}]\\
&=\E[\{V^n_k(\theta^k_n,X^{t_k,x_k}_{\theta^k_n})\mathbbm{1}_{\{\theta^k_n< T\}}+V^n_k(T,X^{t_k,x_k}_T)\mathbbm{1}_{\{\theta^k_n=T\}}\}\mathbbm{1}_{\{\theta^k_n\leq\theta_k\}}+V^n_k(\theta_k,X^{t_k,x_k}_{\theta_k})\mathbbm{1}_{\{\theta^k_n>\theta_k\}}]\\
&\geq \E[\{(u^*_j(\theta^k_n,X^{t_k,x_k}_{\theta^k_n})+c)\mathbbm{1}_{\{\theta^k_n<T\}}+(\epsilon+h_j(X^{t_k,x_k}_T)) \mathbbm{1}_{\{\theta^k_n=T\}}\}\mathbbm{1}_{\{\theta^k_n\leq\theta_k\}}\\
&\qq +\{\epsilon +\max\limits_{k\in A_j}(u^*_k(\theta_k,X^{t_k,x_k}_{\theta_k})-g_{jk}(\theta_k,X^{t_k,x_k}_{\theta_k}))\}\mathbbm{1}_{\{\theta^k_n>\theta_k\}}]\\
&\geq \E[u_j(\theta^k_n\wedge\theta_k,X^{t_k,x_k}_{\theta^k_n\wedge\theta_k})]+c\wedge\epsilon\\
&=\E[u_j(t_k,x_k)-\int^{\theta^k_n\wedge\theta_k}_{t_k}f_j(s,X^{t_k,x_k}_s,(u_l(s,X^{t_k,x_k}_s))_{l=1,m}ds]+c\wedge\epsilon
\end{align*}
since the processes $(Y^j=u_j(.,X.))_{j\in A}$ stopped at time
$\theta^k_n\wedge\theta_k$ solves an explicit RBSDE system with
triple of data given by $((f_j)_{j\in A}, (h_j)_{j\in A},
(g_{i,j})_{i,j\in A})$. In addition, $dK^{j,t,x}=0$ on
$[t_k,\theta_k]$. On the other hand,
$(u_j)_{j\in A}\in \Pi_g$ and then taking into account (\ref{sde1}) and
Assumption (A4)(I)(iii), we deduce that
$$\lim\limits_{k\rightarrow\infty}E[\int^{\theta^k_n\wedge\theta_k}_{t_k}f_j(s,X^{t_k,x_k}_s,(u_l(s,X^{t_k,x_k}_s))_{l=1,m})ds]=0.$$
Taking the limit in the previous inequalities yields: \begin{align*}
&\lim\limits_{k\rightarrow\infty}
V^n_k(t_k,x_k)=\lim
\limits_{k\rightarrow\infty}\{v^n(t_k,x_k)+\sqrt{T-t_k}\}=v^n(T,x_0)\\
&\qq\qq\geq\lim\limits_{k\rightarrow\infty}
u_j(t_k,x_k)+c\wedge\epsilon=u^*_j(T,x_0)+
c\wedge\epsilon.
\end{align*}
As $v^n\rightarrow u^*_j$ pointwisely, then we get a contradiction, when taking the limit in the previous inequalities, and the result follows, i.e., $\forall x\in \R$, $\forall j\in
A$,$$min\{u^*_j(T,x)-h_j(x); ~~u^*_j(T,x)-\max\limits_{l\in
A_j}(u^*_l(T,x)-g_{jl}(T,x))\}=0.$$ Finally the proof of $u^*_j(T,x)=h_j(x), \forall j\in A$, is obtained in the same way as in (\cite{saidmorlais}, pp.180) since the function $(g_{ij})_{i,j\in A}$ verify the non free loop property (A4)(II). \qed
\ms

Now let us  show $(u^*_j)_{j\in A}$ is a subsolution of (\ref{FIPDE}). First note that since $u^n_j\nearrow u_j$ and $u^n_j$ is
continuous, we have
$$u^*_j(t,x)=\limsup_{n\rightarrow\infty}\!\!{}^*
u^n_j(t,x)=\varlimsup\limits_{n\rightarrow\infty , t'\rightarrow
t,x'\rightarrow x}u^n_j(t',x').$$ Besides $\forall j\in A$ and
$n\geq 0$ we deduce from the construction of $u^n_j$  that $$u^n_j(t,x)\geq\max\limits_{l\in A_j}(u^{n-1}_l(t,x)-g_{jl}(t,x))$$ and by taking  the limit in $n$ we obtain: $\forall j\in A$, $\forall x\in \R$,$$u^*_j(t,x)\geq
\max\limits_{l\in A_j}(u^*_l(t,x)-g_{jl}(t,x)).$$ Next fix $j\in
A$. Let $(t,x)\in(0,T)\times \R$ be such that \be\label{14}u^*_j(t,x)-\max\limits_{l\in
A_j}(u^*_l(t,x)-g_{jl}(t,x))>0.\ee
We are going to use once more Definition \ref{defvisco}. Let $(t,x)\in (0,T)\times \R^k$ and $\phi$ be a function of $\ccp$ such that $u^*_j-\phi$ has a global maximum at $(t,x)$ on $\esp$ which wlog we assume strict and verifying $u_j^*(t,x)=\phi(t,x)$. Then there exist subsequences $\{n_k\}$ and $((t'_{n_k},x'_{n_k}))_k$ such that
$$
((t'_{n_k},x'_{n_k}))_k\rw_k (t,x) \mbox{ and }u_j^{n_k}(t'_{n_k},x'_{n_k})\rw_k u_j^*(t,x).
$$
Let now 
$\d>0$ and $(t_{n_k},x_{n_k})$ be the global maximum of $u_j^{n_k}-\phi$ on $[0,T]\times B'(x,2\d C_\sigma )$. Therefore
$$
(t_{n_k},x_{n_k})\rw_k(t,x) \mbox{ and }u_j^{n_k}(t_{n_k},x_{n_k})\rw_k u_j^*(t,x).$$
Actually let us consider a convergent subsequent of $(t_{n_k},x_{n_k})$, which we still denote by $(t_{n_k},x_{n_k})$, and let $(\bar t,\bar x)$ be its limit. Then for some $k_0$ and for $k\geq k_0$ we have
\be\label{eqtnxnsub}
u_j^{n_k}(t_{n_k},x_{n_k})-\phi(t_{n_k},x_{n_k})\geq u_j^{n_k}(t'_{n_k},x'_{n_k})-\phi(t'_{n_k},x'_{n_k}).
\ee
Taking the limit wrt $k$ to obtain
$$
u_j^*(\bar t,\bar x)-\phi(\bar t,\bar x)\geq u_j^*(t,x)-\phi(t,x).
$$As the maximum $(t,x)$ of $u_j^*-\phi$ on $\esp$ is strict then $(\bar t,\bar x)=(t,x)$. It follows that the sequence $((t_{n_k},x_{n_k}))_k$ converges to $(t,x)$. Going back now to (\ref{eqtnxnsub}) and taking the limit wrt $k$ we obtain
$$
u_j^*(t,x)\geq \limsup_ku_j^{n_k}(t_{n_k},x_{n_k})\geq
\liminf_ku_j^{n_k}(t_{n_k},x_{n_k})\geq
\liminf_ku_j^{n_k}(t'_{n_k},x'_{n_k})= u_j^*(t,x)
$$
which implies that $u_j^{n_k}(t_{n_k},x_{n_k})\rw u_j^*(t,x)$ as $k\rw \infty$.
\ms

\nd Now for $k$ large enough,

(i) $(t_{n_k},x_{n_k})\in (0,T)\times B(x,2\d C_\sigma)$ and is the global maximum of $u^{n_k}_j-\phi$ in
$[0,T]\times B(x_{n_k},C_\sigma\delta)$ ;

(ii) $u^{n_k}_j(t_{n_k},x_{n_k})>\max\limits_{l\in
A_j}(u^{{n_k}-1}_l(t_{n_k},x_{n_k})-g_{jl}(t_{n_k},x_{n_k}))$.
\ms

\nd As $u_j^{n_k}$ is a subsolution of (\ref{sq2}), then by Definition \ref{defviscoloc22} in Appendix we have
\be\label{eqlim2}\begin{array}{l}
 -\partial_t\phi(t_{n_k},x_{n_k})-{\cal L}^1
    \phi(t_{n_k},x_{n_k})\le \I^{1,\delta}(t_{n_k},x_{n_k},\phi)+
    \I^{2,\delta}(t_{n_k},x_{n_k},D_x\phi(t_{n_k},x_{n_k}),
    u_j^{{n_k}})+ \\f_j(t_{n_k},x_{n_k},u_1^{{n_k}-1}(t_{n_k},x_{n_k}),\cdots,u_{j-1}^{n_k-1}(t_{n_k},x_{n_k}),u_j^{{n_k}}(t_{n_k},x_{n_k}),u_{j+1}^{{n_k}-1}(t_{n_k},x_{n_k}),\cdots,u_m^{{n_k}-1}(t_{n_k},x_{n_k})).\end{array}
\ee
But there exists a subsequence of $\{n_k\}$ (which we still denote by $\{n_k\}$) such that:

(i) for any $l\in A_j$, $(u_l^{n_k-1}(t_{n_k},x_{n_k}))_k$ is convergent and then $\lim_ku_l^{{n_k}-1}(t_{n_k},x_{n_k})\leq u_l^*(t,x)$ ;

(ii) $(\I^{1,\delta}(t_{n_k},x_{n_k},\phi))_{n_k}\rw_k
\I^{1,\delta}(t,x,\phi)$ ; 

(iii) 
 $$\limsup_k\I^{2,\delta}(t_{n_k},x_{n_k},D_x\phi(t_{n_k},x_{n_k}),u_j^{{n_k}})\leq \I^{2,\delta}(t,x,D_x\phi(t,x),u_j^*).$$Point (i) is due to the fact that $u_l^n$ belongs uniformly to $\Pi_g$ ; (ii) is just the Lebesgue dominated convergence Theorem ; (iii) stems from an adaptation of Fatou's Lemma, definition of $u_j^*$ and finally monotonicity of $\I^{2,\d}$. 

Going back now to (\ref{eqlim2}) and taking the limit superior wrt $k$ (through the previous subsequence), using the fact that $f_j$ is continuous and verifies (A4)(I)(v) to obtain
$$\begin{array}{l}
 -\partial_t\phi(t,x)-{\cal L}^1
    \phi(t,x)\\
   \qq \leq \I^{1,\delta}(t,x,\phi)+\I^{2,\delta}(t,x,D_x\phi(t,x),u_{j}^*)+f_j(t,x,u^*_1(t,x),\cdots,u_{j-1}^*(t,x),u_{j}^*(t,x),u_{j+1}^*(t,x),\cdots,u^*_{m}(t,x))
\\ \qq \leq \I(t,x,D_x\phi(t,x),\phi)+f_j(t,x,u^*_1(t,x),\cdots,u_{j-1}^*(t,x),u_{j}^*(t,x),u_{j+1}^*(t,x),\cdots,u^*_{m}(t,x)).
   \end{array}
$$ This last inequality is due to that fact that $u_j^*\leq \phi$ and since $\I^{1,\d}+\I^{2,\d}=\I$. Finally combining it with (\ref{14}) we obtain that $u_j$ is a viscosity subsolution of
$$\left\{
    \begin{array}{ll}
      min\{u_j(t,x)-\max\limits_{k\in A_j}(u_k^*(t,x)-g_{jk}(t,x));\\\qq\qq
      -\partial_t u_j(t,x)-{\cal L} u_j(t,x)-f_j(t,x,u_1^*(t,x),\cdots,
u_{j-1}^*(t,x),u_j(t,x), 
u_{j+1}^*(t,x),\dots, u_{m}^*(t,x))\}=0;\\
      u_j(T,x)=h_j(x).
    \end{array}
    \right.$$
As $j$ is arbitrary then $(u_j)_{j\in A}$ is a viscosity subsolution of (\ref{FIPDE}).
\end{proof}
\subsection{Uniqueness of the viscosity solution of system (\ref{FIPDE})}
We now give a comparison result of subsolution and supersolution of system (\ref{FIPDE}),
from which we get the continuity and uniqueness of its solution.

\begin{propo}\label{uniq}Assume Assumptions {\bf (A4)} fulfilled. Let $(u_j)_{j\in A}$ (resp. $(w_j)_{j\in A}$) be a subsolution (resp. supersolution) of $(\ref{FIPDE})$ which belongs to $\Pi_g$. Then for any $j\in A$,
$$\forall (t,x)\in [0,T]\times \R,~~u_j(t,x)\leq w_j(t,x)$$\end{propo}
\begin{proof} Let $\gamma$ be a real constant such that for any $j\in A$ and $(t,x)\in \esp$,
$$
|u_j(t,x)|+|w_j(t,x)|\leq C(1+|x|^\gamma).
$$
To begin with we additionally assume the existence of a constant $\lambda$ such that $\lambda<-m.\max\limits_{j\in A}\{C_j\}$ ($C_j$ is the Lipschitz constant of $f_j$ w.r.t $\overrightarrow{y}$) and for any $j\in A$ and any $t,x, y_1,\cdots,y_{j-1},y_{j+1},\cdots, y_m$, $y\geq y', $
\be\label{3.3}f_j(t,x,y_1,\cdots,y_{j-1},y,y_{j+1},\cdots,y_m)-f_j(t,x,y_1,\cdots,y_{j-1},y',y_{j+1}\cdots,y_m)\leq\lambda(y-y').\ee
Thanks to Lemma 4.1 and Remark \ref{modifsub}, we know there exists $\nu$ large enough such
that for any $\theta>0$, $w_{j,\theta,\nu}(t,x)=w_j(t,x)+\theta
e^{-\nu t}(1+|x|^{2\gamma+2})$ (resp. $u_{j,\theta,\nu}(t,x)=u_j(t,x)-\theta
e^{-\nu t}(1+|x|^{2\gamma+2})$) is a supersolution (resp. subsolution). So it is enough to show
that
$$\forall j\in A,\,\,\forall (t,x)\in[0,T]\times \R,~~u_{j,\theta,\nu}(t,x)\leq w_{j,\theta,\nu}(t,x),$$ then taking limits as $\theta\rightarrow 0$, the result follows. By the growth condition there exists a constant $C>0$ such that
\be\label{8101}\forall j\in A, \forall (t,x)\in [0,T]\times \R, ~~s.t.~ |x|\geq
C,~u_{j,\theta,\nu}(t,x)<0<w_{j,\theta,\nu}(t,x).\ee Finally for the sake of simplicity we merely denote $u_{j,\theta,\nu}$ (resp. $w_{j,\theta,\nu}$) by $u_j$ (resp. $w_j$).

To obtain the comparison result, we proceed by contradiction assuming
that $$\exists (t_1,x_1)\in [0,T]\times \R,\mbox{ such that }\max\limits_{j\in
A}(u_j(t_1,x_1)-w_j(t_1,x_1))>0.$$ Taking into account the values of the
subsolution and the supersolution at $T$, there exist
$(\bar{t},\bar{x})\in [0,T[\times B(0,C)$ (wlog we assume that $\bar t>0$), such that :
\begin{align*}
0&<\max\limits_{(t,x)\in [0,T]\times \R}\max\limits_{j\in A}(u_j(t,x)-w_j(t,x))\\
&\qq\qq =\max\limits_{(t,x)\in[0,T[\times B(0,C)}\max\limits_{j\in A}(u_j(t,x)-w_j(t,x))=\max\limits_{j\in A}(u_j(\bar{t},\bar{x})-w_j(\bar{t},\bar{x})).
\end{align*}
We now define the set $\underbar{A}$ as follows:
\be\label{uniquea}\underbar{A}:=\{j\in A,
u_j(\bar{t},\bar{x})-w_j(\bar{t},\bar{x})=\max\limits_{k\in
A}(u_k(\bar{t},\bar{x})-w_k(\bar{t},\bar{x}))\}.\ee By the
assumption {\bf (A4)}(II), using the same argument as in (\cite{saidmorlais}, pp. 171), we can prove that there exists $j\in\underbar{A}$ such that, \be\label{36}u_j(\bar{t},\bar{x})>\max\limits_{k\in
A_j}(u_k(\bar{t},\bar{x})-g_{jk}(\bar{t},\bar{x})).\ee Let us now take such a $j\in \underbar{A}$. For $\varepsilon>0$ and  $\rho>0$,  let us define
$$\Phi^j_{\varepsilon,\rho}(t,x,y):=u_j(t,x)-w_j(t,y)-\frac{|x-y|^2}{\varepsilon}-|t-\bar{t}|^2-\rho|x-\bar{x}|^{4}.$$  
By (\ref{8101}) and since $\lim_{|y|\rw \infty} w_j(t,y)=\infty$, $\lim_{|x|\rw \nf}u_j(t,x)=-\infty$, there exists a constant $C'$ such that for any $t\in [0,T]$, $u_j(t,x)-w_j(t,y)<0$ for any $|x|\geq C'$ or $|y|\geq C'$. It follows that for any $\varepsilon>0$ and $\rho>0$, there exists $(t_0,x_0,y_0)$ such that
$$\Phi^j_{\varepsilon,\rho}(t_0,x_0,y_0)=\max\limits_{(t,x,y)\in [0,T]\times
B'(0,C')^2}
\Phi^j_{\varepsilon,\rho}(t,x,y)=\max\limits_{(t,x,y)\in [0,T]\times
\R^2}\Phi^j_{\varepsilon,\rho}(t,x,y).$$ Note that the maximum exists since
$\Phi^j_{\varepsilon,\rho}$ is $usc$ and $B'(0,C')^2$ is the closure of
$B(0,C')^2$.
On the other hand let us point out that $(t_0,x_0,y_0)$ depends actually on $\varepsilon$ and $\rho$ which we omit for sake of simplicity.  We then have,
\be\label{unique83}
\begin{aligned}
\Phi^j_{\varepsilon,\rho}(\bar t,\bar x,\bar x)&=u_j(\bar{t},\bar{x})-w_j(\bar{t},\bar{x})\\&\leq u_j(\bar t,\bar x)-w_j(\bar t,\bar x)+\frac{|x_0-y_0|^2}{\varepsilon}+
|t_0-\bar{t}|^2+\rho|x_0-\bar{x}|^{4}\\
&\leq u_j({t}_0,{x}_0)-w_j({t}_0,{y}_0).\\
\end{aligned}\ee The  growth condition of $u_j$ and $w_j$ implies that $\varepsilon^{-1}|x_0-y_0|^2+
|t_0-\bar{t}|^2+\rho|x_0-\bar{x}|^{4}$ is bounded and then $\lim\limits_{\varepsilon\rightarrow 0}(x_0-y_0)=0$. Next by (\ref{unique83}), for any
subsequence
$(t_{0_l},x_{0_l},y_{0_l})_l$ which
converges to $(\tilde{t},\tilde{x},\tilde{x})$,
$$u_j(\bar{t},\bar{x})-w_j(\bar{t},\
\bar{x})\leq u_j(\tilde{t},\tilde{x})-w_j(\tilde{t},\tilde{x}),$$
since $u_j$ is $usc$ and $w_j$ is $lsc$. By the definition of
$(\bar{t},\bar{x})$ this last inequality is an equality. Using both
the definiton of $\Phi^j_{\varepsilon,\rho}$ and (\ref{unique83}), it
implies that the sequence
\be\label{unique6}\lim\limits_{\varepsilon\rightarrow
0}(t_0,x_0,y_0)=(\bar{t},\bar{x},\bar{x})\ee
and once more from (\ref{unique83}) we deduce
\be\label{unique3171}
\lim\limits_{\varepsilon\rightarrow
0}\varepsilon^{-1}|x_0-y_0|^2=0. \ee
Finally classically (see e.g. \cite{saidmorlais}, pp. 173) we have also
\be\label{unique85} \lim\limits_{\varepsilon \rw 0}
(u_j(t_0,x_0),w_j(t_0,y_0))=(u_j(\bar{t},\bar{x}),w_j(\bar{t},\bar{x})).\ee
Next as the functions $(u_k)_{k\in A}$ are $usc$ and
$(g_{ij})_{i,j\in A}$ are continuous, and since the index $j$
satisfies (\ref{36}), there exists $r>0$ such that for
$(t,x)\in  B((\bar{t},\bar{x}),r)$ we have
$u_j(t,x)>\max\limits_{k\in A_j}(u_k(t,x)-g_{jk}(t,x))$. But by
(\ref{unique85}), (\ref{unique6}) and once more since $u_j$ is $usc$
then there exists $\varepsilon_0$ such that for any $0<\varepsilon <\varepsilon_0$,
we have:$$u_j(t_0,x_0)>\max\limits_{k\in
A_j}(u_k(t_0,x_0)-g_{ij}(t_0,x_0)).$$

Now for $\varepsilon$ small enough,
 we are able to apply Jensen-Ishii's lemma 
for non local operators (see e.g. Barles et al. \cite{barles}, pp.583 or Biswas et al. \cite{Bis}, Lemma 4.1, pp.64) with  $u_j$, $w_j$ and
$\phi(t,x,y):=\frac{|x-y|^2}{\varepsilon}+|t-\bar{t}|^2+\rho|x-\bar{x}|^{4}$ at point
$(t_0,x_0,y_0)$. For any $\d\in (0,1)$ there are
$p^0_u,q^0_u, p^0_w$, $q^0_w$, $M^0_u$ and $M^0_w$ real constants such that:
\ms

(i)
\be\begin{aligned}
p^0_u-p^0_w=\partial_t{\phi}(t_0,x_0,y_0),
~~q^0_u=\partial_x{\phi}(t_0,x_0,y_0),~~q^0_w=-\partial_y{\phi}(t_0,x_0,y_0)\end{aligned}\ee
and
 \be \label{unique7} \left(\begin{matrix}
      M^0_u & 0 \cr
      0 & -M^0_w
      \end{matrix}\right)\leq \frac{4}{\varepsilon}\left(\begin{matrix}
      1 & -1\cr
      -1 & 1
      \end{matrix}\right)+ \left(\begin{matrix}
     12\rho|x_0-\bar{x}|^2 & 0 \cr
      0 & 0
      \end{matrix}\right)\,;\ee\\
\be\label{uniqueinequa1}
\begin{aligned}
\mbox{ (ii)}\,\,&-p^0_u-\{\sigma(t_0,
x_0)\E(L_1)+b(t_0,
x_0)\}q^0_u-\frac{1}{2}\sigma(t_0,x_0)^2
\varpi^2M^0_u-f_j(t_0,x_0,(u_k(t_0,x_0))^m_{k=1})\\
&
-I^{1,\delta}(t_0,x_0,\phi(t_0,.,y_0))-I^{2,\delta}(t_0,x_0,q^0_u,u_j)\leq
0\,;
\end{aligned}
\ee \be\label{uniqueinequa2}
\begin{aligned}
\mbox{ (iii)}\,\,&-p^0_w-\{\sigma(t_0,
y_0)\E(L_1)+b(t_0,
y_0)\}q^0_w-\frac{1}{2}\sigma(t_0,y_0)^2
\varpi^2M^0_w-f_j(t_0,y_0,(w_k(t_0,y_0))^m_{k=1})\\
&
-I^{1,\delta}(t_0,y_0,-\phi(t_0, x_0,.))-I^{2,\delta}(t_0,y_0,q^0_w,w_j)\geq
0.
\end{aligned}\ee
We are now going to provide estimates for the non-local terms. First let us set $\psi_\rho(t,x):=\rho|x-\bar{x}|^4+|t-\bar{t}|^2$. By definition of
$(t_0,x_0,
y_0)$, for any $d,d'\in \R$,
$$u_j(t_0,x_0+d')-w_j(t_0,y_0+d)-\varepsilon^{-1}|x_0+d'-
y_0-d|^2-\psi_\rho(t_0,x_0+d')\leq
u_j(t_0,x_0)-w_j(t_0,y_0)-\varepsilon^{-1}|x_0-
y_0|^2-\psi_\rho(t_0,x_0).$$Therefore for $z\in \R$, in taking $d'=\sigma(t_0,x_0)z$ and
$d=\sigma(t_0,y_0)z$, we obtain
$$\begin{array}{l}u_j(t_0,x_0+
\sigma(t_0,x_0)z)
-u_j(t_0,x_0)-q_u^0\sigma(t_0,x_0)z\\\qq\leq w_j(t_0,y_0+
\sigma(t_0,y_0)z)-
w_j(t_0,y_0)
-q_w^0\sigma(t_0,y_0)z+
\varepsilon^{-1}|\sigma(t_0,x_0)
-\sigma(t_0,y_0)|^2z^2\\
\qq\qq+\psi_\rho(t_0,x_0+\sigma(t_0,x_0)z)-\psi_\rho(t_0,x_0)-D_x\psi_\rho(t_0,x_0)\sigma(t_0,x_0)z.\end{array}$$
It implies that for any $\d>0$,
\be\label{estimlocal2}
I^{2,\delta}(t_0,x_0,q^0_u,u_j)
-I^{2,\delta}(t_0,y_0,
q^0_w,w_j)\leq C\varepsilon^{-1}|x_0-y_0|^2+I^{2,\delta}(t_0,x_0,D_x\psi_\rho(t_0,x_0),\psi_\rho)\ee since $\sigma(t,x)$ is uniformly Lipschitz $\wr$ $x$. But it easy to check that we have 
$$\bal |I^{2,\delta}(t_0,x_0,D_x\psi_\rho(t_0,x_0),\psi_\rho)|\leq C \rho \int_{|z|\geq \d}\{|z|^2+|z|^4\}\Pi(dz).\ea $$ On the other hand, since $\phi\in \cC^2$
$$\begin{array}{ll}
I^{1,\delta}(t_0,x_0,\phi(t_0,.,y_0))&=
\int_{|z|\leq \d}\{\phi(t_0,x_0+
\sigma(t_0,x_0)z,y_0)-
\phi(t_0,x_0,y_0)-D_x\phi
(t_0,x_0,y_0)\sigma(t_0,x_0)z\}\Pi(dz)
\\\\{}&\leq {\sigma(t_0,x_0)^2}\int_{|z|\leq \d}\{\varepsilon^{-1}+C\rho(1+|z|^2)\}|z|^2\Pi(dz),
\end{array}$$
and
$$\begin{array}{ll}
I^{1,\delta}(t_0,y_0,- \phi(t_0,x_0,.))&=
\int_{|z|\leq \d}\{-\phi(t_0,x_0, y_0+
\sigma(t_0,y_0)z)+
\phi(t_0,x_0,y_0)+D_y\phi
(t_0,x_0,y_0)\sigma(t_0,y_0)z\}\Pi(dz)
\\\\{}&=-\varepsilon^{-1}{\sigma(t_0,y_0)^2}\int_{|z|\leq \d}|z|^2d\Pi(z).
\end{array}$$
Therefore we have  \be\label{estimlocal1}\begin{array}{l}
-I^{1,\delta}(t_0,x_0,\phi(t_0,.,y_0))+I^{1,\delta}(t_0,y_0,-
\phi(t_0,x_0,.))\\\qq\qq\geq - {\sigma(t_0,x_0)^2}\int_{|z|\leq \d}\{\varepsilon^{-1}+C\rho(1+|z|^2)\}|z|^2\Pi(dz)-\varepsilon^{-1}{\sigma(t_0,y_0)^2}\int_{|z|\leq \d}|z|^2d\Pi(z).\end{array} \ee Making now the difference between
(\ref{uniqueinequa1}) and (\ref{uniqueinequa2}) yields
\begin{align*}
&-(p^0_u-p^0_w)-[(\sigma(t_0,x_0)\E(L_1)+
b(t_0,x_0))
q^0_u-(\sigma(t_0,
y_0)\E(L_1)+b(t_0,
y_0))q^0_w]-\frac{1}{2}\varpi^2[\sigma(t_0,
x_0)^2
M^0_u\\
&\qq-\sigma(t_0,y_0)^2M^0_w]-[f_j(t_0,x_0,(u_k(t_0,x_0))^m_{k=1})-f_j(t_0,y_0,(w_k(t_0,y_0))^m_{k=1})]\\
&\qq-I^{1,\delta}(t_0,x_0,\phi(t_0,.,y_0))+I^{1,\delta}(t_0,y_0,-
\phi(t_0,x_0,.))
-I^{2,\delta}(t_0,x_0,q^0_u,u_j)
+I^{2,\delta}(t_0,y_0,
q^0_w,w_j)
\leq 0.
\end{align*}Taking now into account (\ref{estimlocal2}) and (\ref{estimlocal1}) we get
\begin{align*}
&-(p^0_u-p^0_w)-[(\sigma(t_0,
x_0)\E(L_1)+b(t_0,
x_0))q^0_u-
(\sigma(t_0,y_0)
\E(L_1)+b(t_0,y_0))
q^0_w]-\frac{1}{2}\varpi^2[\sigma(t_0,
x_0)^2
M^0_u\\
&\qq-\sigma(t_0,y_0)^2M^0_w]-[f_j(t_0,x_0,(u_k(t_0,x_0))^m_{k=1})-f_j(t_0,y_0,(w_k(t_0,y_0))^m_{k=1})]\\
&\qq- {\sigma(t_0,x_0)^2}\int_{|z|\leq \d}\{\varepsilon^{-1}+C\rho(1+|z|^2)\}|z|^2\Pi(dz)-\varepsilon^{-1}{\sigma(t_0,y_0)^2}\int_{|z|\leq \d}|z|^2d\Pi(z)\\&\qq
-C\varepsilon^{-1}|x_0-y_0|^2-I^{2,\delta}(t_0,x_0,D_x\psi_\rho(t_0,x_0),\psi_\rho)
\leq 0. 
\end{align*}Next by using the properties satisfied by $p^0_u,q^0_u, p^0_w$, $q^0_w$, $M^0_u$ and $M^0_w$ and  sending $\d$ to $0$ to obtain the existence of a constant $C_{\varepsilon,\rho}$ such that for any fixed $\rho$ we have $\limsup\limits_{\varepsilon \rw 0}C_{\varepsilon,\rho}\leq 0$ and
 \be \label{u}
-\{f_j(t_0,x_0,(u_k(t_0,x_0))_{k=1}^m)-f_j(t_0,y_0,(w_k(t_0,y_0))_{k=1}^m)\}\leq
C_{\varepsilon,\rho}+\rho C\int_{\R}\{|z|^2+|z|^4\}\Pi(dz).\ee
Next since $f_j$ is Lipschitz $\wr$ $(y_k)^m_{k=1}$ and by condition (\ref{3.3}) we have
$$-\lambda(u_j(t_0,x_0)-w_j(t_0,y_0))-\sum\limits_{k\in
A_j}\Upsilon_{\varepsilon,\rho}^{j,k}(u_k(t_0,x_0)-w_k(t_0,y_0))\leq
C_{\varepsilon,\rho}+C\rho \int_{\R}\{|z|^2+|z|^4\}\Pi(dz),$$ where $\Upsilon_{\varepsilon,\rho}^{j,k}$ stands for
the increment rate of $f_j$ with respect to $y_k$ $(k\neq j$), which,  by
monotonicity condition
(A4)(I)(v) on $f_j$, is non-negative and bounded by $C_j$. Thus
\begin{align*}
-\lambda(u_j(t_0,x_0)-w_j(t_0,y_0))&\leq\sum\limits_{k\in A_j}\Upsilon_{\varepsilon,\rho}^{j,k}(u_k(t_0,x_0)-w_k(t_0,y_0))^++C_{\varepsilon,\rho}+C\rho \int_{\R}\{|z|^2+|z|^4\}\Pi(dz)\\
&\leq C_j\sum\limits_{k\in
A_j}(u_k(t_0,x_0)-w_k(t_0,y_0))^++C_{\varepsilon,\rho}+C\rho
\int_{\R}\{|z|^2+|z|^4\}\Pi(dz).
\end{align*}
Taking the limit superior in both hand-sides as
$\varepsilon\rightarrow 0$, once again $u_k$ (resp. $w_k$) is $usc$
(resp. $lsc$) and $j\in \underbar{A}$, we get
$$-\lambda(u_j(\bar{t},\bar{x})-w_j(\bar{t},\bar{x}))\leq C_j
\sum\limits_{k\in
A_j}(u_k(\bar{t},\bar{x})-w_k(\bar{t},\bar{x}))^++C\rho \int_{\R}\{|z|^2+|z|^4\}\Pi(dz),$$ finally take $\rho\rw 0$ to obtain,
\begin{align*}
-\lambda(u_j(\bar{t},\bar{x})-w_j(\bar{t},\bar{x}))&\leq C_j \sum\limits_{k\in A_j}(u_k(\bar{t},\bar{x})-w_k(\bar{t},\bar{x}))^+
\leq (m-1)C_j(u_j(\bar{t},\bar{x})-w_j(\bar{t},\bar{x})).
\end{align*}
But this is contradictory since $u_j(\bar{t},\bar{x})-w_j(\bar{t},\bar{x})>0$ and
$-\lambda>(m-1)C_j$. Henceforth for any $j\in A$, $u_j\leq w_j$.
\ms

We now consider the general case. Let $(u_j)_{j\in A}$ (resp. $(w_j)_{j\in A}$) be a subsolution (resp. supersolution) of
(\ref{FIPDE}). Denote $\tilde{u}_j(t,x)= e^{\lambda t}u_j(t,x)$ and
$\tilde{w}_j(t,x)= e^{\lambda t}w_j(t,x)$. Then it is easy to show that
$(\tilde{u}_j)_{j\in A}$ (resp. $(\tilde{w}_j)_{j\in A}$) is a subsolution (resp. supersolution) of the following system of variational inequalities which is similar to (\ref{FIPDE}):\be \label{eqsubvisco}\left\{
    \begin{array}{ll}
      min\{\tilde{u}_j(t,x)-\max\limits_{k\in A_j}(\tilde{u}_k(t,x)-e^{\lambda t}g_{jk}(t,x));\\
      -\partial_t \tilde{u}_j(t,x)-{\cal L} \tilde{u}_j(t,x)+\lambda\tilde{u}_j(t,x)-e^{\lambda t}f_j(t,x,(e^{-\lambda t}\tilde{u}_k)^m_{k=1})\}=0\,;\\
     \tilde{u}_j(T,x)=e^{\lambda T}h_j(x).
    \end{array}
    \right.\ee
Next let us set $$F_j(t,x,\overrightarrow{y}):=-\lambda y_j+e^{\lambda
t}f_j(t,x,(e^{-\lambda t}y_k)^m_{k=1})$$ with $ \lambda$ is chosen such that $\lambda\geq m(1+\max\limits_{k\in
A}C_k)$ where $C_k$ is the Lipschitz constant of $f_k$ w.r.t. to $(y_k)_{k=1}^m$. Then we can mimic the proof of Step 1 to obtain that $\forall j\in A, \tilde{u}_j\leq\tilde{w}_j$
%. the functions $F_k$, $k\in A$, verify condition (\ref{3.3}). It follows, from Step 1, that 
which yields also $u_j\leq w_j$ for any $j\in A$. The proof is now complete.
\end{proof}
As a by-product we have:
\begin{thm}\label{exuncr}Under Assumptions ${\bf (A4)}$, ${\bf (A5)}$,  and (\ref{eqlipb}), (\ref{sigma}) as well, the system of variational inequalities with inter-connected
obstacles $(\ref{FIPDE})$ has a unique continuous viscosity solution
with polynomial growth.\qed
\end{thm}

In the case when the functions $f_j$, $j\in A$, do not depend on $\vec{y}$, by the characterization (\ref{bsderep})-(\ref{lienentreyetv}) (see also Remark \ref{expvalue}), we deduce that the functions $(u_j(t,x))_{j\in A}$ are nothing but $(J^j(t,x))_{j\in A}$. Thus, as a by product of Theorem \ref{exuncr}, we have:
\begin{cor} \lb{corhjb} Assume that:

(i) For any $i=1,\dots,m$, $f_i$ is jointly continuous and of polynomial growth ;

(ii) For any $i,j\in A$, $g_{ij}$ (resp. $h_i$) satisfy (A4)(II) (resp. (A4)(III)). 

\noindent Then the value functions $(J^j(t,x))_{j\in A}$ defined in (\ref{eqcout}) are continuous, belong to $\Pi_g$ and is the unique viscosity solution of the Hamilton-Jacobi-Bellman system associated with the stochastic optimal switching problem which is:
$\forall
j\in A$, \be \label{FIPDEpar}\left\{
    \begin{array}{ll}
      min\{u_j(t,x)-\max\limits_{\ell \in A_j}(u_\ell(t,x)-g_{j\ell}(t,x));\\\qq\qq
      -\partial_t u_j(t,x)-{\cal L} u_j(t,x)-f_j(t,x)\}=0, \,(t,x)\in \esp ;\\
      u_j(T,x)=h_j(x). 
    \end{array}
    \right.\ee\qed
\end{cor}
\subsection{Second existence and uniqueness result}
In this section we consider the issue of existence and uniqueness of a solution for the systems of IPDEs  (\ref{FIPDE}) when the functions $(-f_j)_{j\in A}$ verify (A4)(I). This turns into assuming that $(f_j)_{j\in A}$ verify, instead of (A4)(I)(v), the following:
\ms

\noindent {\bf (A4)($\dag$)}: For any $j\in A$, for any $k\neq j$, the mapping $y_k\rightarrow f_j(t,x,y_1,\cdots,y_{k-1},y_k,y_{k+1},\cdots,y_m)$ is nonincreasing whenever the other components $(t,x,y_1,\cdots,y_{k-1},y_{k+1},\cdots,y_m)$ are fixed.
\ms

The other assumptions on $(-f_j)_{j\in A}$ remain the same.
\begin{thm}\label{resultat1}  Assume that Assumptions (A4)(II)-(III), (A5) are fulfilled and $(-f_j)_{j\in A}$ verify (A4)(I). Then the system of IPDEs (\ref{FIPDE}) has a continuous and of polynomial growth solution which is moreover unique. \end{thm}
\begin{proof}: We first focus on the issue of existence.

For any $j\in A$ and $\lambda\in \R$ let us define $F_j$ by:
$$F_j(t,x,y^1,\cdots,y^m):=e^{\lambda t} f_j(t,x,e^{-\lambda
t}y^1,\cdots,e^{-\lambda t}y^m)-\lambda y^j.$$ Since $f_j$ is
uniformly Lipschitz w.r.t. $(y_k)_{k=1,m}$ then $F_j$ is so and for $\lambda$
large enough, $F_j$ satisfies:
\ms

For any $k=1,m$, the mapping $y_k\rightarrow F_j(t,x,y_1,\cdots,y_{k-1},y_k,y_{k+1},\cdots,y_m)$ is {\bf nonincreasing} whenever the other components $(t,x,y_1,\cdots,y_{k-1},y_{k+1},\cdots,y_m)$ are fixed.
\ms

Let us now
consider the following iterative Picard sequence : $\forall j\in A$,
$Y^{j,0}=0$ and for $n\geq 1$, define:
$$(Y^{1,n},\cdots,Y^{m,n})=\Theta((Y^{1,n-1},\cdots,Y^{m,n-1}))$$
where $\Theta$ is the mapping defined in (\ref{defteta})-(\ref{theta}) where $f_j$ is replaced with $F_j$.  By (\ref{map}), the
sequence $(Y^{j,n})_{j\in A}$ converges in $([H^2]^m,\Vert
.\Vert_{2,\beta})$ to the unique solution $(Y^{j})_{j\in A}$ of the
system of RBSDEs associated with
$$((F_j(s,X^{t,x}_s,y^1,\cdots,y^m))_{j\in A},(e^{\lambda
T}h_j(X^{t,x}_T))_{j\in A},(e^{\lambda
t}g_{jk}(s,X^{t,x}_s))_{j,k\in A}).$$ So using an induction argument
on  $n$  and Theorem \ref{exuncr}, there exist deterministic continuous
functions with polynomial growth $(u^n_j)_{j\in A}$ such that: for any $n\geq 0$ and $j\in A$, 
\be\label{equn}\forall (t,x)\in[0,T]\times \R,\forall s\in[t,T],
Y^{j,n}_s=u^n_j(s,X^{t,x}_s).\ee By (\ref{add5}), take $s=t$ we obtain $$\forall
j, n, q,t\leq T,x\in \R, \,|u_j^n(t,x)-u_j^q(t,x)|=\E[|Y^{j,n}_t-Y^{j,q}_t|^2]\leq
C\Vert(Y^{j,n-1})_{j\in A}-(Y^{j,q-1})_{j\in A}\Vert^2_{2,\beta}.$$
Thus for any $j\in A$, $(u^n_j)_{n\geq 0}$ is of Cauchy type and converges pointwisely to a
deterministic function $u_j$.
But $(Y^j)_{j\in A}=\Theta((Y^j)_{j\in A})$, then once more by (\ref{add5}), we also have: \be\label{2418}\forall s\in[0,T],\,\,\E[|Y^{j}_s-Y^{j,m}_s|^2]\leq C\Vert(Y^{j})_{j\in
A}-(Y^{j,m-1})_{j\in A}\Vert^2_{2,\beta}.\ee By (\ref{equn}) we then obtain \be\label{2419}\forall j\in A,
\forall s\in[t,T], \cP-a.s.,Y^j_s=u_j(s,X^{t,x}_s).\ee
Next as $\Theta$ is a contraction then, by induction on $n$ we have $$\forall
n,q\geq0,~~\Vert(Y^{j,n+q})_{j\in A}-(Y^{j,n})_{j\in
A}\Vert_{2,\beta}\leq\frac{C^n_\Theta}{1-C_\Theta}\Vert(Y^{j,1})_{j\in
A}\Vert_{2,\beta}$$ where $C_\Theta\in ]0,1[$ is the constant of
contraction of $\Theta$. Since the norms $\|.\|$ and
$\|.\|_{2,\beta}$ are equivalent, then there exists a constan $C_1$
such that :$$\forall n,q\geq 0,~~\Vert(Y^{j,n+q})_{j\in
A}-(Y^{j,n})_{j\in A}\Vert\leq C_1C^n_\Theta\Vert(Y^{j,1})_{j\in
A}\Vert.$$
Take now the limit as $q$ goes to $+\infty$ and in the view of
 (\ref{2418}) and (\ref{2419}), if we take $s=t$ we deduce that :
$$\forall (t,x)\in[0,T]\times \R,~~ |u_j(t,x)-u^n_j(t,x)|\leq  C_2\Vert(Y^{j,1})_{j\in A}\Vert.$$
But it is easy to check that $\Vert(Y^{j,1})_{j\in A}\Vert(t,x)$ is of polynomial growth (by (\ref{ini1})-(\ref{ini2}) and since $\E[\sup_{s\leq T}|X^{t,x}_s|^\gamma]$ is of polynomial growth for any $\gamma\geq 0$). Therefore for any $j\in A$, $u_j$ is of polynomial growth, i.e., belongs to $\pl$ since $u^n_j$ is so.
We will now show the continuity of $u_j$. For any $j\in A$, let us set
$$\bar Y^{j,0}_s=C(1+|X^{t,x}_s|^p),\,\,s\leq T,$$ where $C$ and $p$ are related to polynomial growth of $(u_j)_{j\in A}$, i.e., $$\forall j\in A,~|u_j(t,x)|\leq C(1+|x|^p),~~\forall (t,x)\in [0,T]\times \R.$$
Next for any $n\geq 1$ and $j\in A$ let us set
$$(\bar Y^{1,n},\cdots,\bar Y^{m,n})=\Theta((\bar Y^{1,n-1},\cdots,\bar Y^{m,n-1})).$$ As $\Theta$ is a contraction then once more the sequence $((\bar Y^{j,n})_{j\in A})_{n\ge 0}$ converges in
$([H^2]^m,\|.\|_{2,\beta})$ to $(Y^{j,t,x})_{j\in A}$ the unique solution of the system of RBSDEs associated with $$((F_j(s,X^{t,x}_s,y^1,\cdots,y^m))_{j\in A},(e^{\lambda T}h_j(X^{t,x}_T))_{j\in A},(e^{\lambda t}g_{jk}(s,X^{t,x}_s))_{j,k\in A}).$$ By the definition of $\bar Y^{j,0}$, we have
$$\cP-a.s.,~\forall j\in A, s\in [t,T], ~~Y^{j,t,x}_s\leq \bar Y^{j,0}_s $$
and taking into account of {\bf (A4)($\dag$)} we obtain
$$\forall j\in A,\,\forall s\in [t,T],\,
F_j(s,X^{t,x}_s,Y^{1,t,x}_s,\cdots,Y^{m,t,x}_s)\geq F_j(s,X^{t,x}_s,\bar Y_s^{1,0},\cdots,\bar Y_s^{m,0}).
$$
Next by the comparison result of Remark \ref{remarquecomparaison} and since
$(\bar Y^{j,1})_{j\in A}=\Theta((\bar Y^{j,0})_{j\in A})$ , $(Y^{j,t,x})_{j\in A}=\Theta((Y^{j,t,x})_{j\in A})$ we get
$$\forall j\in A, \,s\in [t,T], \,
\bar Y^{j,1}_s\leq Y^{j,t,x}_s.$$
Now by an induction argument we obtain, for any $n\geq 0$ and $j\in A$, \be\label{other1}\forall  s\in [t,T], ~~\bar Y_s^{j,2n+1}\leq Y_s^{j,t,x}\leq \bar Y_s^{j,2n}.\ee In the same way as previously there exist deterministic continuous functions $\bar {u}^n_j$ with polynomial growth such that
$$\forall (t,x)\in [0,T]\times \R,\,\, s\in [t,T],~~\bar Y^{j,n}_s=\bar {u}^n_j(s,X^{t,x}_s).$$ Moreover for any $j\in A$, the sequence $(\bar u_j^n)_{n}$ converges pointwisely to $u$ and by (\ref{other1}) we have
$$\forall j\in A,\, \forall (t,x), \,u_j(t,x)=\lim_n\nearrow \bar {u}^{2n+1}_j(t,x)=\lim_n\searrow\bar {u}^{2n}_j(t,x).$$ Therefore, $u_j$, $j\in A$, is both $lsc$ and $usc$ and then continuous. Finally as
$(Y^{j,t,x})_{j\in A}=\Theta((Y^{j,t,x})_{j\in A})$ and $\forall j\in A,\,\,Y^{j,t,x}_s=u_j(s,X^{t,x}_s)$, $s\in [t,T]$, with $u_j$ a deterministic continuous function with polynomial growth, then $(u_j)_{j\in A}$ is a viscosity solution of the following system of IPDEs:
\be \label{FIPDEmod}\left\{
    \begin{array}{ll}
      min\{u_j(t,x)-\max\limits_{\ell\in A_j}(u_\ell(t,x)-e^{\lambda t}g_{j\ell}(t,x));\\\qq\qq
      -\partial_t u_j(t,x)-{\cal L} u_j(t,x)-F_j(t,x,u_1(t,x),\cdots, u_m(t,x))\}=0, \,(t,x)\in \esp ;\\
      u_j(T,x)=e^{\lambda T}h_j(x), \end{array}
    \right.\ee
thus $(e^{-\lambda t}u_j)_{j\in A}$ is a viscosity solution of the system of IPDEs (4.1) with polynomial growth.
\ms

\nd Let us now deal with the issue of uniqueness. Let $(\bar{u}_j)_{j\in A}$ be another continuous solution of
(\ref{FIPDE}) which belongs to $\pl$ and $(\bar{Y}^j)_{j\in A} \in
[H^2]^m$ such that for any $j\in A, s\in [t,T]$,
$$\bar{Y}^{j,t,x}_s=\bar{u}_j(s,X^{t,x}_s).$$ Define
$(\tilde{Y}^{j,t,x})_{j\in A}$ as follow:
$$(\tilde{Y}^{j,t,x})_{j\in A}=\Theta((\bar{Y}^{j,t,x})_{j\in A}).$$
Then there exist
$(\tilde{u}_j)_{j\in A}$ deterministic continuous functions with polynomial growth $(\tilde{u}_j)_{j\in A}$ such that: $$\forall
j\in A, s\in [t,T],~\tilde{Y}^{j,t,x}_s=\tilde{u}_j(s,X^{t,x}_s).$$
Moreover $(\tilde{u}_j)_{j\in A}$ is the unique viscosity solution
of the following system of IPDEs : $\forall j\in A$

\be \label{unique-f}\left\{
    \begin{array}{ll}
      min\{\tilde{u}_j(t,x)-\max\limits_{k\in
A_j}(\tilde{u}_k(t,x)-g_{jk}(t,x));\\
      \qq\qq-\partial_t u_j(t,x)-{\cal L}\tilde{u}_j(t,x)-f_j(t,x,(\bar{u}_k(t,x))_{k\in A})\}=0\,;\\\\
      \tilde{u}_j(T,x)=h_j(x).
    \end{array}
    \right.\ee
Note that it is $(\bar u_k(t,x))_{k\in A}$ inside the arguments of $f_j$ and not $(\tilde u_k(t,x))_{k\in A}$. As $(\bar{u}_j)_{j\in A}$ is also a solution of (\ref{unique-f}), then by uniqueness of Theorem \ref{exuncr} we obtain $\tilde{u}_j=\bar{u}_j$, for any $j\in A$. Therefore
 $$(\bar{Y}^{j,t,x})_{j\in
A}=\Theta((\bar{Y}^{j,t,x})_{j\in A}).$$ As $(Y^j)_{j\in A}$ is the
unique fixed point of $\Theta$ in $[H^2]^m$, we then have $$\forall j\in
A, s\in[t,T], ~~\bar{Y}^{j,t,x}_s=Y^j_s.$$ It follows that  $\forall
j\in A, \bar{u}_j=u_j$. Finally $(u_j(t,x))_{j\in A}$ is the unique continuous with polynomial
growth functions
viscosity solution of the system of IPDEs (4.1).
\end{proof}
\section{Appendix}

\subsection{Representation of the value function of the stochastic optimal switching problem.}\label{app1}
Let $\Upsilon:=(\theta_n,\a_n)_{n\ge 0}$ be an admissible strategy of switching and let $a=(a_s)_{s\in [0,T]}$ be the process defined by
\be\forall s\leq T,\,\,a_s:=\alpha_0\mathbbm{1}_{\{\theta_0\}}(s)+\sum\limits^\infty_{j=1}\alpha_{j-1}\mathbbm{1}_{]\theta_{j-1}\theta_j]}(s).\ee
Let $t_0\in [0,T]$ and $\Gamma:=((\Gamma^j_s)_{s\in [0,T]})_{j\in
A}\in [H^2]^m$. Let us define the pair of processes $(V^a,N^a):=(V^a_s, N^a_s)_{s\in [0,T]}$ as the solution of the following BSDE:\be\label{switchingBSDE}\left\{
    \begin{array}{ll}
     V^a\in \cS^2,~N^a\in \cH^2(l^2)\\
V^{a}_s=h_{a(T)}(X^{t,x}_T)+\int^T_s \mathbbm{1}_{\{r\geq t_0\}}f_{a(r)}(r,X^{t,x}_r
,\overrightarrow{\Gamma_r},N^a_r)dr-\sum\limits^\infty_{i=1}\int^T_s
N^{a,i}_rdH^{(i)}_r-A^a_T+A^a_s,\, s\in [0,T],
    \end{array}
    \right.\ee where $\overrightarrow{\Gamma_r}=(\Gamma^k_r)_{k\in A}$ and $A^a$ is the cumulative switching cost associated with the strategy $a$ or $\Upsilon$ (see (\ref{def_A}) for its definition). This BSDE is not a standard one, but in assuming that $\E[(A^a_T)^2]<\infty$ and by setting 
$\bar{V}^a=V^a-A^a$, it becomes a standard one and then it has a unique solution. Note that $V^a$ is RCLL since $A^a$ is so. 
\begin{propo}\label{repres} Under Assumption {\bf (A4)}(ii)-(iv), (II) and (III) the solution of  BSDE
(\ref{switchingBSDE}) satisfies: $\forall j\in
A$, \be\label{sw}Y^{\Gamma,j}_{t_0}=\ess_{a\in {\cal
A}^j_{t_0}}( V^a_{t_0}-A^a_{t_0}),\,\,\cP-a.s.\ee where $(Y^{\Gamma,j})_{j\in A}$ is the first component of the
solution of the BSDE (\ref{theta}). Thus the solution of
(\ref{theta}) is unique. Moreover there exists $a^*\in \cA^j_{t_0}$ such that $Y^{\Gamma,j}_{t_0}=V^{a^*}_{t_0}-A^{a^*}_{t_0}.$
\end{propo}
\begin{proof}
Let $(Y^{\Gamma,j}, U^{\Gamma,j},
K^{\Gamma,j})_{j\in A}$ be the solution of the system (\ref{theta}). Let $a\in \cA^j_{t_0}$ and let us define
$$\tilde{K}^a_T=(K^{\Gamma,j}_{\theta_1}-K^{\Gamma,j}_{t_0})+\sum\limits_{n\geq
1}(K^{\Gamma,\alpha_n}_{\theta_{n+1}}-K^{\Gamma,\alpha_n}_{\theta_{n}}) \mbox{ and }$$
 $$\forall i\geq 1\mbox{ and }r\leq T\,\\,U^{a,i}_r=\sum\limits_{n\geq 0}U^{\Gamma,\alpha_n,i}_r\mathbbm{1}_{[\theta_n\leq
    r<\theta_{n+1}[} \mbox{ and }U^a:=(U^{a,i})_{i\geq 1}.$$
Therefore
 \begin{align*}
Y^{\Gamma,j}_{t_0}&=Y^{\Gamma,j}_{\theta_1}+\int^{\theta_1}_{t_0}
f_j(r,X^{t,x}_r,\overrightarrow{\Gamma_r},
U^{\Gamma,j}_r)dr-\sum\limits_{i=1}^\infty\int^{\theta_1}_{t_0}
U^{\Gamma,j,i}_rdH^{(i)}_r+(K^{\Gamma,j}_{\theta_1}-K^{\Gamma,j}_{t_0})\\
&\geq
(Y^{\Gamma,\alpha_1}_{\theta_1}-g_{j,\alpha_1}(\theta_1,X^{t,x}_{\theta_1}))1_{[\theta_1<T]}+
1_{[\theta_1=T]}h_{\a_0}(X_T^{t,x})
+\int^{\theta_1}_{t_0}
f_{a(r)}(r,X^{t,x}_r,\overrightarrow{\Gamma_r},
U^{a}_r)dr\\&\qq-\sum\limits_{i=1}^\infty\int^{\theta_1}_{t_0}
U^{a,i}_rdH^{(i)}_r
+(K^{\Gamma,j}_{\theta_1}-K^{\Gamma,j}_{t_0})\\
 &= Y^{\Gamma,\alpha_1}_{\theta_2}1_{[\theta_1<T]}+\int^{\theta_2}_{t_0}
f_{a(r)}(r,X^{t,x}_r,\overrightarrow{\Gamma_r},
U^{a}_r)dr-\sum\limits_{i=1}^\infty\int^{\theta_2}_{t_0}
U^{a,i}_rdH^{(i)}_r+(K^{\Gamma,j}_{\theta_1}-K^{\Gamma,j}_{t_0})+(K^{\Gamma,\alpha_1}_{\theta_{2}}-K^{\Gamma,\alpha_1}_{\theta_{1}})\\
&\qq\qq-g_{j,\alpha_1}(\theta_1,X^{t,x}_{\theta_1})1_{[\theta_1<T]}+
1_{[\theta_1=T]}h_{\a_0}(X^{t,x}_T).\end{align*}
Repeat now this procedue as many times as necessary and since $a$ is an admissible startegy (i.e. $\cP[\theta_n<T,\forall n\geq 0]=0)$ we obtain:\be\label{ap1} Y^{\Gamma,j}_{t_0}\geq
h_{a(T)}(X^{t,x}_T)+\int^T_{t_0}
f_{a(r)}(r,X^{t,x}_r,\overrightarrow{\Gamma_r},
U^{a}_r)dr-\sum\limits_{i=1}^\infty\int^T_{t_0}
U^{a,i}_rdH^{(i)}_r-A^a_T+\tilde{K}^a_T.\ee
As $\tilde{K}^a_T\geq 0$ and by \eqref{switchingBSDE} we have
\begin{align*}
Y^{\Gamma,j}_{t_0}-V^a_{t_0}+A^a_{t_0}&\geq \int^T_{t_0}
(f_{a(r)}(r,X^{t,x}_r,\overrightarrow{\Gamma_r},
U^{a}_r)-f_{a(r)}(r,X^{t,x}_r
,\overrightarrow{\Gamma_r},N^a_r))dr-\sum\limits_{i=1}^\infty\int^T_{t_0}
(U^{a,i}_r-N^{a,i}_r)dH^{(i)}_r\\
&\geq \int^T_{t_0} \langle V^{U^a,N^a,a}, U^a-N^a\rangle^p_s
ds-\sum\limits_{i=1}^\infty\int^T_{t_0}
(U^{a,i}_r-N^{a,i}_r)dH^{(i)}_r
\end{align*}
Next by Girsanov's Theorem (\cite{philip}, pp.136), under the probability
measure
$d\tilde{\cP}:=\varepsilon
(\sum\limits_{i=1}^\infty
\int^\cdot_{t_0}V^{U^a,N^a,a,i}_rdH^{(i)}_r)_T
d\cP$, $(M_t:=\int^t_{t_0} \langle V^{U^a,N^a,a}, U^a-N^a\rangle^p_s
ds-\sum\limits_{i=1}^\infty\int^t_{t_0}
(U^{a,i}_r-N^{a,i}_r)dH^{(i)}_r)_{t\in [t_0,T]}$ is a martingale, and by taking
conditional expectation of $Y^{\Gamma,j}_{t_0}-V^a_{t_0}+A^a_{t_0}$, we obtain
\begin{align*}
\E_{\tilde{\cP}}[Y^{\Gamma,j}_{t_0}-V^a_{t_0}+A^a_{t_0}|\cF_{t_0}]&\geq \E_{\tilde{\cP}}[ \int^T_{t_0}
\langle V^{U^a,N^a,a,i}, U^a-N^a\rangle^p_s
ds-\sum\limits_{i=1}^\infty\int^T_{t_0}
(U^{a,i}_r-N^{a,i}_r)dH^{(i)}_r|\cF_{t_0}]= 0.
\end{align*}
Thus $Y^{\Gamma,j}_{t_0}\geq V^a_{t_0}-A^a_{t_0},\,\tilde{\cP}-a.s.$ and then, since $\cP$ and $\tilde{\cP}$ are equivalent, for any $a\in {\cal A}^j_{t_0}$,
\be\label{05151}Y^{\Gamma,j}_{t_0}\geq
V^a_{t_0}-A^a_{t_0},\,\cP-a.s..\ee

Next let us consider $a^*$ the strategy defined by $a^*(r)=\alpha^*_0\mathbbm{1}_{\{t_0\}}(r)+\sum\limits^\infty_{k=1}
\alpha^*_{k-1}\mathbbm{1}_{]\theta^*_{k-1}\theta^*_k]}(r)$, $r\leq T$, where $\theta^*_0=t_0$,
 $\alpha^*_0=j$ and for $n\geq 0$, $$\theta^*_{n+1}=\inf\{r\geq
 \theta^*_{n},~Y^{\Gamma,\alpha^*_n}_r=\max\limits_{k\in
 A_{\alpha^*_n}}(Y^{\Gamma,k}_r-g_{\alpha^*_n,k}(r,X^{t,x}_r))\}\wedge
 T,$$ and $$\alpha^*_{n+1}=arg\max\limits_{k\in
 A_{\alpha^*_n}}\{Y^{\Gamma,k}_{\theta^*_{n+1}}-g_{\alpha^*_n,k}(\theta^*_{n+1},X^{t,x}_{\theta^*_{n+1}})\}.$$
Let us show that $a^*\in {\cal A}^j_{t_0}$. We first prove that
$\cP[\theta^*_n< T,~\forall n\geq 0]=0$. We proceed by contradiction
assuming that $\cP[\theta^*_n< T,~\forall n\geq 0]>0$. By definition
of $\theta^*_n$, we then have
$$\cP[Y^{\Gamma,\alpha^*_n}_{\theta^*_{n+1}}=Y^{\Gamma,\alpha^*_{n+1}}_{\theta^*_{n+1}}-g_{\alpha^*_n,\alpha^*_{n+1}}(\theta^*_{n+1},X^{t,x}_{\theta^*_{n+1}}),~\a^*_{n+1}\in
A_{\a^*_{n}},~\forall n\geq 0]>0.$$ But $A$ is finite, then there is a loop $i_0, i_1,\cdots,i_k,i_0$ $(i_1\neq i_0)$ of
elements of $A$ and a subsequence $(n_q(\omega))_{q\geq
0}$ such that:
\be \label{eqlim22}\cP[Y^{\Gamma,i_{l}}_{\theta^*_{n_{q+l}}}=Y^{\Gamma,i_{l+1}}_{\theta^*_{n_{q+l}}}-g_{i_l,i_{l+1}}(\theta^*_{n_{q+l}},X^{t,x}_{\theta^*_{n_{q+l}}}),
~l=1,\cdots, k,~(i_{k+1}=i_0),~\forall q\geq 0]>0.\ee Next let us consider $\theta^*=\lim_{n\rw \infty}\theta^*_n$ and $\Theta=\{\theta^*_n<\theta^*,\forall n\geq 0\}.$ Thanks to the non free loop property \\ $\cP[(\theta^*<T)\cap \Theta^c]=0$ and then $\theta^*$ is an accessible stopping time (see e.g. \cite{delmeyerlr}, pp.214 for more details). But for any $j\in A$, the process $Y^j$ has only inaccessible jump times and $\theta^*$ is accessible, therefore for any $j\in A$, $\Delta Y^j_{\theta^*}=0, \cP-a.s.$. Going back to (\ref{eqlim22}) and take the limit w.r.t. $q$ to obtain:
$$\cP[g_{i_0,i_1}(\theta^*,X^{t,x}_{\theta^*})+\cdots+g_{i_k,i_0}(\theta^*,X^{t,x}_{\theta^*})=0]>0,$$
which contradicts the non free loop property. We then have
$\cP[\theta^*_j<T,~\forall j\geq 0]=0$.\\

Now it remains to prove that $\E[(A^{a^*}_T)^2]<\infty$ and
$a^*$ is optimal in $\cA^j_{t_0}$ for the switching problem (\ref{sw}).
Since $(Y^{\Gamma,j})_{j\in A}$ solves the RBSDE (\ref{theta}) and
by the definition of $a^*$, it yields:
\be\label{513}Y^{\Gamma,j}_{t_0}=Y^{\Gamma,j}_{
\theta^*_1}+\int^{\theta^*_1}_{t_0}
f_{a^*(r)}(r,X^{t,x}_r,\overrightarrow{\Gamma_r},U^{a^*}_r)dr-
\sum\limits_{k=1}^\infty\int^{\theta^*_1}_{t_0}
U^{a^*,k}_rdH^{(k)}_r\ee since $K^{\Gamma,j}_r-K^{\Gamma,j}_{\theta^*_0}=0$ holds
for any $r\in [{t_0}, \theta^*_{1}]$.
But $$
Y^{\Gamma,j}_{
\theta^*_1}=(Y^{\Gamma,\alpha_1^*}_{
\theta^*_1}-g_{j\alpha_1^*}(\theta_1^*,X^{t,x}_{\theta_1^*}))1_{[\theta^*_1<T]}+h_j(X^{t,x}_T)1_{[\theta_1^*=T]}
$$
then
\be\label{513x1}\begin{array}{ll} Y^{\Gamma,j}_{t_0}&=
(Y^{\Gamma,\alpha_1^*}_{
\theta^*_1}-g_{j\alpha_1^*}(\theta_1^*,X^{t,x}_{\theta_1^*}))
1_{[\theta^*_1<T]}+h_j(X^{t,x}_T)
1_{[\theta_1^*=T]}\\{}&\qq\qq+\int^{\theta^*_1}_{t_0}
f_{a^*(r)}(r,X^{t,x}_r,\overrightarrow{\Gamma_r},U^{a^*}_r)dr-
\sum\limits_{k=1}^\infty\int^{\theta^*_1}_{t_0}
U^{a^*,k}_rdH^{(k)}_r\\
{}&=
Y^{\Gamma,\alpha_1^*}_{
\theta^*_1}1_{[\theta^*_1<T]}
+h_j(X^{t,x}_T)
1_{[\theta_1^*=T]}\\{}&\qq\qq+\int^{\theta^*_1}_{t_0}
f_{a^*(r)}(r,X^{t,x}_r,\overrightarrow{\Gamma_r},U^{a^*}_r)dr-
\sum\limits_{k=1}^\infty\int^{\theta^*_1}_{t_0}
U^{a^*,k}_rdH^{(k)}_r-A^{a^*}_{\theta_1^*}.\end{array}
\ee
But we can do the same for the quantity
$
Y^{\Gamma,\alpha_1^*}_{
\theta^*_1}1_{[\theta^*_1<T]}
$ to obtain
$$
Y^{\Gamma,\alpha_1^*}_{
\theta^*_1}1_{[\theta^*_1<T]}=
Y^{\Gamma,\alpha_1^*}_{
\theta^*_2}1_{[\theta^*_2<T]}
+h_{\a_1^*}(X^{t,x}_T)
1_{[\theta_2^*=T]}1_{[\theta^*_1<T]}+
\int^{\theta^*_2}_{\theta_1^*}
f_{a^*(r)}(r,X^{t,x}_r,\overrightarrow{\Gamma_r},U^{a^*}_r)dr-
\sum\limits_{k=1}^\infty\int^{\theta^*_2}_{\theta^*_1}
U^{a^*,k}_rdH^{(k)}_r.
$$
Substitute now this equality in the previous one
and since $\a_2^*$ is the optimal index at
$\theta_2^*$
to obtain:
\be\label{513x1}\begin{array}{ll} Y^{\Gamma,j}_{t_0}&=(Y^{\Gamma,\alpha_2^*}_{
\theta^*_2}-g_{\alpha_1^*\alpha_2^*}(\theta_2^*,X^{t,x}_{\theta_2^*}))1_{[\theta^*_2<T]}
+h_{\a_1^*}(X^{t,x}_T)
1_{[\theta_2^*=T]}1_{[\theta^*_1<T]}
+h_j(X^{t,x}_T)
1_{[\theta_1^*=T]}\\{}&\qq\qq+\int^{\theta^*_2}_{t_0}
f_{a^*(r)}(r,X^{t,x}_r,\overrightarrow{\Gamma_r},U^{a^*}_r)dr-
\sum\limits_{k=1}^\infty\int^{\theta^*_2}_{t_0}
U^{a^*,k}_rdH^{(k)}_r-A^{a^*}_{\theta_1^*}\\{}&=Y^{\Gamma,\alpha_2^*}_{
\theta^*_2}1_{[\theta^*_2<T]}
+h_{\a_1^*}(X^{t,x}_T)
1_{[\theta_2^*=T]}1_{[\theta^*_1<T]}
+h_j(X^{t,x}_T)
1_{[\theta_1^*=T]}\\{}&\qq\qq+\int^{\theta^*_2}_{t_0}
f_{a^*(r)}(r,X^{t,x}_r,\overrightarrow{\Gamma_r},U^{a^*}_r)dr-
\sum\limits_{k=1}^\infty\int^{\theta^*_2}_{t_0}
U^{a^*,k}_rdH^{(k)}_r-A^{a^*}_{\theta_2^*}.\end{array}
\ee
Repeating this procedure as many times as necessary and since $\cP[\theta^*_j< T,~\forall j\geq 0]=0$ to get
\be\label{514}\bal Y^{\Gamma,j}_{t_0}=h_{a^*(T)}(X^{t,x}_T)+\int^T_{t_0}
f_{a^*(r)}(r,X^{t,x}_r,\overrightarrow{\Gamma_r},U^{a^*}_r)dr-\sum\limits_{k=1}^\infty\int^T_{t_0}
U^{a^*,k}_rdH^{(k)}_r-A^{a^*}_T.\ea\ee Now since $\Gamma\in [H^2]^m$,
$U^{a^*}\in \cH^2(\ell^2)$ and $Y^{\Gamma,j}\in \cS^2$, we deduce from
(\ref{514}) that $\E[(A^{a^*}_T)^2]<\infty$. Next by (\ref{switchingBSDE}),$$
\begin{array}{ll}
V^{a^*}_{t_0}-A^{a^*}_{t_0}-Y^{\Gamma,j}_{t_0}&=\int^T_{t_0}
f_{a^*(r)}(r,X^{t,x}_r,\overrightarrow{\Gamma_r},N^{a^*}_r)dr-
\int^T_{t_0}
f_{a^*(r)}(r,X^{t,x}_r,\overrightarrow{\Gamma _r},U^{a^*}_r)dr-
\sum\limits_{k=1}^\infty\int^T_{t_0}
(N^{a^*}_r-U^{a^*,k}_r)dH^{(k)}_r\\\\
&\geq  \int^T_{t_0} \langle V^{N^{a^*},U^{a^*},{a^*}},N^{a^*}
-U^{a^*}\rangle^p_r dr-\sum\limits_{k=1}^\infty\int^T_{t_0}
(N^{a^*,k}_r-U^{a^*,k}_r)dH^{(k)}_r.
\end{array}$$
Once more using Girsanov's Theorem, as in the bulk of the proof of Theorem \ref{maintheoremsystem}, to obtain \\$\E_{\tilde \cP}[V^{a^*}_{t_0}-A^{a^*}_{t_0}-Y^{\Gamma,i}_{t_0}|\cF_{t_0}]\geq 0$ and then $V^{a^*}_{t_0}-A^{a^*}_{t_0}-Y^{\Gamma,i}_{t_0}\geq 0,\cP-a.s.$ Taking now into account (\ref{05151}) leads to the desired result.
\end{proof}
\begin{rem}\label{expvalue} As a by product of (\ref{sw}) we have also:$$\forall j\in A,\,\,\E[Y^{\Gamma,j}_{t_0}]=\sup_{a\in {\cal
A}^j_{t_0}}\E[ V^a_{t_0}-A^a_{t_0}]. \qed
$$
\end{rem}
\subsection{Other equivalent definitions of viscosity solution of IPDEs}\label{app2}
The following definition is an equivalent one for the solution of the IPDE (\ref{IPDE}) in the case when $f$ does not depend on the component $\zeta$. Basically it is an adaptation to our framework, which is of evolution type, of Definitions 1 and 2 given in \cite{barles} in the stationary case.

\begin{defi}\label{defviscoloc22} Assume that the function $f$ of IPDE (\ref{IPDE}) does not depend on $\z$. Let $u:[0,T]\times \R\rightarrow \R$ be a continuous function which belongs to $\Pi_g$. It is said a viscosity subsolution (resp. supersolution) of $(\ref{IPDE})$ if:
\ms

\noindent (i) $u(T,x)\leq  h(x)$ (resp.  $u(T,x)\geq  h(x)$), $\forall x\in \R$ ; 

\noindent (ii) for any $(t,x)\in (0,T)\times \R$, $\d>0$ and a function
$\varphi\in \ccp $ such that
$u(t,x)=\varphi(t,x)$ and $u-\varphi$ has a global maximum
(resp. minimum) at  $(t,x)$ on $[0,T]\times B(x,
C_\sigma\delta)$, we have:
$$\min\Big \{u(t,x)-\Psi(t,x);-\partial_t \varphi(t,x)-
{\cal L}^1  \varphi(t,x)\\
-{\cal I}^{1,\delta}(t,x,\varphi)-
{\cal I}^{2,\delta}(t,x,u,D_x\varphi(t,x))-f(t,x,u(t,x))\Big \}\leq 0\,\, (resp. \,\geq 0).$$
The function $u$ is said to be a viscosity solution of
$(\ref{IPDE})$ if it is both its viscosity subsolution and supersolution.\qed
\end{defi}
\begin{propo} If $f$ does not depend on $\z$ then Definitions (\ref{defviscosimple}) and (\ref{defviscoloc22}) are equivalent.
\end{propo}

\begin{proof}  We prove it only for the subsolution property since the
supersolution one is similar.  Let $u$ be a subsolution of equation (\ref{IPDE}) according to Definition \ref{defviscoloc22}. Then for any $x_0\in \R$ we have $u(T,x_0)\leq h(x_0)$. Next let $(t_0,x_0)\in (0,T)\times \R$ and
$\varphi\in \ccp$ such that $u-\varphi$ has a global maximum at $(t_0,x_0)$ in $\esp$. If we set $\bar \varphi(t,x):=\varphi(t,x)+u(t_0,x_0)-
\varphi(t_0,x_0)$, then $\bar \varphi$ belongs also to $\ccp$ and $u-\bar \varphi$ has a global maximum at $(t_0,x_0)$ in $\esp$ and finally verifies
$\bar \varphi(t_0,x_0):=u(t_0,x_0)$. Applying Definition  \ref{defviscoloc22} with $\bar \varphi$ yields:
\begin{align*}
min&\Big\{u(t_0,x_0)-\Psi(t_0,x_0);-\partial_t\varphi(t_0,x_0)-{\cal L}^1  \varphi(t_0,x_0)-{\cal I}^{1,\delta}(t_0,x_0,\varphi)-
{\cal I}^{2,\delta}(t_0,x_0,u(t_0,x_0),D_x\varphi(t_0,x_0))\\&-
f(t_0,x_0,u(t_0,x_0))\Big\}\leq
0
\end{align*} for any $\d>0$.
Next since $(t_0,x_0)\in (0,T)\times \R$ is a global maximum point of
$u-\varphi$, we then have
$$u(t_0,x_0+\sigma(t_0,x_0)y)-u(t_0,x_0)\leq
\varphi(t_0,x_0+\sigma(t_0,x_0)y)-
\varphi(t_0,x_0)$$ which implies that $I^{2,\delta}(t_0,x_0,D_x\varphi(t_0,x_0),u)\leq
I^{2,\delta}(t_0,x_0,D_x\varphi(t_0,x_0),\varphi)$ and then
\begin{align*}
min&\{u(t_0,x_0)-\Psi(t_0,x_0);-\partial_t\varphi(t_0,x_0)-{\cal L}\varphi(t_0,x_0)
-f(t_0,x_0,u(t_0,x_0))\}\leq
0
\end{align*} which means that $u$ is a subsolution for (\ref{IPDE}) according to Definition  \ref{defviscosimple}.
\ms

We are going now to show that if $u$ is a subsolution of (\ref{IPDE}) according to Definition  \ref{defviscosimple} then it is a subsolution according to Definition  \ref{defviscoloc22}. 
Once more let us consider  a continuous function $u$ which belong to $\Pi_g$ which is a subsolution of (\ref{IPDE}) according to Definition  \ref{defviscosimple}.
Then for all $x_0\in \R$, $u(T,x_0)\leq h(x_0)$.
Next let us fix $\d>0$, $(t_0,x_0)\in (0,T)\times
\R$ and finally let us consider $\varphi\in
\ccp$ such that $u-\varphi$ has a
global maximum at $(t_0,x_0)$ on $[0,T]\times B(x_0,C_\sigma\d)$ and
$u(t_0,x_0)=\varphi(t_0,x_0)$. There exists a function $\tilde
\varphi$ which belongs to $\ccp$
such that $u-\tilde{\varphi}$ attains a global maximum in
$(t_0,x_0)$ on $[0,T]\times \R$ and satisfying $\tilde {\varphi}
(s,y) =\varphi (s,y)$, for any $(s,y)$ such that $|(s,y)-(t_0,x_0)|<
\frac{C_\sigma \delta}{2}.$ Consequently we have also \be
\label{egalites2}\partial_t\tilde{\varphi}(t_0,x_0)=\partial_t\varphi(t_0,x_0),\,\,
D_x\tilde{\varphi}(t_0,x_0)=D_x\varphi(t_0,x_0),\,\,
D^2_{xx}\tilde{\varphi}(t_0,x_0)=D^2_{xx}\varphi(t_0,x_0),\,\,
u(t_0,x_0)=\tilde{\varphi}(t_0,x_0).\ee Next for any $\eps >0$,
there exists $\varphi_\eps$ element of $\cC^{1,2}([0,T]\times \R)$
such that $u\leq \varphi_\eps\leq \tilde \varphi$ and
$\varphi_\eps\rw u$ as $\eps\rw 0$, a.e. (see e.g. Lemma 4.7 in
\cite{jakobsen} or \cite{awatif}). It implies that
$u-\varphi_\eps$ and $\varphi_\eps-\tilde \varphi$ have a global
maximum at $(t_0,x_0)$ on $\esp$. Therefore, on the one hand, we
have
\be \label{egaliteinterm}\partial_t{\varphi_\eps}(t_0,x_0)=\partial_t\tilde\varphi(t_0,x_0),\,\,
D_x{\varphi_\eps}(t_0,x_0)=D_x\tilde\varphi(t_0,x_0),\,\,
D^2_{xx}{\varphi_\eps}(t_0,x_0)\leq D^2_{xx}\tilde\varphi(t_0,x_0)\ee
and, on the other hand, by Definition  \ref{defviscosimple} it holds
\be \label{egalites}\bal
min\Big \{u(t_0,x_0)-\Psi(t_0,x_0)\,;
-\partial_t\varphi_\eps(t_0,x_0)-
{\cal L}^1\varphi_\eps(t_0,x_0)-
{\cal I}(t_0,x_0,\varphi_\eps)
-f(t_0,x_0,u(t_0,x_0))\Big \}\leq
0.\ea
\ee
Recall now the definition of $\cL^1$ in (\ref{operateur}) and taking into account of (\ref{egalites2}) and (\ref{egaliteinterm}) to obtain
\be \label{ineqinterm2}
{\cal L}^1\varphi_\eps(t_0,x_0)\leq {\cal L}^1\varphi(t_0,x_0).
\ee
On the other hand
\be\label{ineqinterm1}
\begin{array}{ll}
{\cal I}(t_0,x_0,\varphi_\eps)&=
{\cal I}^{1,\frac{\d}{2}}(t_0,x_0,\varphi_\eps)+
{\cal I}^{2,\frac{\d}{2}}(t_0,x_0,D_x\varphi_\eps(t_0,x_0),\varphi_\eps)\\&\leq {\cal I}^{1,\frac{\d}{2}}(t_0,x_0,\tilde \varphi)+{\cal I}^{2,\frac{\d}{2}}(t_0,x_0,D_x\varphi(t_0,x_0),\varphi_\eps)\\
&={\cal I}^{1,\frac{\d}{2}}(t_0,x_0, \varphi)+{\cal I}^{2,\frac{\d}{2}}(t_0,x_0,D_x\varphi(t_0,x_0),\varphi_\eps).
\end{array}\ee
Plug now (\ref{ineqinterm2}) and (\ref{ineqinterm1}) in (\ref{egalites}) to obtain
\be \label{egalites2x}\bal
min\Big \{u(t_0,x_0)-\Psi(t_0,x_0)\,;\\\qq-\partial_t\varphi(t_0,x_0)-
{\cal L}^1\varphi(t_0,x_0)-{\cal I}^{1,\frac{\d}{2}}(t_0,x_0, \varphi)
-{\cal I}^{2,\frac{\d}{2}}(t_0,x_0,D_x\varphi(t_0,x_0),\varphi_\eps)
-f(t_0,x_0,u(t_0,x_0))\Big \}\leq
0.\ea
\ee
Take now the limit as $\eps \rw 0$ in (\ref{egalites2x}), using the Lebesgue dominated convergence theorem and by the following inequality (which is valid since $u\leq \varphi$ in $[0,T]\times B(x_0,C_\sigma \d)$ and $u(t_0,x_0)=\varphi(t_0,x_0)$)
$$\begin{array}{l}
\int_{\frac{\d}{2}<|z|\leq \d}(\varphi(t_0,x_0+\sigma(t_0,x_0)z)-\varphi(t_0,x_0)-
D_x\varphi(t_0,x_0)
\sigma(t_0,x_0)z\}d\Pi(z)\geq\\ \qq\qq
\int_{\frac{\d}{2}<|z|\leq \d}(u(t_0,x_0+\sigma(t_0,x_0)z)-u(t_0,x_0)-
D_x\varphi(t_0,x_0)
\sigma(t_0,x_0)z\}d\Pi(z)
\end{array}
$$
we obtain
\begin{align*}
min&\Big\{u(t_0,x_0)-\Psi(t_0,x_0)\,;-\partial_t\varphi(t_0,x_0)-{\cal L}^1  \varphi(t_0,x_0)\\
&-{\cal I}^{1,\delta}(t_0,x_0,\varphi)-
{\cal I}^{2,\delta}(t_0,x_0,D_x\varphi(t_0,x_0),u)-
f(t_0,x_0,u(t_0,x_0))\Big\}\leq
0
\end{align*}
which is the desired result.
\end{proof}

Similarly, there is another equivalent defintion for system of IPDEs (\ref{FIPDE}) which is:

\begin{defi} \label{seconddef} A  function $(u_1,\cdots,u_m):[0,T]\times \R\rightarrow \R^m\in \Pi_g$ such that for any $i\in A$, $u_i$ is usc (resp. lsc), is said to be a viscosity subsolution  (resp. supersolution) of (\ref{FIPDE}) if for any $i\in A$,

(i) $u_i(T,x_0)\leq h_i(x_0)$ (resp. $u_i(T,x)\geq h_i(x)$), $\forall x_0\in \R$;

(ii) for any $(t_0,x_0)\in (0,T)\times \R$, $\d>0$ and a function
$\varphi\in \ccp $ such that
$u_i(t_0,x_0)=\varphi(t_0,x_0)$ and $u_i-\varphi$ has a global maximum
(resp. minimum) at  $(t_0,x_0)$ on $[0,T]\times B(x_0,
C_\sigma\delta)$, we have
\begin{align*}
min&\Big \{u_i(t_0,x_0)-\max\limits_{j\in A_i}(u_j(t_0,x_0)-g_{ij}(t_0,x_0));-\partial_t\varphi(t_0,x_0)-{\cal L}^1
    \varphi(t_0,x_0)-\I^{1,\delta}(t_0,x_0,\varphi)-
    I^{2,\delta}(t_0,x_0,D_x\varphi(t_0,x_0),u_{i})\\
&-f_i(t_0,x_0,u_1(t_0,x_0),\cdots,u_{i-1}(t_0,x_0),u_i(t_0,x_0),\cdots,u_m(t_0,x_0))\Big\}\leq
0~~(resp.\geq 0).
\end{align*}
The functions $(u_i)^m_{i=1}$ is called a viscosity
solution of \eqref{FIPDE} if
$(u_{i*})^m_{i=1}$ and $(u^*_{i})^m_{i=1}$ are respectively viscosity
supersolution and viscosity subsolution of \eqref{FIPDE}.\end{defi}

We then have the following result whose proof is just an adaptaion of the previous one and then is left for the reader.

\begin{propo}
Definitions (\ref{seconddef}) and  (\ref{defvisco}) are equivalent.\qed
\end{propo}

\end{document}